\documentclass{article}

\usepackage{amsfonts, amsmath, amssymb, amsthm, amscd}
\usepackage{ascmac}
\usepackage{verbatim}
\usepackage[dvips]{graphicx}
\usepackage{graphics}
\usepackage[all,2cell]{xypic}
\objectmargin+{1mm}
\labelmargin+{0.8mm}
\SelectTips{cm}{12}

\UseAllTwocells

\theoremstyle{plain}  
\newtheorem{thm}{Theorem}[section]

\newtheorem{cor}[thm]{Corollary}
\newtheorem{lem}[thm]{Lemma}
\newtheorem{prop}[thm]{Proposition}

\theoremstyle{definition}

\newtheorem{df}[thm]{Definition}
\newtheorem{ex}[thm]{Example}

\newtheorem{nt}[thm]{Notations}

\newtheorem{rem}[thm]{Remark}

\newtheorem*{acknow}{Acknowledgements}
\newtheorem{para}[thm]{}
\newtheorem{lemdf}[thm]{Lemma-Definition}

\theoremstyle{remark}

\DeclareMathOperator{\Homo}{H}

\DeclareMathOperator{\bA}{\mathbf{A}}
\DeclareMathOperator{\bB}{\mathbf{B}}
\DeclareMathOperator{\bC}{\mathbf{C}}

\DeclareMathOperator{\bE}{\mathbf{E}}
\DeclareMathOperator{\bF}{\mathbf{F}}
\DeclareMathOperator{\bG}{\mathbf{G}}
\DeclareMathOperator{\bH}{\mathbf{H}}
\DeclareMathOperator{\bI}{\mathbf{I}}

\DeclareMathOperator{\bX}{\mathbf{X}}

\DeclareMathOperator{\bbA}{\mathbb{A}}

\DeclareMathOperator{\bbK}{\mathbb{K}}

\DeclareMathOperator{\bbP}{\mathbb{P}}

\DeclareMathOperator{\bbZ}{\mathbb{Z}}

\DeclareMathOperator{\ff}{\mathfrak{f}}
\DeclareMathOperator{\fF}{\mathfrak{F}}

\DeclareMathOperator{\fS}{\mathfrak{S}}

\DeclareMathOperator{\cA}{\mathcal{A}}
\DeclareMathOperator{\cB}{\mathcal{B}}
\DeclareMathOperator{\cC}{\mathcal{C}}
\DeclareMathOperator{\calD}{\mathcal{D}}
\DeclareMathOperator{\cE}{\mathcal{E}}
\DeclareMathOperator{\cF}{\mathcal{F}}
\DeclareMathOperator{\cG}{\mathcal{G}}
\DeclareMathOperator{\calH}{\mathcal{H}}
\DeclareMathOperator{\cI}{\mathcal{I}}

\DeclareMathOperator{\calL}{\mathcal{L}}
\DeclareMathOperator{\cM}{\mathcal{M}}
\DeclareMathOperator{\cN}{\mathcal{N}}
\DeclareMathOperator{\cO}{\mathcal{O}}
\DeclareMathOperator{\cP}{\mathcal{P}}

\DeclareMathOperator{\calR}{\mathcal{R}}
\DeclareMathOperator{\cS}{\mathcal{S}}
\DeclareMathOperator{\cT}{\mathcal{T}}

\DeclareMathOperator{\cX}{\mathcal{X}}
\DeclareMathOperator{\cY}{\mathcal{Y}}
\DeclareMathOperator{\cZ}{\mathcal{Z}}

\newcommand{\Acy}{\operatorname{Acy}}

\newcommand{\Ar}{\operatorname{Ar}}

\newcommand{\biComp}{\operatorname{\bf BiComp}}
\newcommand{\biCompPair}{\operatorname{\bf BiCompPair}}

\newcommand{\Ch}{\operatorname{\bf Ch}}
\newcommand{\CLat}{\operatorname{\bf CLat}}

\newcommand{\coker}{\operatorname{Coker}}

\newcommand{\comp}{\operatorname{comp}}
\newcommand{\Cone}{\operatorname{Cone}}
\newcommand{\consist}{\operatorname{consist}}

\renewcommand{\coprod}{\sqcup}
\newcommand{\Cub}{\operatorname{\bf Cub}}
\newcommand{\CW}{\operatorname{\bf CW}}
\newcommand{\Cyl}{\operatorname{Cyl}}

\newcommand{\deq}{\operatorname{deq}}

\newcommand{\DM}{\operatorname{\bf DM}}
\newcommand{\dom}{\operatorname{dom}}

\newcommand{\eff}{\operatorname{eff}}
\newcommand{\Ex}{\operatorname{\bf Ex}}
\newcommand{\ExCat}{\operatorname{\bf ExCat}}

\newcommand{\ext}{\operatorname{ext}}

\newcommand{\frob}{\operatorname{frob}}

\newcommand{\HOM}{\mathcal{HOM}}
\newcommand{\Hom}{\operatorname{Hom}}

\newcommand{\id}{\operatorname{id}}
\newcommand{\In}{\operatorname{In}}

\newcommand{\im}{\operatorname{Im}}

\newcommand{\isom}{\operatorname{isom}}
\newcommand{\isoto}{\overset{\scriptstyle{\sim}}{\to}}

\newcommand{\Ker}{\operatorname{Ker}}
\newcommand{\Kos}{\operatorname{\bf Kos}}

\newcommand{\length}{\operatorname{length}}

\newcommand{\linc}{\hookleftarrow}
\newcommand{\linf}{\leftarrowtail}

\newcommand{\Mor}{\operatorname{Mor}}

\newcommand{\NC}{\operatorname{\bf NC}}
\newcommand{\Nis}{\operatorname{Nis}}

\newcommand{\nul}{\operatorname{nul}}

\newcommand{\Ob}{\operatorname{Ob}}
\newcommand{\onto}[1]{\stackrel{#1}{\to}}
\newcommand{\op}{\operatorname{op}}

\newcommand{\Perf}{\operatorname{\bf Perf}}

\newcommand{\qis}{\operatorname{qis}}

\newcommand{\ran}{\operatorname{ran}}

\newcommand{\rdef}{\twoheadrightarrow}

\newcommand{\RelCat}{\operatorname{\bf RelCat}}
\newcommand{\RelEx}{\operatorname{\bf RelEx}}
\newcommand{\res}{\operatorname{res}}
\newcommand{\rinc}{\hookrightarrow}
\newcommand{\rinf}{\rightarrowtail}

\newcommand{\Sh}{\operatorname{\bf Sh}}
\newcommand{\SmCor}{\operatorname{\bf SmCor}}
\newcommand{\SOL}{\operatorname{solid}}

\newcommand{\Spec}{\operatorname{Spec}}
\newcommand{\ssm}{\smallsetminus}
\newcommand{\strict}{\operatorname{strict}}

\newcommand{\thi}{\operatorname{thi}}

\newcommand{\Tot}{\operatorname{Tot}}
\newcommand{\tq}{\operatorname{tq}}

\newcommand{\Tri}{\operatorname{Tri}}
\newcommand{\tri}{\operatorname{tri}}
\newcommand{\TriCat}{\operatorname{\bf TriCat}}

\newcommand{\WalEx}{\operatorname{\bf WalEx}}
\newcommand{\weak}{\operatorname{weak}}
\newcommand{\wide}{\operatorname{wide}}

\def\sn{\smallskip\noindent}
\def\mn{\medskip\noindent}

\newcommand{\cf}{\textrm{cf.}\;}

\title{Non-connective $K$-theory of 
relative exact categories}
\author{Satoshi Mochizuki}
\date{}

\begin{document}

\maketitle

\begin{abstract}
The main objective of this paper is to 
propose a definition of non-connective $K$-theory for 
a wide class of relative exact categories 
which, in general, do not satisfy the factorization axiom 
and confirm that it agrees with the non-connective $K$-theory 
for exact categories and complicial exact categories with weak equivalences. 
The main application is to study 
the topological filtrations of non-connective $K$-theory 
of a noetherian commutative ring with unit 
in terms of Koszul cubes. 
\end{abstract}

\section*{Introduction}

As in \cite{Sch04}, \cite{Sch06} and \cite{Sch11}, 
Schlichting developed the non-connective $K$-theory for 
the wide class 
of Waldhausen exact categories, 
and the non-connective $K$-theory 
for differential graded categories and 
for stable infinity categories 
are characterized by D.-C.~Cisikinsi and G.~Tabuada, 
A.J.~Blumberg and D.~Gepner and G.~Tabuada 
in \cite{CT11} and \cite{BGT10} respectively. 
This generalizes 
the definition of Bass \cite{Bas68}, 
Karoubi \cite{Kar70}, 
Pedersen-Weibel \cite{Ped84}, \cite{PW89}, 
Thomason \cite{TT90}, 
Carter \cite{Car80} and Yao \cite{Yao92}. 
My motivational theme is 
to study the topological filtrations 
of non-connective $K$-theory of 
a noetherian commutative ring with unit 
in terms of Koszul cubes in \cite{Moc11}. 
As precisely mentioned in Remark~\ref{rem:cubeisnotfibrational}, 
the biWaldhausen category of Koszul cubes does 
not satisfy the factorization axiom in \cite{Sch06} 
and the first purpose of this paper is to establish a general theory about 
non-connective $K$-theory for a certain wide class of Waldhausen 
exact categories which, 
in general, 
need not satisfy the factorization axiom.

\sn
Let $\bE=(\cE,w)$ be a {\bf relative exact category}, 
that is, 
a pair of an exact category $\cE$ 
with a specific zero object $0$ 
and a class of morphisms $w$ in $\cE$ which is 
closed under finite compositions. (See Definition~\ref{df:rel exact cat}). 
We let $\cE^w$ denote the full subcategory of 
$\cE$ consisting of those objects $x$ such that 
the canonical morphism from the zero object $0 \to x$ is in $w$. 
We say that $\bE$ is {\bf strict} if 
$\cE^w$ is an exact category such that 
the inclusion functor $\cE^w\rinc \cE$ is exact and reflects exactness. 
(See Ibid). 
For example, $\bE$ is strict if either
$w$ satisfies the extensional axiom or 
$\bE$ is a Waldhausen exact category. 
(See Proposition~\ref{prop:strict exact categories}). 
We denote 
the bounded derived category of 
an exact category $\cF$ 
by $\calD_b(\cF)$. 
We shall define the bounded derived category of 
a strict relative exact category 
$\bE=(\cE,w)$ by the formula 
$\calD_b(\bE):=\coker(\calD_b(\cE^w)\to \calD_b(\cE))$. 
(See Definition~\ref{df:qw}). 
Let $j_{\cE}:\cE \to \Ch_b(\cE)$ denote 
the exact functor which sends an object $x$ to 
a complex $j_{\cE}(x)$ such that $j_{\cE}(x)_k$ 
is $x$ if $k=0$ and is $0$ if $k\neq 0$ and 
we write $P_{\bE}:\Ch_b(\cE) \to \calD_b(\bE)$ 
for the canonical projection functor. 
We say that a morphism $f:x\to y$ in 
$\Ch_b(\cE)$ is a {\bf quasi-weak equivalence} 
if $P_{\cE}(f)$ is an isomorphism in $\calD_b(\bE)$. 
We write $qw$ for the class of quasi-weak equivalences in $\Ch_b(\cE)$ and 
we put $\Ch_b(\bE):=(\Ch_b(\cE),qw)$. 
For example, 
if $w$ is the class of all isomorphisms in $\cE$, 
then $qw$ is just the class of all quasi-isomorphisms in $\Ch_b(\cE)$. 
$\Ch_b(\bE)$ is a complicial exact category 
with weak equivalences 
in the sense of \cite{Sch11}. 
(See Proposition~\ref{prop:compness}). 
We say that a strict relative exact category 
$\bE=(\cE,w)$ is a {\bf consistent} relative exact category if 
$j_{\cE}(w)\subset qw$. (See Lemma-Definition~\ref{lemdf:adm}). 
We will build the universal property of $\Ch_b(\bE)$ 
for any consisitent relative exact category $\bE$ in 
Corollary~\ref{cor:univ} 
which vouches for the pedigree of the operation $\Ch_b(-)$. 
The first main theorem below also warrants 
$\Ch_b(\bE)$ to be a natural object.

\begin{thm}[\bf Derived Gillet-Waldhausen theorem]
\label{thm:intro derived GW}
{\rm (A part of Corollary~\ref{cor:comp of derived cat}).} 
For any consistent relative exact category $\bE=(\cE,w)$, 
the canonical functor $j_{\cE}:\bE \to \Ch_b(\bE)$ induces 
an equivalence of triangulated categories 
$\calD_b(\bE)\isoto\calD_b(\Ch_b(\bE))$. 
\end{thm}

\sn
The theorem 
roughly says that the process of taking $\Ch_b(-)$ 
does not change the mattters up to derived equivalences 
(See also Corollary~\ref{cor:homotopy inv of homotopy theories})
and encourages us to define the non-connective $K$-theory 
of a consistent relative exact category $\bE=(\cE,w)$ by the formula 
$\bbK(\bE)=\bbK^S(\Ch_b(\bE))$. 
(See Definition~\ref{df:non-connective K-theory}). 
Here $\bbK^S$ means the Schlichting non-connective $K$-theory 
in \cite{Sch06} or \cite{Sch11}. 
Theorem~\ref{thm:intro derived GW} and 
Schlichting theory in \cite{Sch11} imply 
that 
the non-connective $K$-theory of 
consistent relative exact categories is a {\it localizing theory} as in 
Corollary~\ref{cor:localization} below. 
We say that a sequence of triangulated categories 
$\cT \onto{i} \cT' \onto{p} \cT''$ is {\bf weakly exact} 
if $pi$ is isomorphic to the zero functor, 
$i$ is fully faithful and the induced functor 
$\cT'/\cT \to \cT''$ is {\bf cofinal}. 
The last condition means that it is fully faithful, 
and every object in $\cT''$ is a direct summand of an object of $\cT'/\cT$.

\begin{cor}
\label{cor:localization}
For a sequence of 
consistent relative exact categories 
 $\bE \to \bF \to \bG$, 
if the induced sequence of triangulated categories 
$\calD_b(\bE) \to \calD_b(\bF) \to \calD_b(\bG)$ 
is weaky exact, then 
the sequence induces a fibration sequence of spectra
$$\bbK(\bE) \to \bbK(\bF) \to \bbK(\bG).$$
In particular if  
the induced morphism 
$\calD_b(\bE) \to \calD_b(\bF)$ 
is an equivalence of triangulated categories up to factor, then 
the induced morphism 
$\bbK(\bE) \to \bbK(\bF)$ is a homotopy equivalence of spectra. 
\end{cor}

\sn
A proof of Corollary~\ref{cor:localization} 
will be given at \ref{proof:locinv}. 
Next, 
if $\bE=(\cE,w)$ is a Waldhausen exact category, 
we can also define $K^W(\bE)=K^W(\cE;w)$ 
the Waldhausen $K$-theory of $\bE$. 
There is a question of what is a sufficient conditions 
that $j_{\cE}:\bE \to \Ch_b(\bE)$ induces an isomorphism 
$K_n^W(\bE)\to K_n^W(\Ch_b(\bE))=\bbK_n^S(\Ch_b(\bE))=\bbK_n(\bE)$ 
for any positive integer $n$. 
We will assay this problem by axiomatic approach in 
section \ref{sec:Localizing theory} and 
carve out the agreement result Theorem~\ref{thm:agreement} below. 
To state the theorem, 
we prepare or recall the notations. 
A strict relative exact category $\bE=(\cE,w)$ is {\bf solid} if 
for any morphism $f:x\to y$ in $\cE$, 
there is a zig-zag sequence of quasi-isomorphisms 
connecting the mapping cone $\Cone f=[x \onto{f}y]$ 
with a bounded complex in $\cE^w$. 
(See Lemma-Definition~\ref{lemdf:solid axiom}). 
We say that a strict relative exact category $\bE$ is {\bf very strict} 
if the inclusion functor $\cE^w \rinc \cE$ induces a fully faithful functor 
on the bounded derived categories $\calD_b(\cE^w)\rinc \calD_b(\cE)$. 
(See Definition~\ref{df:rel exact cat}). 
For example, if either $w$ is the class of all isomorphisms in $\cE$ or $\bE$ 
is a complicial exact category with weak equivalences, 
then $\bE$ is very strict and solid. 
(See Proposition~\ref{prop:bicompair is very strict} and 
Corollary~\ref{prop:bicomp is adm}). 
Let $\bC=(\cC,v)$ be a category with cofibrations and weak equivalences. 
We denote the Waldhausen $K$-theory of $\bC$ by $K^W(\bC)=K^W(\cC;v)$. 
If $v$ is the class of all isomorphisms in $\cC$, 
we shortly write $K^W(\cC)$ for $K^W(\cC;v)$. 
We say that $\bC=(\cC,v)$ satisfies the {\bf $K^W$-fibration axiom} 
if $\cC^v\rinc \cC$ is a subcategory with cofibrations and 
if the inculsion functor $\cC^v\rinc \cC$ and 
the identity functor of $\cC$ induce 
a fibration sequence of spectra 
$K^W(\cC^v) \to  K^W(\cC) \to K^W(\cC;v)$. 
(See Lemma-Definition~\ref{lemdf:excellent}). 
It is well-known that if $v$ satisfies the extensional, 
saturated and factorization axioms, 
then $\bC$ satisfies the $K^W$-fibration axiom. 
(See \cite[Theorem 11]{Sch06}). 
We have the following agreement results. 

\begin{thm}[\bf Agreement]
\label{thm:agreement}
Let $\bE=(\cE,w)$ is a consistent relative exact category. 
Then\\
$\mathrm{(1)}$ 
{\bf (Agreement with Grothendieck groups).} 
$\bbK_0(\bE)$ is isomorphic to the Grothendieck group of 
the idempotent completion of 
the triangulated category $\calD_b(\bE)$.\\
$\mathrm{(2)}$ 
{\bf (Agreement with Schlichting $K$-theory).} 
If either $w$ is the class of all isomorphisms in $\cE$ 
or $\bE$ is a complicial exact category with weak equivalences, 
then the exact functor $j_{\cE}:\bE \to \Ch_b(\bE)$ induces 
a homotopy equivalences of spectra 
$\bbK^S(\bE)\isoto\bbK(\bE)$.\\
$\mathrm{(3)}$ 
{\bf (Agreement with Waldhausen $K$-theory).} 
If $\bE$ is a very strict solid Waldhausen exact category 
which satisfies the $K^W$-fibration axiom, 
then for any positive integer $n$, 
the $n$-th homotopy group of $\bbK(\bE)$ 
is isomorphic to the $n$-th Waldhausen $K$-theory $K^W_n(\bE)$ of $\bE$. 
\end{thm}

\sn
A proof of Theorem~\ref{thm:agreement} will be given in \ref{proof:agreement}. 
The second purpose of this paper is 
to generalize the results in \cite{Moc11} from 
the connective $K$-theory to the non-connective $K$-theory.

\sn
Let us fix a commutative noetherian ring with unit $A$, 
a finite set $S$ and 
a family of elements $\ff_S=\{f_s\}_{s\in S}$ in $A$ which is an 
$A$-regular sequence in any order. 
Let us denote the power set of $S$ by $\cP(S)$, the set of all 
subsets in $S$ with usual inclusion order. 
A {\bf Koszul cube} $x$ associated with a sequence 
$\ff_S=\{f_s\}_{s\in S}$ 
is a contravariant functor from 
$\cP(S)$ to the category of finitely generated 
projective 
$A$-modules 
such that 
for any subset $T$ of $S$ and any element $k$ in $T$, 
$d^k_T:=x(T\ssm\{k\}\rinc T)$ is an injection and 
$f_k^{m_k}\coker d^k_T=0$ for some $m_k$. 
We denote the category of Koszul cubes 
associated with $\ff_S$ by $\Kos_A^{\ff_S}$ 
where morphisms between Koszul cubes are natural transformations. 
(See Definition~\ref{df:Koszul cube df}). 
We let $\Perf_{\Spec A}^{V(\ff_S)}$ denote 
the category of perfect complexes on $\Spec A$ 
whose homological support is in $V(\ff_S)$. 
We denote the class of all quasi-isomorphisms in $\Perf_{\Spec A}^{V(\ff_S)}$ by $\qis$. 
There exists the exact functor 
$\Tot:\Kos_A^{\ff_S} \to \Perf_{\Spec_A}^{V(\ff_S)}$. 
We define the {\bf class of total quasi-isomorphisms} 
by pull-back of $\qis$ in $\Perf_{\Spec A}^{V(\ff_S)}$ 
by $\Tot$ 
and denote it by $\tq$. 
(See Definition~\ref{df:weak equivalences on ltimesfF} 
and \ref{para:genkoscube}). 
Theorem~\ref{thm:intro Kos is device} below 
together with Theorem~\ref{thm:agreement} $\mathrm{(3)}$ 
convince us that 
the non-connective $K$-theory of 
the relative exact category $(\Kos_A^{\ff_S},\tq)$ accords with 
the Waldhausen $K$-theory of it. 

\begin{thm}[\bf A part of Corollary~\ref{cor:Kos is devices}] 
\label{thm:intro Kos is device}
$(\Kos_A^{\ff_S},\tq)$ is a very strict 
solid Waldhausen exact category which satisfies the $K^W$-fibration axiom. 
\end{thm}

\sn
The key ingredient to figure out 
the structure of $(\Kos_A^{\ff_S},\tq)$ 
is the existence of the flag structure 
(See Definition~\ref{df:flag}) 
on $(\Kos_A^{\ff_S},\tq)$ 
and this fact is verified by polishing up results in \cite{Moc11}. 
The final main theorem below 
is the comparison theorem referred in the Abstract. 

\begin{thm}[\bf Weak geometric presentation theorem]
\label{thm:weak geom present}
The exact functor 
$\Tot:(\Kos_A^{\ff_S},\tq) \to (\Perf_{\Spec A}^{V(\ff_S)},\qis)$ 
induces an equivalence of 
the bounded derived categories 
$\calD_b(\Kos_A^{\ff_S},\tq) \to \calD_b(\Perf_{\Spec A}^{V(\ff_S)},\qis)$. 
In particular it also induces a homotopy equivalence of 
spectra $\bbK(\Kos_A^{\ff_S},\tq) \to \bbK(\Perf_{\Spec A}^{V(\ff_S)},\qis)$. 
\end{thm}

\sn
A proof of Theorem~\ref{thm:weak geom present} 
will be given at \ref{para:proof of WGP}. 
In my subsequent paper, 
we will sophisticate Theorem~\ref{thm:weak geom present} 
as a comparison of two different kinds of weights 
immanent in a scheme 
from the viewpoint of non-commutative motive theory. 
It is the reason where the term ``geometric presentation" 
comes from. 
Now we give a guide for the structure of this paper. 
To demonstrate Theorem~\ref{thm:intro derived GW}, 
full techniques evolved from section~\ref{sec:trisubcat} to 
section~\ref{sec:tw} are required. 
Under the influence of supprot varieties theory, 
like as \cite{Bal07} and \cite{BKS07}, 
in this paper we accentuate 
lattice structures of subcategories in appropriate classes of categories. 
We will probe the lattice structures of 
(thick) triangulated subcategories of 
triangulated categories in section~\ref{sec:trisubcat} 
and of null classes of 
bicomplicial categories in section~\ref{sec:widely exact func}. 
These observations leads us to 
characterization of quasi-weak equivalences 
in terms of total weak equivalences in section~\ref{sec:tw}. 
Section~\ref{sec:exact cat with we eq} and 
section~\ref{sec:quasi-split} are 
devoted to glossaries of relative exact categories and 
quasi-split exact sequences respectively. 
The contents in section~\ref{sec:quasi-split} is 
related with \cite{FL12}. 
As pointed out in \cite[A.2.8]{Sch11}, 
quasi-split exact sequences of triangulated categories 
are a variations of Bousfield-Neeman localization theory 
unfolded in \cite[\S 9]{Nee01}. 
In section~\ref{sec:resol thm}, we will sum up 
resolution theorems for relative exact categories 
from \cite{Sch11} and \cite{Moc11}. 
As a sequel to \cite{Moc11}, 
we will inquire into derived categories 
and derived flags 
of multi semi-direct products 
of exact categories in section~\ref{sec:Multi semi-direct} 
and of Koszul cubes in section~\ref{sec:Kos cube}. 
As a matter of my principle, 
the establishment of motive theory 
for a wide class of relative exact categories 
including the relative exact categories of Koszul cubes 
is a gateway to 
decipher motivic interpretations of 
Gersten's conjecture in \cite{Ger73} and 
Weibel's $K$-dimensional conjecture in \cite{Wei80}. 
(They are theorems in many situations). 
Section~\ref{sec:Localizing theory} covers 
general additive and localizing theories over 
relative exact categories 
like as companion theories over DG-categories or infinity-categories 
in \cite{Tab08} or \cite{BGT10}. 
(See also Corollary~\ref{cor:homotopy inv of homotopy theories}). 
These results allow us to 
apply knowledge of Koszul cubes in section~\ref{sec:Kos cube} to 
not only the non-connective $K$-theory 
but also any other localizing theories. 
We conclude with the remark 
that there is a strong resemblance 
between the flags of the relative exact categories 
of perfect complexes on projective spaces in Example~\ref{ex:flag} 
and of Koszul cubes in Corollary~\ref{cor:Kos is devices}. 
In a sequel work, 
I hope to tie up Grayson-Walker weight complexes in 
\cite{Gra95} and \cite{Wal00} 
with Koszul cubes by utilizing the blow up formula in 
\cite{Tho93} or \cite{CHSW08}.

\begin{acknow}
\label{acknow} 
The author wishes to express his deep gratitude to 
Seidai Yasuda for stimulating discussions. 
\end{acknow}

\sn
\textbf{Conventions.}

\mn
$\mathrm{(1)}$ 
{\bf General assumptions}\\
Throughout the paper, 
we use the letters $A$, $\cA$ and $S$ 
to denote a commutative ring with $1$, 
an essentially small abelian category and 
a set respectively.

\sn
$\mathrm{(2)}$ 
{\bf Set theory}\\
$\mathrm{(i)}$ 
We denote the cardinality of a set $S$ by $\# S$.\\
$\mathrm{(ii)}$ 
For any pair of sets $S$ and $T$, 
we put $S\ssm T:=\{x\in S;x\notin T\}$. 
If $S$ and $T$ are disjoint, then we write $S\coprod T$ for 
the union $S\cup T$ of $S$ and $T$.\\
$\mathrm{(iii)}$ 
For any set $S$, 
we write $\cP(S)$ for its power set. 
Namely $\cP(S)$ is the set of all subsets of $S$. 
We consider $\cP(S)$ to be a partially ordered set under inclusion.

\sn
$\mathrm{(3)}$ 
{\bf Partially ordered sets}\\
$\mathrm{(i)}$ 
For any elements $a$ and $b$ in a partially ordered set $P$, 
we write $[a,b]$ for the set of all elements $u$ in $P$ 
such that $a\leq u\leq b$. 
We consider $[a,b]$ as a partially ordered subset of $P$ if $a\leq b$ 
and $[a,b]=\emptyset$ if otherwise. 
We often use this notation to any integers $a$ and $b$.\\
$\mathrm{(ii)}$ 
For any non-negative integer $n$ and any positive integer $m$, 
we denote $[0,n]$ and $[1,m]$ by $[n]$ and $(m]$ respectively.\\
$\mathrm{(iii)}$ 
The {\bf trivial ordering} $\leq$ on a set $S$ is defined by 
$x\leq y$ if and only if $x=y$.\\
$\mathrm{(iv)}$ 
An element $x$ in a partially ordered set $P$ is 
{\bf maximal} (resp. {\bf minimal}) if 
for any element $a$ in $P$, 
the inequality $x\leq a$ (resp. $a\leq x$) implies 
the equality $x=a$. 
An element $x$ in a partially ordered set $P$ is 
{\bf maximum} (resp. {\bf minimum}) if 
the inequality $a\leq x$ (resp. $x\leq a$) 
holds for any element $a$ in $P$.\\
$\mathrm{(v)}$ 
We say that a partially ordered set 
$L$ is a {\bf lattice} 
(resp. {\bf $\vee$-complete}, resp. {\bf $\wedge$-complete}, resp. a {\bf complete lattice}) 
if for any elements $x$ and $y$ (resp. non-empty subset $S$) in $L$, 
there exists 
both $\sup\{x,y\}$ and $\inf\{x,y\}$ 
(resp. $\sup S$, resp. $\inf S$, resp. both $\sup S$ and $\inf S$) in $L$.
We write $x \vee y$ and $x\wedge y$ for 
$\sup\{x,y\}$ and $\inf\{x,y\}$ respectively 
and call them the {\bf join} and {\bf meet} 
of $x$ and $y$ respectively. 
We also denote $\sup S$ (resp. $\inf S$) by 
$\underset{u\in S}{\vee} u$ or $\vee S$ 
(resp. $\underset{u\in S}{\wedge} u$ or $\wedge S$). 
In particular, 
if $L$ is $\vee$-complete (resp. $\wedge$-complete) 
the maximum (resp. minimum) element $1_L=\sup L$ 
(resp. $0_L=\inf L$) exists. 
We write $\CLat$ for the 
full subcategory of complete lattices 
in the category of partially ordered sets.

\sn
$\mathrm{(4)}$ 
{\bf Commutative algebra}\\
$\mathrm{(i)}$ 
We write $A^{\times}$ for 
the group of units in $A$.\\
$\mathrm{(ii)}$ 
If $\ff_S={\{f_s\}}_{s\in S}$ is a subset of $A$, 
we write the same symbole $\ff_S$ for the ideal they generate. 
By convention, we set $\ff_{\emptyset}=(0)$.\\ 
$\mathrm{(iii)}$ 
For any ideal $I$ of $A$ and a letter $p$ which is a natural number or $\infty$, 
we let $\cM_A^I(p)$ denote the category of finitely generated 
$A$-modules $M$ of projective dimension $\leq p$ and 
whose supprt are in $V(I)$.
We also put $\cP_A=\cM_A^{(0)}(0)$ and $\cM_A=\cM_A^{(0)}(\infty)$.

\sn
$\mathrm{(5)}$ 
{\bf Category theory}\\
$\mathrm{(i)}$ 
For any category $\cX$, 
we denote the calss of objects in $\cX$ by $\Ob\cX$ and 
for any objects $x$ and $y$ in $\cX$, 
we write $\Hom_{\cX}(x,y)$ (or shortly $\Hom(x,y)$) 
for the class of all morphisms from $x$ to $y$. 
We say a category $\cX$ is {\bf locally small} (resp. {\bf small}) if 
for any objects $x$ and $y$, $\Hom_{\cX}(x,y)$ forms a set 
(resp. if $\cX$ is locall small and $\Ob\cX$ forms a set).\\
$\mathrm{(ii)}$ 
For any categories $\cX$ and $\cY$, 
we denote the (large) category of functors from $\cX$ to $\cY$ 
by $\cY^{\cX}$. 
Here the morphisms between functors from 
$\cX$ to $\cY$ are just natural transformations. 
In particular, we write $\Ar\cX$ for $\cX^{[1]}$ 
the category of morphisms in $\cX$. 
There are canonical two functors 
$\dom_{\cX}$ and $\ran_{\cX}$ from $\Ar\cX$ 
to $\cX$ which send a morphism $f:x\to y$ in $\cX$ to 
$x$ and $y$ respectively. 
There is also a canonical natural transformation $\epsilon_{\cX}:\dom \to \ran$ 
defined by $\epsilon(f):=f$ for any morphism $f:x\to y$ in $\cX$.\\
$\mathrm{(iii)}$ 
For any partially ordered set $P$, 
we regard it as a category $P$ in the natural way. 
Namely, $P$ is a category whose 
class of objects is $P$ and for any elements $x$ and $y$ in $P$, 
$\Hom_P(x,y)$ is the singleton $\{(x,y)\}$ if $x\leq y$ and is 
the emptyset $\emptyset$ if otherwise. 
In particular, 
we regard any set $S$ as a category by the trivial ordering on $S$.\\
$\mathrm{(iv)}$ 
For any category $\cX$, 
we denote the {\bf dual category} of $\cX$ by $\cX^{\op}$. 
Namely $\Ob\cX^{\op}=\Ob\cX$ and for any objects $x$ and $y$ in $\cX$, 
$\Hom_{\cX^{\op}}(x,y):=\Hom_{\cX}(y,x)$. 
Compositions of morphisms in $\cX^{\op}$ is just same as in $\cX$. 
In particular, for any partially ordered set $P$, 
$P^{\op}$ is said to be the {\bf opposite partially ordered set} of $P$.\\
$\mathrm{(v)}$ 
For any category $\cX$, 
we write 
$i_{\cX}$ or shortly $i$ for the class of all isomorphisms in $\cX$.\\
$\mathrm{(vi)}$ 
Let $\cC$ be a category and $\cS$ a subclass of $\Ob\cC$. 
We write the same letter $\cS$ for the full subcategory 
$\cX$ in $\cC$ such that $\Ob\cX=\cS$ and call it 
the {\bf full subcategory} ({\bf in $\cC$}) {\bf spanned by $\cS$}.\\
$\mathrm{(vii)}$ 
Let $\cC \to \cC'$ be a functor and $\cS$ a full subcategory of $\cC'$. 
We write $f^{-1}\cS$ for the full subcategory of $\cC$ 
spanned by $f^{-1}(\Ob\cS)$ 
and call it 
the {\bf pull-back of $S$} ({\bf by $f$}).\\
$\mathrm{(viii)}$ 
Let $\cC$ be a category and $\calD$ 
a full subcategory of $\cC$ and 
$w$ a class of morphisms in $\cC$. 
We write $\calD_{w,\cC}$ or simply $\calD_w$ for 
the full subcategory of $\cC$ consisting of 
those objects $x$ such that 
there is a zig-zag sequence of morphisms in $w$ connecting 
it to an object in $\calD$. 
In particular, 
we write $\calD_{\isom,\cC}$ or simply $\calD_{\isom}$ 
for $\calD_{i_{\cC},\cC}$ and 
call it the {\bf isomorphisms closure of $\calD$} ({\bf in $\cC$}).
We say that a full subcategory $\calD'$ of $\cC$ 
is {\bf tight} if $\calD'_{\isom,\cC}=\calD'$.\\
$\mathrm{(xi)}$ 
We say that 
a morphism $p:y \to x$ in a category $\cX$ is a {\bf retraction} or 
$x$ is a {\bf retraction} of $y$ 
if there exists a morphism 
$i:x \to y$ such that $pi=\id_x$. 
We say that a full subcategory $\calD$ of $\cX$ is 
{\bf closed under retractions} if for 
any objects $x$ and $y$ in $\cX$, 
if $y$ is in $\calD$ and $x$ is a retraction of $y$, 
then $x$ is also in $\calD$. 
We say that a class of morphisms $w$ of $\cX$ 
satisfies 
the {\bf retraction axiom} if 
$w$ is closed under retractions in the morphism category of $\cX$.\\
$\mathrm{(xii)}$ 
A class of morphisms $w$ in a category $\cC$ 
is a {\bf multiplicative system} (resp. satisfies the {\bf saturated axiom}) 
if $w$ 
is closed under finite compositions 
and closed under isomorphisms 
(resp. for a pair of composable morphisms 
$\bullet \onto{f} \bullet \onto{g} \bullet$ in $\cC$, 
if two of $gf$, $g$ and $f$ are in $w$, 
then the other one is also in $w$).\\
$\mathrm{(xiii)}$ 
We mean that 
a {\bf $2$-category} is a category of enriched 
by the category of (small) categories. 
For any objects $x$ and $y$ in a $2$-category $\cX$, 
we write $\HOM_{\cX}(x,y)$ for 
the Hom category from $x$ to $y$.\\
$\mathrm{(xiv)}$ 
A functor $f:\cX \to \cY$ is {\bf conservative} 
if it reflects isomorphisms, namely for any morphism $a:x\to y$ in $\cX$, 
if $fa$ is an isomorphism in $\cY$, then $a$ is an isomorphism in $\cX$.

\sn
$\mathrm{(6)}$ 
{\bf Relative categories}\\
$\mathrm{(i)}$ 
We use the notations of relative categories theroy from \cite{BK12}. 
A {\bf relative category} $\bC=(\cC,w)$ is a pair of 
a category $\cC$ and a class of morphisms $w$ in $\cC$ such that 
$w$ is closed under finite compositions. 
Namely if $\bullet \onto{f} \bullet \onto{g} \bullet$ are composable morphisms in $w$, 
then $gf$ is also in $w$ and $\id_x$ is in $w$ for any object $x$ in $\cC$. 
A {\bf relative functor} between relative categories 
$f:\bC=(\cC,w) \to \bC'=(\cC',w')$ is a functor $f:\cC \to \cC'$ 
such that $f(w)\subset w'$. 
A {\bf relative natural equivalence} $\theta:f \to f'$ 
between relative functors $f$, $f':\bC=(\cC,w) \to \bC'=(\cC',w')$ 
is a natural transformation $\theta:f \to f'$ such that $\theta(x)$ is in $w'$ 
for any object $x$ in $\cC$. 
We let $\RelCat$ (resp. $\RelCat^{+}$) denote the $2$-category of 
essentially small 
relative categories and relative functors and relative natural equivalencs 
(resp. relative categories and relative functors and natural transformations).\\
$\mathrm{(ii)}$ 
Relative functors $f$, $f':\bC \to \bC'$ are {\bf weakly homotopic} 
if there is a zig-zag sequence of ralative natural transformations connecting $f$ to $f'$. 
A relative functor $f:\bC \to \bC'$ is a {\bf homotopy equivalence} 
if there is a relative functor $g:\bC' \to \bC$ 
such that $gf$ and $fg$ are weakly homotopic to identity functors respectively.\\
$\mathrm{(iii)}$ 
Let $\calR$ be a subcategory of $\RelEx$. 
A functor $F$ from $\calR$ to a category $\cX$ is 
({\bf categorical}) {\bf homotopy invariant} if 
for any relative functors $f$, $g:\bC \to \bC'$ such that 
$f$ and $g$ are weakly homotopic, 
we have the equality $F(f)=F(g)$.\\
$\mathrm{(iv)}$ 
Let $\bC=(\cC,w)$ be a relative category 
with a specfic zero object $0$ in $\cC$. 
We say that an object $x$ in $\cC$ is {\bf $w$-trivial} 
if the canonical morphism $0\to x$ is in $w$. 
We write $\cC^w$ for the full subcategory of $w$-trivial objects in $\cC$.

\sn
$\mathrm{(7)}$ 
{\bf Exact categories, Waldhausen categories and algebraic $K$-theory}\\ 
$\mathrm{(i)}$ 
Basically, 
for exact categories, 
we follow the notations in \cite{Qui73} and 
for connective $K$-theory of categories 
with cofibrations and weak equivalences, 
we follow the notations in \cite{Wal85} 
and for non-connective $K$-theory of 
exact categories, Frobenius pairs and complicial exact categories with weak equivalences, 
we follows the notations in \cite{Sch04}, \cite{Sch06} and \cite{Sch11}. 
As a comprehensive reference, 
please see also \cite{Wei12}.\\ 
$\mathrm{(ii)}$ 
We denote 
a cofibration and 
an admissible monomorphism (resp. an admissible epimoprphism) by the arrow 
``$\rinf$" (resp. ``$\rdef$").  
We sometimes denote 
 a cofibration sequence 
$x\overset{i}{\rinf} y \overset{p}{\rdef} z$ 
by $(i,p)$. 
We write the same letter $0$ for a specific zero object 
of a category with cofibrations. 
We assume that an exact functor between 
categories with cofibrations (or exact categories) 
preserves a specfic zero object. 
We denote the $2$-category of essentially small exact categories 
by $\ExCat$.\\
$\mathrm{(iii)}$ 
We call a category with 
cofibrations and weak equivalences 
a {\bf Waldhausen category} and a 
complicial exact category with weak equivalences 
a {\bf complicial Waldhausen category}.\\ 
$\mathrm{(iv)}$ 
For a Waldhausen category $\bX=(\cX,w)$, 
we write $K^W(\bX)$ or $K^W(\cX;w)$ for 
the Waldhausen $K$-theory of  $\bX$. 
We also write $K(\cX)$ for $K(\cX;i)$. 
For any exact category $\cE$ and any complicial Waldhausen category 
$\bC=(\cC,v)$, we write $\bbK^S(\cE)$ and $\bbK^S(\bC)=\bbK^S(\cC;v)$ 
for Schlichting non-connective $K$-theory of $\cE$ and $\bC$ respectively.\\
$\mathrm{(v)}$ 
We say that 
a functor between exact categories 
(resp. categories with cofibrations) 
$f:\cX \to \cY$ {\bf reflects exactness} 
if for a sequence $x \to y \to z$ in $\cX$ 
such that $fx \to fy \to fz$ is an admissible exact sequence 
(resp. a cofibration sequence) in $\cY$, 
$x\to y \to z$ is an admissible exact sequence 
(resp. a cofibration sequence) in $\cX$.\\ 
$\mathrm{(vi)}$ 
For an exact category $\cE$, 
we say that its full subcategory $\cF$ is 
an {\bf exact subcategory} 
(resp. a {\bf strict exact subcategory}) 
if it is an exact category and 
the inclusion functor is exact (and reflects exactness).\\ 
$\mathrm{(vii)}$ 
Notice that as in \cite[p.321, p.327]{Wal85}, 
the concept of subcategories with cofibrations 
(resp. Waldhausen subcategories) is 
stronger than that of exact subcategories. 
Namely we say that $\cC'$ is a subcategory with cofibrations 
of a category with cofibration $\cC$ if 
a morphism in $\cC'$ is a cofibration in $\cC'$ if 
and only if it is a cofibration in $\cC$ and 
the quotient is in $\cC'$ (up to isomorphism). 
That is, the inclusion functor $\cC' \rinc \cC$ is 
exact and reflects exactness. 
For example, 
let $\cE$ be a non-semisimple exact category. 
Then $\cE$ with semi-simple exact structure 
is not 
a subcategory with cofibrations of $\cE$, 
but a exact subcategory of $\cE$.\\ 
$\mathrm{(viii)}$ 
Let $\cE$ be an exact category and $\cF$ a full subcategory of $\cE$. 
We say that $\cF$ is {\bf closed under kernels} 
({\bf of admissible epimorphisms}) 
if for any admissible exact sequence $x \rinf y \rdef z$ in $\cE$ 
if $y$ is 
isomorphic to object in $\cF$, 
then $x$ is also 
isomorphic to an object in $\cF$. (See \cite[II.7.0]{Wei12}).\\ 
$\mathrm{(ix)}$ 
We say that the class of morphisms $w$ 
in an exact category $\cE$ 
satisfies the {\bf cogluing axiom} 
if $(\cE^{\op},w^{\op})$ 
satisfies the gluing axiom.\\ 
$\mathrm{(x)}$ 
A pair of an exact category $\cE$ and a class of morphisms $w$ in $\cE$ is 
said to be a {\bf Waldhausen exact category} 
if $(\cE,w)$ and $(\cE^{\op},w^{\op})$ are Waldhausen categories. 
We let $\WalEx^{\#}$ denote the $2$-subcategory of 
essentially small 
Waldhausen exact categories and exact functors in $\RelCat^{\#}$ 
for $\#\in\{+,\text{nothing}\}$.\\ 
$\mathrm{(xi)}$ 
For a Waldhausen category $(\cC,w)$, 
we write $w(\cC)$ 
if we wish to emphasis that $w$ 
is the class of weak equivalences in $\cC$. 
We sometimes write $\cC$ for $(\cC,w)$ 
when $w$ is the class of all isomorphisms in $\cC$.\\
$\mathrm{(xii)}$ 
(cf. \cite[p.129 Definition 11]{Sch06}). 
Let $\cC$ be a category with cofibrations and $w$ a class of morphisms in $w$. 
We say that $w$ or $(\cC,w)$ satisfies the {\bf factorization axiom} 
(resp. {\bf extensional axiom}) 
if for any morphism $f:x\to y$ in $\cC$, 
there is a cofibration $i:x \to z$ and a morphism 
$a:z \to y$ in $w$ such that $f=ai$ (resp. $w$ is closed under extensions). 
In this case, moreover if $(\cC,w)$ is a relative category or a Waldhausen category, 
then we say that $(\cC,w)$ is a factorization relative category or 
an extensional Waldhausen category and so on respectively.\\
$\mathrm{(xiii)}$ 
A morphism of Waldhausen categories $f:(\cC,w) \to (\cC',w')$ 
is a {\bf $K^W$-equivalence} if 
it induces a homotopy equivalence on Waldhausen $K$-theory.\\
$\mathrm{(xiv)}$ 
Let $\cZ$ be a category with cofibrations and $\cX$, $\cY$ subcategories with cofibrations. 
We write $E(\cX,\cZ,\cY)$ for the category with cofibrations of cofibration sequences 
$x \rinf z \rdef z$ such that $x$ is in $\cX$ and $y$ is in $\cY$. 
There are three exact functors $s_{E(\cX,\cZ,\cY)}$, $m_{E(\cX,\cZ,\cY)}$ and $q_{E(\cX,\cZ,\cY)}$ 
or shortly $s$, $m$ and $q$ 
from $E(\cX,\cZ,\cY)$ to $\cX$, $\cZ$ and $\cY$ 
which send a cofibration sequence $x\rinf z \rdef y$ to $x$, $z$ and $y$ respectively. 
We let $\In_{\cX}^l$ (resp. $\In_{\cY}^r$) denote an exact functor 
from $\cX$ (resp. $\cY$) to $E(\cX,\cZ,\cY)$ 
which sends an object $a$ in $\cX$ (resp. $\cY$) to 
a cofibration sequence $a \overset{\id_a}{\rinf} a \overset{0}{\rdef} 0$ 
(resp. $0\overset{0}{\rinf} a \overset{\id_a}{\rdef} a$). 
If there is a class of morphisms $w$ in $\cZ$, 
then we define the class of morphisms 
$E(w)$ in $E(\cX,\cZ,\cY)$ by $s^{-1}(w)\cap m^{-1}(w)\cap q^{-1}(w)$. 
We write $E(\cZ)$ for $E(\cZ,\cZ,\cZ)$.\\
$\mathrm{(xv)}$ 
A sequence of exact functors $f \onto{A} g\onto{B} h$ 
between exact categories $\cE \to \cE'$ is 
admissible exact if for any object $x$ in $\cE$, 
the sequence $f(x) \onto{A(x)} g(x) \onto{B(x)} h(x)$ 
is an admissible exact sequence in $\cE'$.

\sn
$\mathrm{(8)}$ 
{\bf Triangulated categories}\\
$\mathrm{(i)}$ 
We follows the notations about triangulated category theory for \cite{Kel96} and \cite{Nee01}. 
We denote a triangulated category by $(\cT,\Sigma,\Delta)$ or simply $\cT$ 
where $\cT$ is an additive category, $\Sigma$ is an additive self category equivalence on $\cT$ 
which is said to be the {\bf suspension functor} and 
$\Delta$ is a class of sequences in $\cT$ of the form 
\begin{equation}
\label{equ:sigma triangle}
x \onto{u} y \onto{v} z \onto{w} \Sigma x
\end{equation}
such that $vu=0$ and $wv=0$ which we denote by $(u,v,w)$ and call it 
a ({\bf $\Sigma$}){\bf -exact triangle} and they satisfies the usual Verdier axioms 
from {\bf (TR $1$)} to {\bf (TR $4$)}. 
In the sequence $\mathrm{(\ref{equ:sigma triangle})}$, 
we sometimes write $\Cone u$ for the object $z$. 
A {\bf triangle functor} between triangulated categories 
from $(\cT,\sigma)$ 
to $(\cT',\sigma')$ 
is a pair $(f,\alpha)$ consisting of an additive functor 
$f:\cT \to \cT'$ and 
a natural equivalence 
$\alpha:f\Sigma \to \Sigma'f$ such that 
they preserves exaxt triangles. 
A {\bf triangle natural transformation} 
$\theta:(f,\alpha) \to (g,\beta)$ between 
triangulated functors $(f,\alpha)$, $(g,\beta):(\cT,\Sigma) \to (\cT',\Sigma')$ 
is a natural transformation 
$\theta:f\to g$ such that $(\Sigma'(\theta)) 
\alpha=\beta (\theta(\Sigma))$. 
Let us denote the $2$-category of essentially small triangulated categories 
by $\TriCat$.\\
$\mathrm{(ii)}$ 
Let $(\cT,\Sigma)$ be a triangulated category. 
We say that 
a full subcategory $\calD$ of $\cT$ is 
a {\bf quasi-triangulated subcategory} ({\bf of $\cT$}) 
if $(\calD,\Sigma)$ is a triangulated category and the inclusion functor 
$(\iota,\id_{\Sigma}):\calD \to \cT$ is a triangulated functor. 
We say that a quasi-triangulated subcategory $\calD$ of $\cT$ is 
a {\bf triangulated subcategory} ({\bf of $\cT$}) if 
it is tight. 
Namely an additive full subcategory $\calD$ of $\cT$ is 
a triangulated subcategory if 
$\Sigma^{\pm}(\calD)\subset \calD$ and if 
for any $\Sigma$-exact triangle $x \to y \to z \to \Sigma x$ 
in $\cT$, 
if $x$ and $y$ are in $\calD$, then $z$ is also in $\calD$. 
Assuming the condition $\Sigma^{\pm}(\calD)\subset \calD$, 
the last condition is equivalent to the condition that 
if two of $x$, $y$ and $z$ are in $\calD$, 
then the other one is also in $\calD$. 
We call this condition the 
{\bf two out of three} ({\bf for $\Sigma$-exact triangles}) {\bf axiom}.\\
$\mathrm{(iii)}$ 
We say that 
a triangulated subcategory $\calD$ of $\cT$ is {\bf thick} 
if $\calD$ is closed under direct summand. 
Namely for any objects $x$ and $y$ in $\cT$, 
if $x\oplus y$ is in $\calD$, 
then both $x$ and $y$ are also in $\calD$.\\
$\mathrm{(iv)}$ 
Let $\cT$ be a triangulated category and $\calD$ a full subcategory of $\cT$. 
We write $\calD_{\tri,\cT}$ or shortly $\calD_{\tri}$ 
(resp. $\calD_{\thi,\cT}$ or shortly $\calD_{\thi}$) 
for the smallest 
triangulated subcategory (resp. thick subcategory) contains $\calD$ in $\cT$ 
and call it  {\bf triangulated} (resp. {\bf thick}) {\bf closure of $\calD$} 
({\bf in $\cT$}).\\
$\mathrm{(v)}$ 
Let $f:\cT \to \cT'$ be a triangulated functor 
between triangulated categories. 
The ({\bf triangulated}) {\bf image of $f$} 
$(\im f)_{\tri}$  or shortly $\im f$ 
is the smallest triangulated category 
which contains the full subcategory 
spanned by $f(\Ob\cT)$. 
The ({\bf triangulated}) {\bf cokernel of $f$} 
$(\coker f)_{\tri}$ or shortly $\coker f$ 
is defined by the Verdier quotient $\coker f:=\cT'/\im f$.\\
$\mathrm{(vi)}$ 
(cf. \cite[Definition 1]{Sch06}). 
A sequence of triangulated categories 
$\cT \onto{i} \cT' \onto{p} \cT''$ is {\bf exact} (resp. {\bf weakly exact}) 
if $pi$ is isomorphic to the zero functor, 
$i$ and $p$ induce equivalence of triangulated categories 
$\cT\isoto \Ker p$ and $\cT'/\Ker p\isoto \cT''$. 
(resp. 
$i$ is fully faithful and the induced functor 
$\cT'/\cT \to \cT''$ is {\bf cofinal}. 
The last condition means that it is fully faithful, 
and every object in $\cT''$ is a direct summand of an object of $\cT'/\cT$.)\\
$\mathrm{(vii)}$ 
(\cf \cite[\S 8]{Kel96}). 
Let $(R,\rho):(\cS,\Sigma) \to (\cT,\Sigma')$ 
and $(L,\lambda):\cT \to \cS$ be two triangle functors such that 
$L$ is left adjoint to $R$. 
Let $A:\id_{\cT} \to RL$ and $B:LR\to \id_{\cS}$ 
be adjunction morphisms. 
For any objects $x$ in $\cT$ and $y$ in $\cS$, 
we write $\mu(x,y)$ for the bijection 
$\Hom_{\cS}(Lx,y) \to \Hom_{\cT}(x,Ry)$, 
$f \mapsto (Rf)(Ax)$. 
We say that $(L,\lambda)$ (resp. $(R,\rho)$) is 
{\bf left} (resp. {\bf right}) 
{\bf triangle adjoint} to $(R,\rho)$ (resp. $(L,\lambda)$) 
if the following equivalent conditions hold.\\
$\mathrm{(a)}$ 
$\lambda=(B\Sigma L)(L \rho^{-1} L)(L\Sigma' A)$.\\
$\mathrm{(b)}$ 
$\rho^{-1}=(R\Sigma B)(R\lambda R)(A \Sigma' R)$.\\
$\mathrm{(c)}$ 
$B\Sigma=(\Sigma B)(\lambda R)(L \rho)$.\\
$\mathrm{(d)}$ 
$\Sigma' A=(\rho L)(R\lambda)(B\Sigma')$.\\
$\mathrm{(e)}$ 
$\mu(\Sigma',\Sigma)\Hom(\lambda,\Sigma)\Sigma=
\Hom(\Sigma',\rho^{-1})\Sigma'\mu$. 
Namely for any objects $x$ in $\cT$ and $y$ in $\cS$, 
the diagram below is commutative.
$$
{\footnotesize{\xymatrix{
\Hom_{\cS}(Lx,y) \ar[r]^{\!\!\!\!\!\Sigma} \ar[d]_{\mu(x,y)} & 
\Hom_{\cS}(\Sigma Lx,\Sigma y) \ar[r]^{\Hom(\lambda x,\Sigma y)} & 
\Hom_{\cS}(L\Sigma' x,\Sigma y) \ar[d]^{\mu(\Sigma' x,\Sigma y)}\\
\Hom_{\cT}(x,Ry) \ar[r]_{\!\!\Sigma'} & 
\Hom_{\cT}(\Sigma x,\Sigma' Ry) \ar[r]_{\Hom(\Sigma' x,R\Sigma y)} & 
\Hom_{\cT}(\Sigma' x,R\Sigma y).
}}}$$
\sn
$\mathrm{(9)}$ 
{\bf Chain complexes}\\
$\mathrm{(i)}$ 
For the notations about chain complexes, 
we basically follow in \cite{Wei94}. 
For a chain complex, we use the homological notation. 
Namely a boundary morphisms are degree $-1$. 
For an additive category $\cB$, 
we denote the category of 
bounded complexes on $\cB$ by $\Ch_b(\cB)$. 
There exists the canonical functor 
$j_{\cB}:\cB \to \Ch_b(\cB)$ where 
$j_{\cB}(x)_k$ is $x$ if $k=0$ and $0$ if $k\ne 0$.\\
$\mathrm{(ii)}$ 
Let $f:x \to y$ be a morphism between complexes in $\Ch_b(\cB)$ 
and $k$ an integer. 
We define the complex $x[k]$ and morphism $f[k]:x[k] \to y[k]$ by 
$x[k]_n=x_{n+k}$ and $d^{x[k]}_n={(-1)}^{k}d_{n+k}^x$ and 
$f[k]_n=f_{n+k}$.\\ 
$\mathrm{(iii)}$ 
For a complex $x$ in an additive category $\cB$, 
we define the brutal truncation $\sigma_{\geq k}x$ 
as follows. 
$(\sigma_{\geq k}x)_i$ 
is equal to $0$ 
if $i<k$ and 
$x_i$ if $i\geq k$ 
and we put $\sigma_{\leq k-1}x:=x/\sigma_{\geq k}x$. 
If $\Ker d_{k+1}^x$ 
(resp. $\im d_{k+1}^x$)
exists, 
then we put 
the complex $\tau_{\leq k}x$ 
(resp. $\tau_{\geq k+1}x$) 
as follows. 
$(\tau_{\leq k}x)_n$ 
(resp. $(\tau_{\geq k+1}x)_n$)
is $x_n$ if $n\leq k$ (resp. $n\geq k+2$), is $\Ker d_{k+1}^x$ 
(resp. $\im d_{k+1}^x$) 
if $n=k+1$ 
and $0$ for otherwise.\\ 
$\mathrm{(iv)}$ 
For any chain morphism $f:x \to y$ of complexes in an additive category $\cB$, 
we define the 
{\bf mapping cone} of $f$, $\Cone f$ by the formula 
$(\Cone f)_n=x_{n-1}\oplus y_n$ and $\displaystyle{d^{\Cone f}_n=
\begin{pmatrix} 
-d^x_{n-1} & 0\\
-f_{n-1} & d_n^y
\end{pmatrix}}$.\\
$\mathrm{(v)}$ 
For a pair of integers $a\leq b$, 
let $\Ch_{[a,b]}(\cB)$ be the full subcategory of 
$\Ch_b(\cB)$ consisting of those complexes $z$ 
such that $z_n=0$ unless $n\in [a,b]$. 
For any $z$ in $\Ch_b(\cB)$, we put 
$$\length z:=\min\{b-a;z\in \Ch_{[a,b]}(\cE) \}.$$
and call it the {\bf length of $z$}.\\
$\mathrm{(vi)}$ 
We denote the homotopy category of 
$\Ch_b(\cB)$ by $\calH_b(\cB)$. 
Namely the class of objects of $\calH_b(\cB)$ 
is same as $\Ch_b(\cB)$, 
morphisms of $\calH_b(\cB)$ is 
chain homotopy classes of morphisms 
in $\Ch_b(\cB)$ and the composition of 
morphisms is induced from 
$\Ch_b(\cB)$. 
Therefore there exists the canonical additive functor 
$P_{\cB}:\Ch_b(\cB) \to \calH_b(\cB)$. 
It is well-known that $\calH_b(\cB)$ naturally becomes a triangulated 
category. (cf. \cite[6.2, \S 7]{Kel96}).\\
$\mathrm{(vii)}$ 
A {\bf strictly acyclic complex} in a Quillen exact category is 
a chain complex which decomposed into admissible short exact sequences (see \cite[\S 11]{Kel96}). 
{\bf Acyclic complexes} are chain complexes 
which are chain homotopy equivalent to strictly acyclic complexes. 
A morphism between chain complexes is a {\bf quasi-isomorphism} 
if its mapping cone is an acyclic complex. 
We denote the category of bounded acyclic complexes on a Quillen exact category 
$\cE$ by $\Acy_b(\cE)$. 
It is well-known that 
$P_{\cE}(\Acy_b\cE)$ 
the full subcategory of $\calH_b(\cE)$ spanned 
by $P_{\cE}(\Ob\Acy_b(\cE))$ where $P_{\cE}$ is the canonical 
projection functor $P_{\cE}:\Ch_b(\cE) \to \calH_b(\cE)$, is 
a triangulated subcategory. (cf. \cite[11.3]{Kel96}). 
We define the {\bf bounded derived category of $\cE$}, 
$\calD_b(\cE)$ 
by $\calD_b(\cE):=\calH_b(\cB)/P_{\cE}(\Acy_b(\cE))$.

\section{Triangulated subcategories}
\label{sec:trisubcat}

In this section, 
we will study the lattices structure and the 
functorial behaviour of the set of 
(thick) triangulated subcategories 
of triangulated categories. 
The key proposition \ref{prop:homthmfortricat} says that 
there exists the canonical lattice isomorphism between 
a lattice of (thick) triangulated subcategories of 
a triangulated category and 
that of the quotient triangulated category. 
Utilizing this proposition and the notion of 
{\it factorizable pairs}, we explicitly describe the 
join of a factorizable pair of thick subcategories 
in \ref{cor:veethi} 
and the join of general triangulated subcategories 
in \ref{prop:general vee}. 
Recall the conventions of 
partially ordered sets and 
triangulated categories 
from 
Conventions $\mathrm{(3)}$ and $\mathrm{(8)}$. 
We start by introducing 
useful lemmata~\ref{lem:complete criterion} and \ref{lem:ret in tri cat} 
to study the lattices of triangulated subcategories. 

\begin{lem}
\label{lem:complete criterion}
Let $L$ be a $\vee$-complete {\rm (}resp. $\wedge$-complete{\rm )} lattice 
with the minimum {\rm (}resp. maximum{\rm )} element. 
Then $L$ is a complete lattice.
\end{lem}

\begin{proof}[\bf Proof] 
We only prove for a $\vee$-complete lattice $L$. 
To prove for a $\wedge$-complete lattice is similar. 
Let $S\subset L$ be a non-empty subset. 
Then we put 
$$l(S):=\{u\in L;\text{$u\geq x$ for any $x\in S$}\}.$$
Since the minimum element is in $l(S)$, 
$l(S)$ is not the empty set. 
Therefore there exsits the element $\inf S=\sup l(S)$ in $L$. 
\end{proof}

\begin{lem}
\label{lem:ret in tri cat}
For objects $x$ and $y$ in a triangulated category $\cT$, 
if $x$ is a retraction of $y$, 
then $x$ is a direct summand of $y$. 
In particular, 
a triangulated subcategory $\calD$ of $\cT$ is thick 
if and only if it closed under retractions.
\end{lem}

\begin{proof}[\bf Proof]
Let $i:x \to y$ and $p:y\to x$ be morphisms such that 
$pi=\id_x$. 
Let us consider the diagram of $\Sigma$-exact triangles below 
$${\footnotesize{\xymatrix{
x \ar[r]^i \ar[d]_{\id_x} & y \ar[d]_{\begin{pmatrix} p\\ q\end{pmatrix}} 
\ar[r]^q & z \ar@{-->}[d]^h \ar[r]^r & \Sigma x \ar[d]^{\id_x}\\ 
x \ar[r]_{\begin{pmatrix}\id_x\\ 0 \end{pmatrix}} & 
x\oplus z \ar[r]_{\begin{pmatrix}0 & \id_z \end{pmatrix}} 
& z \ar[r]_0 & \Sigma x  .
}}}$$
By the axiom of triangulated categories, 
there exists a morphism $h$ 
which makes diagram above commutative. 
Therefore it turns out that $r=0$. 
Hence we can take $h=\id_z$. 
Now by the five lemma of $\Sigma$-exact triangles, 
we learn that 
$\displaystyle{\begin{pmatrix}p\\ q\end{pmatrix}:
y \to x\oplus z}$ 
is an isomorphism. 
\end{proof}

\begin{df}
\label{nt:Tridf}
Let $\cT$ be an essentially small triangulated category and 
$\calL$ a triangulated subcategory of $\cT$. 
We write $V_{\emptyset}(\calL)$ or $V(\calL)$ (resp. $V_{\thi}(\calL)$) 
for the set of all (thick) triangulated subcategories 
which contains $\calL$. 
In particular, we put 
$\Tri_{\#}(\cT):=V_{\#}(\{0\})$ 
where $\#=\emptyset$ or $\thi$. 
Obviously $V_{\#}(\calL)$ is 
a partially ordered set with the usual inclusion order. 
\end{df}

\begin{lemdf}
\label{lemdf:lattice functor Tri}
Let $\cT$ and $\cT'$ be 
essentially small triangulated categories and 
$f:\cT \to \cT'$ a triangulated functor. 
Then\\
$\mathrm{(1)}$ 
For any triangulated subcategory $\cN$ of $\cT$, 
$V_{\#}(\cN)$ ($\#=\emptyset$ or $\thi$) is a complete lattice. 
We denote the join and the meet in $V_{\#}(\cN)$ by 
$\vee_{\#}$ and $\wedge_{\#}$ respectively.\\
$\mathrm{(2)}$ 
For any $\cN\in\Tri_{\#}(\cT')$, 
$f^{-1}\cN$ is in $\Tri_{\#}(\cT)$.\\
$\mathrm{(3)}$ 
The association 
$\Tri_{\#}(f):\Tri_{\#}(\cT') \to \Tri_{\#}(\cT),\ \cN \mapsto f^{-1}\cN$ 
is an order preserving map.\\
In particular, there exist the functors 
$$\Tri_{\#}:\TriCat^{\op} \to \CLat\ \ (\text{$\#=\emptyset$ or $\Tri$}).$$
\end{lemdf}

\begin{proof}[\bf Proof] 
Obviously $V_{\#}(\cN)$ has the maximum element $\cT$ and 
is $\wedge_{\#}$-complete. 
(The operation $\wedge_{\#}$ is just the intersection). 
Therefore by \ref{lem:complete criterion}, it is complete. 
Assertions $\mathrm{(2)}$ and $\mathrm{(3)}$ are straightforwards. 
\end{proof}

\sn
In this section, from now on, 
let $\cT$ be a triangulated category. 

\begin{prop}
\label{prop:homthmfortricat}
For a thick subcategory $\calL$ 
of $\cT$, 
the canonical projection 
$\pi:\cT \to \cT/\calL$ induces the isomorphism 
$$\Tri_{\#}(\pi):\Tri_{\#}(\cT/\calL)\isoto V_{\#}(\calL),\ \ 
\cN\mapsto \pi^{-1}\cN$$
where $\#=\emptyset$ or $\thi$.
\end{prop}

\begin{proof}[\bf Proof]
The proof is carried out in several steps.

\begin{para}[Step 1]
\label{para:homthmfortricatproof1}
We will construct the 
inverse map of $\Tri_{\#}(\pi)$. 
To do so, 
we need to prove that 
for any (thick) triangulated subcategory 
$\cN$ in $V_{\#}(\calL)$, 
$\pi(\cN)$ the isomorphisms closure 
of the full subcategory spanned by 
$\pi(\Ob\cN)$ is a (thick) triangulated subcategory. 
Then the association $\cN \mapsto \pi(\cN)$ is the 
desired inverse map.
\end{para}

\begin{para}[Step 2]
\label{para:homthmfortricatproof2}
First we will prove that for any object $x$ 
in $\pi(\cN)$, $\Sigma x$ and $\Sigma^{-1}x$ are 
in $\pi(\cN)$ again. 
By definition, 
there exists an object $y$ in $\cN$, 
$x$ is isomorphic to $\pi y$ in $\cT/\calL$. 
Therefore we have the isomorphisms 
$$\Sigma^{\pm 1}x \isoto \Sigma^{\pm 1}\pi y\isoto \pi\Sigma^{\pm 1}y$$
with $\Sigma^{\pm 1}y \in \cN$. 
Hence we get the assertion. 
\end{para}

\begin{para}[Step 3]
\label{para:homthmfortricatproof3}
Next we prove that 
for any $\Sigma$-exact triangle 
$$x \onto{u} y \to z \to \Sigma x$$ 
in $\cT/\calL$, if $x$ and $y$ are in $\pi(\cN)$, 
then $z$ is also in $\pi(\cN)$. 
We represent the corresponding 
objects of $x$ and $y$ in $\cT$ by 
the same letters $x$ and $y$ respectively 
and shall assume that 
both $x$ and $y$ are in $\cN$. 
The morphism $u$ is represented by a morphisms 
$x \overset{s_u}{\leftarrow} w \overset{f_u}{\to} y$ in 
$\cT$ with $\Cone s_u \in \calL$. 
Since $\calL$ is contained in $\cN$, 
by the two out of three axiom, 
we learn that $w$ is in $\cN$ and therefore 
by the two out of three axiom again, 
$\Cone f_u$ is also in $\cN$. 
Since $z$ is isomorphic to $\pi(\Cone f_u)$, 
we notice that $z$ is in $\pi(\cN)$. 
\end{para}

\begin{para}[Step 4]
\label{para:homthmfortricatproof4}
We will prove that for any object $x$ in $\cT$, 
$\pi(x)$ is in $\pi(\cN)$ 
if and only if $x$ is in $\cN$. 
Let us assume that $\pi(x)$ is in $\pi(\cN)$. 
Then there exists an object $y$ in $\cN$ such that $\pi(y)$ 
is isomorphic to $\pi(x)$. 
The isomorphism between $\pi(x)$ and $\pi(y)$ in $\cT/\calL$ 
is represented by morphisms $x \overset{s}{\leftarrow} z \onto{t} y$ 
with $\Cone s$, $\Cone t\in \calL$ ($\subset \cN$). 
Therefore by the two out of three axiom, 
$z$ and $x$ are in $\cN$. 
The converse assertion is trivial.
\end{para}

\begin{para}[Step 5]
\label{para:homthmfortricatproof5}
We prove that if $\cN$ is thick, 
then $\pi(\cN)$ is also thick. 
We just check that for any objects $x$ and $y$ in $\cT$ 
if $\pi(x)\oplus \pi(y)$ is in $\pi(\cN)$, 
then $\pi(x)$ and $\pi(y)$ are in $\pi(\cN)$. 
By \ref{para:homthmfortricatproof4}, $x\oplus y$ is in $\cN$. 
Therefore we get the assertion.
\end{para}

\begin{para}[Step 6]
\label{para:homthmfortricatproof6}
We prove $\Tri_{\#}(\pi)$ and $\pi$ are 
inverse functors in each other. 
First we prove that 
for any $\cN$ in $\Tri_{\#}(\cT)$, 
we have the equality $\Tri_{\#}(\pi)(\pi(\cN))=\cN$. 
For any object $x$ in $\Tri_{\#}(\pi)(\pi(\cN))$, 
$\pi(x)$ is in $\pi(\cN)$ and 
therefore by \ref{para:homthmfortricatproof4}, 
$x$ is in $\cN$. Hence we get the result.
Finally we will prove that for any $\cN'$ in $\Tri_{\#}(\cT/\calL)$, 
we have the equality $\pi(\Tri_{\#}(\pi)(\cN'))=\cN'$. 
For any object $y$ in $\cT$ such that 
$\pi(y)$ is in $\pi(\Tri_{\#}(\pi)(\cN'))$, 
by \ref{para:homthmfortricatproof4}, $y$ is in $\Tri_{\#}(\pi)(\cN')$ 
and this is equivalent to the condition that $\pi(y)$ is 
in $\cN'$. 
\end{para}
\end{proof}

\begin{df}[\bf Factorizable pair] 
\label{df;factorizable pair} 
Let $\cM$ and $\cN$ be triangulated subcategories of $\cT$. 
We say that the ordered pair $(\cN,\cM)$ 
is {\bf factorizable} ({\bf in $\cT$}) if 
any morphism from an object $x$ in $\cN$ to 
an object $y$ in $\cM$ admits a factorization 
$x \to z \to y$ with $z\in\cN\cap \cM$. 
\end{df}

\begin{lem}[\bf Quotient of factorizable pairs]
\label{lem:quot of factorizable pair} 
Let $(\cN,\cM)$ be a factorizable pair in $\cT$, 
then $({(\cN/\cN\cap\cM)}_{\isom},{(\cM/\cN\cap\cM)}_{\isom})$ is 
factorizable in $\cT/\cN\cap\cM$. 
Namely any morphism in $\cT/\cN\cap\cM$ from an object in 
$\cN/\cN\cap\cM$ to an object in $\cM/\cN\cap\cM$ 
is the zero morphism.
\end{lem}

\begin{proof}[\bf Proof] 
First notice that by \ref{prop:homthmfortricat}, 
${(\cN/\cN\cap\cM)}_{\isom}$ and ${(\cM/\cN\cap\cM)}_{\isom}$ are 
triangulated subcategories of $\cT/\cN\cap\cM$. 
For any morphism from an object $x$ in $\cN/\cN\cap\cM$ 
to an object $y$ in $\cM/\cN\cap\cM$ is 
written by $x \overset{s}{\leftarrow} z \overset{f}{\to} y$ 
with $\Cone s\in\cN\cap\cM$. 
Then we notice that 
$z$ is in $\cM$ and therefore 
$z \onto{f} y$ admits a factorization 
$z \to w \to y$ with $w\in\cN\cap\cM$. 
\end{proof}

\begin{prop}
\label{prop:perpendicular subcategories}
Let $\cM$ and $\cN$ be triangulated subcategories 
of $\cT$. 
Let us assume that any morphism from an object in $\cM$ to 
an object in $\cN$ is the zero morphism. 
Then\\
$\mathrm{(1)}$ 
The composition 
$$\cM \rinc \cT \onto{Q} \cT/\cN$$
is fully faithful where $Q$ is the canonical quotient functor.\\
$\mathrm{(2)}$ 
$\cM_{\isom,\cT/\cN}$ is a triangulated subcategory in $\cT/\cN$.\\
$\mathrm{(3)}$ 
Moreover if $\cM$ is a thick subcategory in $\cT$, 
then $\cM_{\isom,\cT/\cN}$ is a thick subcategory 
in $\cT/\cN$. 
\end{prop}

\begin{proof}[\bf Proof] 
$\mathrm{(1)}$ 
Let $x$ and $y$ be objects in $\cM$. 
We prove that any morphism in $\Hom_{\cT/\cN}(x,y)$ is 
represented by a morphism in 
$\Hom_{\cT}(x,y)$ uniquely. 
Let $x \overset{s}{\leftarrow} z \overset{f}{\to} y$ 
be a morphism from $x$ to $y$ in $\cT/\cN$. 
Then since $\Cone s$ is in $\cN$, 
by the assumption, $s$ is an isomorphism. 
This means that 
we have the equality 
$(x \overset{s}{\leftarrow} z \overset{f}{\to} y)=
(x \overset{\id_x}{\leftarrow} z \overset{fs^{-1}}{\to} y)$. 

\sn
$\mathrm{(2)}$ 
First we prove that 
$\cM_{\isom,\cT/\cN}$ is closed under $\Sigma^{\pm}$. 
For any object $x$ in $\cM_{\isom,\cT/\cN}$, 
there are an object $y$ in $\cM$ and 
an isomorphism $x\isoto Q(y)$ in $\cT/\cN$. 
Then we have the isomorphisms 
$$\Sigma^{\pm}x\isoto \Sigma^{\pm}Qy \isoto Q\Sigma^{\pm} y.$$
Since $\Sigma^{\pm}y$ is in $\cM$, 
we learn that $\Sigma^{\pm}x$ is in $\cM_{\isom,\cT/\cN}$. 
Second we prove that 
for any $\Sigma$-exact triangle 
$x \to y \to z \to \Sigma x$ in $\cT/\cN$, 
if $x$ and $y$ are in $\cM_{\isom,\cT/\cN}$, 
then $z$ is also in $\cM_{\isom,\cT/\cN}$. 
Then there are objects $x'$ and $y'$ in $\cM$ such that 
$Q(x')\isoto x$ and $Q(y')\isoto y$ in $\cT/\cN$. 
Then by $\mathrm{(1)}$, 
there exists the morphism 
$u:x' \overset{s_u}{\leftarrow} z' \onto{f_u} y'$ 
in $\cT/\cN$ 
which makes the diagram below commutative. 
$${\footnotesize{\xymatrix{  
x \ar[r] \ar[d]_{\wr} & y \ar[r] \ar[d]_{\wr} & 
z \ar[r] \ar[d]^{\wr} & \Sigma x \ar[d]^{\wr}\\
Q(x') \ar[r]_u & Q(y') \ar[r] & Q(\Cone f_u) \ar[r] & Q(x')  .
}}}$$
Therefore we have an isomorphism $Q(\Cone f_u)$ 
and $z$, 
and it turns out that $z$ is in $\cM_{\isom,\cT/\cN}$. 

\sn
$\mathrm{(3)}$ 
Let $x$ and $y$ be objects in $\cT$ and let us 
assume that $Q(x\oplus y)$ is in $\cM_{\isom,\cT/\cN}$. 
We will prove that $Q(x)$ is in $\cM_{\isom,\cT/\cN}$. 
Then there are an object $u$ in $\cM$ and 
an isomorphism $Q(x\oplus y)\overset{a}{\leftarrow} 
Q(w) \onto{b} Q(u)$ in $\cT/\cN$. 
Since $\Cone b$ is in $\cN$, 
$b$ is an isomorphism. 
Therefore there exists the morphism 
$c=ab^{-1}:u \to x\oplus y$ with $\Cone c\in\cN$. 
Hence 
by the $3\times 3$-lemma for $\Sigma$-exact triangles, 
there exists a morphism 
$z:=\Sigma^{-1}\Cone(\Pr_y c) \onto{d} x$ which makes 
the diagram below commutative 
$${\footnotesize{\xymatrix{
z \ar[r] \ar[d]_d & u \ar[r] \ar[d]_c & 
y \ar[r] \ar[d]_{\id_y} & \Cone (\Pr_y c) \ar[d]\\
x \ar[r]_{i_x} \ar[d] & x\oplus y \ar[r]_{\Pr_y} \ar[d] & 
y \ar[r]_0 \ar[d] & \Sigma x \ar[d]\\
\Cone d \ar[r] & \Cone c \ar[r] & 
0 \ar[r] & \Sigma \Cone d
}}}$$
where the morphisms $i_x:x \to x\oplus y$ and 
$\Pr_y:x\oplus y \to y$ 
are the canonical morphisms 
and all horizontal lines are $\Sigma$-exact triangles. 
Then it turns out that 
$\Cone d$ is isomorphic to $\Cone c$ and 
hence it is in $\cN$. 
Hence we have an isomorphism 
$Q(z)\isoto Q(x)$ in $\cT/\cN$. 
Since $\cM$ is thick, $y$ is in $\cM$ and therefore 
$z$ is also in $\cM$ and it turns out that $Q(x)$ is in 
$\cM_{\isom,\cT/\cN}$. 
\end{proof}

\begin{prop}
\label{prop:fully faithful quotient}
Let $\cN$ and $\cM$ be triangulated subcategories of $\cT$. 
Let us assume that $(\cM,\cN)$ or $(\cN,\cM)$ is 
factorizable in $\cT$. 
Then\\
$\mathrm{(1)}$ 
The canonical functors 
$$\cM/\cN\cap\cM \to \cT/\cN\text{  and  } \cN/\cN\cap\cM \to \cT/\cM$$ 
are fully faithful.\\
$\mathrm{(2)}$ 
Moreover if $\cN$ and $\cM$ are thick, then 
${(\cM/\cN\cap\cM)}_{\isom,\cT/\cN}$ 
{\rm (}resp. ${(\cN/\cN\cap\cM)}_{\isom,\cT/\cM}${\rm )} 
is a thick subcategory of $\cT/\cN$ {\rm (}resp. $\cT/\cM${\rm )}.
\end{prop}

\begin{proof}[\bf Proof] 
First notice that if $(\cN,\cM)$ 
is factorizable in $\cT$, 
then $(\cM^{\op},\cN^{\op})$ is 
factorizable in $\cT^{\op}$. 
Therefore we shall just check that if 
$(\cN,\cM)$ is factorizable, then 
the canonical functor $\cM/\cN\cap\cM \to \cT/\cN$ is 
fully faithful and moreover if 
both $\cN$ and $\cM$ are thick, 
then ${(\cM/\cN\cap\cM)}_{\isom,\cT/\cN}$ 
is a thick subcategory of $\cT/\cN$. 
The first assertion is mentioned in \cite[10.3]{Kel96}. 
To prove the second assertion, 
let us consider the factorization 
$${(\cM/\cN\cap\cM)}_{\isom,\cT/\cN} \to {(\cT/\cN\cap\cM)}_{\isom,\cT/\cN} 
\onto{Q} \cT/\cN$$
where $Q$ is the canonical quotient functor. 
${(\cM/\cN\cap\cM)}_{\isom}$ and ${(\cN/\cN\cap\cM)}_{\isom}$ 
are thick subcategories of $\cT/\cN\cap\cM$ 
by \ref{prop:homthmfortricat} 
and the ordered pair $({(\cN/\cN\cap\cM)}_{\isom},{(\cM/\cN\cap\cM)}_{\isom})$ 
is factorizable 
by \ref{lem:quot of factorizable pair}. 
Therefore the assertion follows 
from \ref{prop:perpendicular subcategories}. 
\end{proof}

\begin{cor}[\bf $\vee_{\thi}$ of $\Tri_{\thi}(\cT)$] 
\label{cor:veethi}
Let $\cN$ and $\cM$ be thick subcategories of $\cT$. 
If $(\cN,\cM)$ or $(\cM,\cN)$ is factorizable in $\cT$, 
then we have the equalities 
$$\cM\vee_{\thi} \cN=Q^{-1}({(\cM/\cN\cap\cM)}_{\isom,\cT/\cN})
={Q'}^{-1}({(\cN/\cN\cap\cM)}_{\isom,\cT/\cM})$$ 
where $Q:\cT \to \cT/\cN$ and $Q':\cT \to \cT/\cM$ are 
the canonical quotient functors. 
\end{cor}

\begin{proof}[\bf Proof] 
By the symmetry of $\cM$ and $\cN$, 
we shall just check the first equality. 
For simplicity, we put $\cO:=Q^{-1}({(\cM/\cN\cap\cM)}_{\isom,\cT/\cN})$. 
Obviously $\cO$ contains $\cM$ and $\cN$. 
For any thick subcategory $\calL$ which contains $\cM$ and $\cN$, 
we have the fully faithful embeddings 
$$\cM/\cN\cap\cM \rinc \calL/\cN \rinc \cT/\cN$$ 
by \ref{prop:fully faithful quotient}. 
Therefore by \ref{prop:homthmfortricat}, 
$\calL=Q^{-1}({(\calL/\cN)}_{\isom,\cT/\cN})$ contains $\cO$. 
Hence we have $\cM\vee_{\thi} \cN=\cO$. 
\end{proof}

\sn
For general $\cN$ and $\cM$ in 
the proposition above, we need more subtle 
argument.

\begin{prop}[\bf $\vee_{\#}$ of $\Tri_{\#}(\cT)$ II] 
\label{prop:general vee}
Let $\cN$ and $\cM$ be are triangulated subcategories of $\cT$ 
and $Q:\cT \to \cT/\cN$ the canonical quotient functor. 
Then\\
$\mathrm{(1)}$ 
By abuse of the notations, we write 
$\cN\vee_{\thi} \cM$ for the smallest thick subcategory of $\cT$ 
which contains both $\cN$ and $\cM$. 
Then we have 
$$\cN\vee_{\thi}\cM={Q^{-1}(\im(\cM \to \cT \onto{Q} \cT/\cN))}_{\thi}.$$
In particular, we have the formula
$$\coker(\cN/\cN\cap \cM \to \cT/\cM)\isoto \cT/(\cN\vee_{\thi}\cM).$$

\sn
$\mathrm{(2)}$ 
If $\cN$ is thick, then we have 
$$\cN\vee\cM= Q^{-1}(\im(\cM \to \cT \onto{Q} \cT/\cN)).$$
\end{prop}

\begin{proof}[\bf Proof] 
For simplicity, 
we put $\#=\emptyset$ if $\cN$ is thick and $\#=\thi$ if $\cN$ is not thick. 
Moreover let us put 
$\cY=\im(\cM \to \cT\onto{Q} \cT/\cN)_{\#}$ and $\cX=Q^{-1}\cY$. 
Since $\cY$ is in $\Tri_{\#}(\cN_{\#})$, 
by \ref{prop:homthmfortricat}, 
$\cX=Q^{-1}\cY$ is in $V_{\#}(\cN_{\#})$. 
Since $Q(x)$ is in $\cY$ for any object $x$ in $\cM$, 
$\cX$ is also in $V_{\#}(\cM)$. 
Next let us take $\calL\in V_{\#}(\cN)\cap V_{\#}(\cM)$ and put 
$\cZ=\im(\calL \to \cT\onto{Q} \cT/\cN)_{\#}$. 
Then obviously we have $\cY\subset \cZ$, 
therefore $\cX=Q^{-1}\cY \subset Q^{-1}\cZ=\calL$. 
Here the last equality follows from \ref{prop:homthmfortricat}. 
\end{proof}

\section{Relative exact categories}
\label{sec:exact cat with we eq}

\sn
In this section, we study relative exact categories. 
In particular, we define the bounded derived categories 
of strict relative exact categories. 
We start by preparing a useful lemma to treat relative exact categories.

\begin{lem}
\label{lem:isom closed}
Let $\cC$ be a category with cofibrations and $w$ a class of morphisms in $\cC$ such that 
$w$ contains all identity morphisms in $\cC$. Then\\
$\mathrm{(1)}$ 
If $w$ satisfies either the extensional or the gluing axiom, 
then $w$ contains all isomorphisms in $\cC$.\\
$\mathrm{(2)}$ 
If $w$ satisfies the extensional axiom, 
then $w$ is closed under co-base change along cofibrations. 
\end{lem}

\begin{proof}[\bf Proof]
Assertion $\mathrm{(2)}$ is porven in \cite[A.21]{Uni}. 
We will give a proof of $\mathrm{(1)}$. 
We denote the zero object in $\cC$ by $0$. 
Let $f:x\to y$ be an isomorphism in $\cC$. 
Then $f=\id_x\coprod_{\id_{0}}\id_0$ by the push-out diagram below and 
$f$ is also extension of $\id_x$ and $\id_0$ by the commutative diagram 
of admissible short exact sequences below. 
$${\footnotesize{
\xymatrix{
& 0 \ar[rr] \ar[ld] \ar@{-}[d] & & 0 \ar[dd]^{\id_0} \ar[ld]\\
x \ar[rr]^{\ \ \ \ \ \ \ \ \ \ \ \ \id_x} \ar[dd]_{\id_x} & \ar[d] & x \ar[dd]^{\!\!\! f}\\
 & 0 \ar[ld] 
\ar@{-}[r] & 
\ar[r] & 
0 \ar[ld]\\
x \ar[rr]_{f} & & y
} \ \ 
\xymatrix{
x \ar@{>->}[r]^{\id_x} \ar[d]_{\id_x} & x \ar[d]_f \ar@{->>}[r] & 0 \ar[d]^{\id_0}\\
x \ar@{>->}[r]_{f} & y \ar@{->>}[r] & 0.
}}}$$
Therefore if $w$ satisfies either the extensional or the 
gluing axioms, then $f$ is in $w$. 
\end{proof}

\sn
Recall the notations of relative categories, 
exact categories and chain complexes 
from Conventions $\mathrm{(6)}$, $\mathrm{(7)}$ and $\mathrm{(9)}$.

\begin{df}[\bf Relative exact categories]
\label{df:rel exact cat}
$\mathrm{(1)}$ 
A relative category $\bE=(\cE,w)$ is 
a {\bf relative exact category} 
if the underyling category $\cE$ is a Quillen exact category 
with a specific zero object $0$.\\
$\mathrm{(2)}$ 
A {\bf Relative exact functor} 
$f:\bE=(\cE,w) \to \bE'=(\cE',w')$ between 
relative exact categories is a relative functor such that 
$f:\cE \to \cE'$ is an exact functor and $f(0)=0$.\\
$\mathrm{(3)}$ 
We denote the $2$-subcategory of relative exact categories and relative 
exact functors in $\RelCat^{\#}$ by $\RelEx^{\#}$ 
for $\#\in\{+,\text{nothing}\}$.\\
$\mathrm{(4)}$ 
A relative exact functor $f:\bE \to \bE'$ 
is an {\bf exact homotopy equivalence} 
if there is a relative exact functor $g:\bE' \to \bE$ 
such that both $fg$ and $gf$ are weakly homotopic to identity functors respectively.\\
$\mathrm{(5)}$ 
A relative exact category $\bE=(\cE,w)$ is {\bf strict} if 
$\cE^w$ is a strict exact subcategory of $\cE$. 
A strict relative exact category $\bE=(\cE,w)$ is {\bf very strict} if 
the inclusion functor $\cE^w\rinc \cE$ induces a fully faithful functor 
on the bounded derived categories $\calD_b(\cE^w)\rinc\calD_b(\cE)$. 
We deno the full $2$-subcategory of strict 
(resp. very strict) 
relative exact categories in $\RelEx^{\#}$ by 
$\RelEx_{\strict}^{\#}$ (resp. $\RelEx_{v.s}^{\#}$) 
for $\#\in\{+,\text{nothing}\}$. 
\end{df}

\begin{rem}
\label{rem:E^w}
$\mathrm{(1)}$ 
Let $\bE=(\cE,w)$ be a relative exact category. 
If $w$ contains all isomorphisms between zero objects in $\cE$, 
then $\cE^w$ does not depend upon a choice 
of a specfic zero object in $\cE$.\\
$\mathrm{(2)}$ 
A relative exact functor $f:\bE=(\cE,w) \to \bE'=(\cE',w')$ 
induces a functor $f:\cE^w \to {\cE'}^w$. 
If both $\bE$ and $\bE'$ are strict, 
then the induced functor $f:\cE^w \to {\cE'}^w$ is exact. 
\end{rem}

\sn
Recall the definition of Waldhausen exact categories from 
Conventions $\mathrm{(7)}$ $\mathrm{(x)}$.

\begin{prop}[\bf Examples of strict relative exact categories]
\label{prop:strict exact categories}
For any relative exact category $\bE=(\cE,w)$ if 
either $w$ satisfies the extensional axiom or 
$\bE$ is a Waldhausen exact category, 
then $\bE$ is a strict relative exact category.
\end{prop}

\begin{proof}[\bf Proof]
In the commutative diagram of admissible short exact sequences in $\cE$ below, 
we have the equalities $a=\id_0\times_c b$ and $c=\id_0\coprod_a b$. 
$${\footnotesize{
\xymatrix{
x \ar@{>->}[r] \ar[d]_a & y \ar@{->>}[r] \ar[d]_b & z \ar[d]^c\\
0 \ar@{>->}[r] & 0 \ar@{->>}[r] & 0.
}}}$$
Therefore 
if $w$ satisfies the extensional (resp. gluing, cogluing) axiom 
and if $x$ and $z$ (resp. $x$ and $y$, $y$ and $z$) are in $\cE^w$, 
then $y$ (resp. $z$, $x$) is also in $\cE^w$. 
Hence if $w$ satisfies the extensional axiom (resp. $\bE$ is a Waldhausen exact category), 
then $\cE^w$ is closed under extensions 
(resp. taking admissible sub- and quotient objects 
and finite direct sums) in $\cE$ 
and $\cE^w$ is a strict exact subcategory of $\cE$ 
(by \cite[5.3]{Moc11}). 
\end{proof}

\sn
In the rest of this section, 
let $\bE=(\cE,w)$ and $\bF=(\cF,v)$ 
be relative exact categories.

\begin{df}[\bf Level and quasi-weak equivalences]
\label{df:qw} 
$\mathrm{(1)}$ 
A morphism $f:x \to y$ 
in $\Ch_b(\cE)$ is a 
{\bf level-weak equivalence} 
if $f_n:x_n \to y_n$ is in $w$ for any integer $n$. 
We denote the class of level-weak equivalences by $lw$.\\
$\mathrm{(2)}$ 
Assume that $\bE=(\cE,w)$ is 
a strict relative exact category. 
We define the {\bf bounded derived category} of $\bE$ 
by $\calD_b(\bE)=\calD_b(\cE,w):=
\coker(\calD_b(\cE^w) \to \calD_b(\cE))$.\\
$\mathrm{(3)}$ 
In the situation $\mathrm{(2)}$, 
a morphism in $\Ch_b(\cE)$ is said to be a 
{\bf quasi-weak equivalence} if 
its image in $\calD_b(\bE)$ is an isomorphism. 
We denote the class of quasi-weak equivalences in $\Ch_b(\cE)$ 
by $qw$ and put $\Ch_b(\bE):=(\Ch_b(\cE),qw)$. 
This association defines 
the $2$-functor $\Ch_b:\RelEx^{+}_{\strict} \to \RelEx^{+}$.\\
$\mathrm{(4)}$ 
Let $f:\bE \to \bF$ be a morphism of 
strict relative exact categories. 
We say that $f$ is a {\bf derived equivalence} 
(resp. {\bf weakly derived equivalence}, {\bf derived fully faithful}) 
if induced functor on bounded derived categories 
is an equivalence of triangulated category 
(resp. equivalence of up to factor, fully faithful).\\
$\mathrm{(5)}$ 
Let $\calR$ be a subcategory of 
$\RelEx_{\strict}$. 
We denote the class of derived equivalences in 
$\calR$ by $\deq_{\calR}$ or shortly $\deq$. 
We call the (large) relative category 
$(\calR,\deq_{\calR})$ the 
{\bf homotopy theory of relative exact categories in $\calR$}. 
\end{df}

\begin{rem}
\label{rem:acyqw} 
Recall the functor $P_{\cE}:\Ch_b(\cE) \to \calH_b(\cE)$ 
is the canonical projection functor. 
We have the formula 
$$\calD_b(\cE,w)=\calH_b(\cE)/(\calH_b(\cE^w)\vee_{\thi}P_{\cE}(\Acy_b(\cE))) $$
by \ref{prop:general vee}. 
We put 
$$\Acy^{qw}_b(\cE):=P_{\cE}^{-1}
(\calH_b(\cE^w)\vee_{\thi}P_{\cE}(\Acy_b(\cE))),$$ 
$$\Acy^{lw}_b(\cE):=\Ch_b(\cE^w).$$
\end{rem}

\sn
Recall the definition of (weakly) exact sequences of triangulated categories 
from Conventions $\mathrm{(8)}$ $\mathrm{(vi)}$. 

\begin{df}[\bf Exact sequences of relative categories]
\label{df:exact seq in RelEx}
$\mathrm{(1)}$ 
A sequence $\bE \onto{u} \bF \onto{v} \bG$ 
of strict relative exact categories is 
{\bf exact}
(resp. {\bf weakly exact}) 
if induced sequence of triangulated categories 
$\calD_b(\bE) \onto{\calD_b(u)} \calD_b(\bF) \onto{\calD_b(v)} \calD_b(\bG)$ is 
exact (resp. weakly exact). 
We sometimes denote the sequence above by $(u,v)$. 
For a full subcategory $\calR$ of $\RelEx_{\strict}^{\#}$, 
we let $E(\calR)$ (resp. $E_{\weak}(\calR)$) denote the category of 
exact sequences (resp. weakly exact sequences) in $\calR$ 
as the full subcategory of $\calR^{[2]}$. 
We define three functors $s^{\calR}$, $m^{\calR}$ and $q^{\calR}$ from 
$E_{\#}(\calR)$ to $\calR$ which sends weakly exact sequence 
$\bE \to \bF \to \bG$ to $\bE$, $\bF$ and $\bG$ respectively.\\
$\mathrm{(2)}$ 
Let $\calR$ and $\calR'$ be full subcategories of $\RelEx_{\strict}$. 
A functor $F:\calR \to \calR'$ is {\bf exact} 
(resp. {\bf weakly exact}) 
if it sends an exact (resp. a weakly exact) sequence in $\calR$ 
to an exact (resp. a weakly exact) sequence in $\calR'$. 
\end{df}

\begin{prop}[\bf Example of weakly exact sequences]
\label{prop:ex of weakly exact sequences}
Let $\cE$ be an exact category and 
$v$ and $w$ classes of morphisms in $\cE$ 
such that $v\subset w$ and both $(\cE,v)$ and $(\cE,w)$ 
are very strict relative exact categories. 
Then the inclusion functor $\cE^w\rinc\cE$ and 
the identity functor of $\cE$ induce a weaky exact sequence 
$$(\cE^w,v)\to (\cE,v) \to (\cE,w).$$ 
\end{prop}

\begin{proof}[\bf Proof]
We apply \ref{prop:fully faithful quotient} $\mathrm{(1)}$ 
to the fully faithful functors 
$\calD_b(\cE^v) \to \calD_b(\cE^w) \to \calD_b(\cE)$. 
We learn that the induced functor 
$\displaystyle{\calD_b(\cE^w,v)=\frac{\calD_b(\cE)}{\calD_b(\cE^v)} \to 
\calD_b(\cE,v)=\frac{\calD_b(\cE)}{\calD_b(\cE^v)}}$ 
is fully faithful. 
On the other hand, 
we have an equivalence of triangulated categorie 
$\displaystyle{\calD_b(\cE,w)=\frac{\calD_b(\cE)}{\calD_b(\cE^w)}\isoto
\frac{\frac{\calD_b(\cE)}{\calD_b(\cE^v)}}{\frac{\calD_b(\cE^w)}{\calD_b(\cE^v)}}}$ 
which makes the diagram below commutative
$${\footnotesize{\xymatrix{ 
 & \calD_b(\cE)/\calD_b(\cE^v) \ar[rd] \ar[ld]& \\
\calD_b(\cE)/\calD_b(\cE^w) \ar[rr]^{\sim} & & 
\frac{\calD_b(\cE)/\calD_b(\cE^v)}
{\calD_b(\cE^w)/\calD_b(\cE^v)}
}}}$$
where all functors above are induced from the identity functor of $\cE$. 
Hence we obtain the result. 
\end{proof}

\begin{lemdf}[\bf Consistent axiom]
\label{lemdf:adm}
For any strict relative exact category 
$\bE=(\cE,w)$, the following two conditions are equivalent.\\
$\mathrm{(1)}$ 
$lw \subset qw$.\\
$\mathrm{(2)}$ 
The canonical functor 
$j_{\cE}:\cE \to \Ch_b(\cE)$ 
is a relative exact functor 
$\bE \to \Ch_b(\bE)$.\\

\sn
In this case, 
we say that $w$ (or $\bE=(\cE,w)$) satisfies the {\bf consistent axiom} 
or $w$ (or $\bE$) is consistent. 
We write $\RelEx_{\consist}^{\#}$ (resp. $\WalEx_{\consist}^{\#}$) 
for the full $2$-subcategory of 
consistent relative exact categories 
(resp. consistent Waldhausen exact categories) 
in $\RelEx^{\#}$ (resp. $\WalEx^{\#}$) 
for $\#\in\{+,\text{nothing}\}$.
\end{lemdf}

\begin{proof}[\bf Proof] 
We can easily check that 
condition $\mathrm{(1)}$ implies 
condition $\mathrm{(2)}$. 
Assuming condition $\mathrm{(2)}$, 
we prove condition $\mathrm{(1)}$. 
Let $f:x\to y$ be a morphism in $\Ch_b(\cE)$. 
First let us notice that 
obviously the class $qw$ is closed under the degree shift. 
Namely if $f$ is in $qw$, then $f[n]$ is also in $qw$ for any integer $n$. 
In \ref{prop:compness}, 
we will prove that $qw$ satisfies the extensional axiom. 
Suppose that $f$ is in $lw$, 
then by the following short exact sequences 
$${\footnotesize{\xymatrix{
\sigma_{\leq n}x \ar[r] \ar[d]_{\sigma_{\leq n}f} & x \ar[r] \ar[d]_f & 
\sigma_{\geq n+1}x \ar[d]^{\sigma_{\geq n+1}f}\\
\sigma_{\leq n}y \ar[r] & y \ar[r] & \sigma_{\geq n+1}y  ,
}}}$$
induction on the length of chain complexes 
and by the extensional axiom for $qw$, 
we get the desired result. 
\end{proof}

\begin{ex}
\label{ex:adm exact cat with weak equi} 
$\mathrm{(1)}$ 
A Quillen exact category $\cE$ with the class of all isomorphisms 
$(\cE,i_{\cE})$ is a consistent Waldhausen exact category.\\
$\mathrm{(2)}$ 
We will prove in \ref{prop:compness} and \ref{prop:bicomp is adm} 
that for any strict relative exact category 
$\bE$, $\Ch_b(\bE)$ is 
a complicial Waldhausen category. 
In particular, 
the category of bounded chain complexes on a Quillen exact category $\cE$ with 
the class of all quasi-isomorphisms $(\Ch_b(\cE),\qis)$ is 
a consistent Waldhaseun exact category.\\
$\mathrm{(3)}$ 
Let $A$ be a Cohen-Macaulay ring and $p$ a non-negative integer 
less than $\dim A$. 
Let us denote the category of finitely generated $A$-modules $M$ 
whose codimension is greater than $p$ by $\cM^p_A$ and 
the full subcategory of $\cM^p_A$ consisting of those $A$-modules 
such that its projective dimension is less than $p$ by $\cM^p_A(p)$. 
Then one can easily prove that $\cM^p_A(p)$ is closed under extensions 
in $\cM^p_A$ and therefore it 
naturally becomes a Quillen exact category. 
A morphism $f:x \to y$ in $\cM^p_A(p)$ 
is a {\bf generic isomorphism} 
if the codimensions of $\Ker f$ and $\coker f$ are greater than $p+1$. 
We denote the class of generic isomorphisms in $\cM^p_A(p)$ by $w$. 
Then one can easily prove that $w$ satisfies the extensional axiom 
and $(\cM^p_A(p))^w=\{0\}$. 
Therefore $qw=\qis$ and obviously 
$w$ does not satisfy the consistent axiom. 
\end{ex}

\sn
Recall the definition of the category of admissible exact sequences in 
a category with cofibrations from Conventions $\mathrm{(7)}$ $\mathrm{(xiv)}$.

\begin{lemdf}
\label{lemdf:extensional, hom rel ex cat}
$\mathrm{(1)}$ 
Let $\cG$ and $\calH$ be strict exact subcategories of $\cE$ and 
we put $\bG:=(\cG,w\cap\cG)$ and $\bH=(\calH,w\cap\calH)$. 
Let $E(\bG,\bE,\bH)$ denote the relative exact category 
$(E(\cG,\cE,\calH),E(w))$. 
If $\bE$, $\bG$ and 
$\bH$ are strict, 
then $E(\bG,\bE,\bH)$ is also.\\
$\mathrm{(2)}$ 
We write $\HOM(\bE,\bF)$ for the relative category 
$(\HOM_{\RelEx^{+}}(\bE,\bF),\Mor\HOM_{\RelEx}(\bE,\bF))$. 
In other words, 
$\HOM(\bE,\bF)$ is a relative category 
whose underyling category is 
the category of relative exact functors from $\bE$ to $\bF$ 
and whose weak equivalences are relative natural equivalences. 
If $\bF$ is a Waldhausen exact category, 
then $\HOM(\bE,\bF)$ is a relative exact category. 
Here a sequence $f \onto{a} g \onto{b} h$ of relative exact functors 
from $\bE$ to $\bF$ is an admissible exact sequence 
if for any object $x$ in $\cE$, 
a sequence $f(x) \onto{a(x)} g(x) \onto{b(x)} h(x)$ is 
an admissible exact sequence in $\cF$.\\
$\mathrm{(3)}$ 
If $\bF$ is consistent, 
then the functor 
$\Ch_b:\HOM_{\RelEx^{+}}(\bE,\bF)\to\HOM_{\RelEx^{+}}(\Ch_b(\bE),\Ch_b(\bF))$ 
preserves relative natural weak equivalences. 
In particular, 
if $\bF$ is a consistent Waldhausen exact category, 
then 
the functor $\Ch_b:\HOM(\bE,\bF) \to \HOM(\Ch_b(\bE),\Ch_b(\bF))$ 
is a relative exact functor. 
\end{lemdf}

\begin{proof}[\bf Proof]
$\mathrm{(1)}$ 
We have an equality $E(\cG,\cE,\calH)^{E(w)}=E(\cG^w,\cE^w,\calH^w)$. 
Hence if $\bE$, $\bG$ and $\bH$ are strict, then $E(\bG,\bE,\bH)$ is strict. 

\sn
$\mathrm{(2)}$ 
If $w=i_{\cE}$ and $v=i_{\cF}$, 
then $\Mor\HOM_{\RelEx}(\bE,\bF)$ is $i_{\HOM_{\RelEx^{+}}(\bE,\bF)}$ 
and in this case, 
the assertion was essentially proven in \cite[A.11, A.18]{Uni}. 
For general case, 
only non-trivial point is that 
for any diagrams $g \linf h \to f$ and $g'\rdef h' \leftarrow f'$ 
of relative exact functors from 
$\bE$ to $\bF$, 
the functors $g\coprod_h f$, $g'\times_{h'}f':\cE \to \cF$ 
preserve weak equivalences. 
Let $a:x\to y$ be a morphism in $w$, 
then $g(a)\coprod_{h(a)}f(a)$ and 
$g'(a)\times_{h'(a)}f'(a)$ are in $v$ by the gluing and cogluing axioms. 
Hence we obtain the result. 

\sn
$\mathrm{(3)}$ 
For each natural transformation $\theta$ 
between relative exact functors 
$f$, $g:\bE=(\cE,w) \to \bF=(\cF,v)$, 
if $\theta$ is a relative natural equivalence, 
$\Ch_b(\theta):\Ch_b(f) \to \Ch_b(g)$ is contained in $lv$. 
Therefore if $v$ satisfies the consistent axiom, then 
$\Ch_b(\theta)$ is a relative natural transformation 
between the morphisms $\Ch_b(f)$, $\Ch_b(g):\Ch_b(\bE) \to \Ch_b(\bF)$.
\end{proof}

\section{Widely exact functors}
\label{sec:widely exact func}

The purpose of this section is to prove 
that for any strict relative exact category 
$\bE$, $\Ch_b(\bE)$ is a Waldhausen exact category 
which satisfies 
the extensional, the saturated, 
the factorization and the retraction axioms 
in \ref{prop:compness}. 
The key tools is the concept about {\it widely exact functors} 
which is roughly saying exact functors from 
suitable exact categories to triangulated categories. 
Here ``suitable" exact category means exact categories like 
the categories of complexes, 
namely {\it bicomplicial categories} in \ref{df:bicomp}. 
We review the notion of {\it null classes} of 
a bicomplicial category 
which is bicomplicial variant of 
triangulated subcategories in \ref{para:NC and comp weak equiv} 
and study 
functorial behaviour of lattices of null classes 
and lattices of triangulated subcategories 
by widely exact functors in \ref{cor:pullbackofwidelyexact} 
and \ref{ex:widelyexact}. 
At the end of this section, 
we will define the non-connective $K$-theory for 
consistent relative exact categories in \ref{df:non-connective K-theory} 
and prove that it is a categorical homotopy invariant functor 
in \ref{cor:homotopy inv KS}. 
We start by recalling the foundation of 
bicomplicial categories theory from \cite{Uni}. 

\begin{df}[\bf Bicomplicial category]  
\label{df:bicomp} 
$\mathrm{(1)}$ 
A {\bf bicomplicial category} is a system 
$(\cE,C,r,\iota,\sigma)$ consisting of 
a Quillen exact category $\cE$, an exact endofunctor $C:\cE \to \cE$, 
natural transformations $\iota:\id_{\cE} \to C$, $r:CC \to C$ 
and a natural equivalence $\sigma:CC\isoto CC$ 
which satisfies the following axioms:\\
$\mathrm{(i)}$ 
We have the equalities 
$rC(\iota)=r\iota_C=\id_{C}$, $\sigma C(\iota)=\iota_C$ and 
$\sigma\sigma=\id_{CC}$.\\
$\mathrm{(ii)}$ 
For any object $x$ in $\cE$, 
the morphism $\iota_x:x \to C(x)$ is 
an admissible monomorphism.\\ 
Then we put $T:=C/\id_{\cE}$ and call it the {\bf suspension functor}.\\
$\mathrm{(iii)}$ 
$T$ is essentially surjective and fully faithful.\\
We often omit $C$, $\iota$, $r$ and $\sigma$ in the notation.\\ 
$\mathrm{(2)}$ 
A {\bf complicial functor} between bicomplicial categories 
$\cE \to \cE'$ is a pair of an exact functor $f:\cE \to \cE'$ and 
a natural equivalence $c:C_{\cE'}f\isoto fC_{\cE}$ which satisfies the 
equality $c\iota_f^{\cE'}=f(\iota^{\cE})$. 
We often omit $c$ in the notation.\\ 
$\mathrm{(3)}$ 
For complicial functors 
$\cE \onto{(f,c)} \cE' \onto{(g,d)} \cE''$, 
their composition is defined by 
$(g,d)(f,c):=(gf,d\odot c)$ 
where we put $d \odot c:=g(c)d_f$.\\
$\mathrm{(4)}$ 
A {\bf complicial natural transformation} 
between complicial functors 
$(f,c)$, $(g,d):\cE \to \cE'$ from $(f,c)$ to $(g,d)$ 
is a natural transformation $\phi:f \to g$ 
which subjects to 
the condition that $d C\phi=\phi_C c$.\\
$\mathrm{(5)}$ 
We denote the $2$-category of essentially small bicomplicial categories 
by $\biComp$. 
\end{df}

\begin{ex}[\bf Bicomplicial categories] 
\label{ex:bicomp cat}
$\mathrm{(1)}$ 
Let $\cE$ be a Quillen exact category. 
Now we give the bicomplicial structure on $\Ch_{\#}(\cE)$ 
($\#=b, \pm, \emptyset$) as follows. 
The functor $C:\Ch_b(\cE) \to \Ch_b(\cE)$ 
is given by $x \mapsto Cx:=\Cone\id_x$ and 
for any complex $x$, we define morphisms 
${(\iota_x)}:x \to C(x)$, 
${(r_x)}:CC(x) \to C(x)$ 
and ${(\sigma_x)}:CC(x) \to CC(x)$ by 
$${(\iota_x)}_n={\footnotesize{
\begin{pmatrix}
0\\
\id_{x_n}
\end{pmatrix}}},\ 
{(r_x)}_n={\footnotesize{ 
\begin{pmatrix}
0 &\!\! \id_{x_{n-1}} &\!\! \id_{x_{n-1}} &\!\! 0\\
0 &\!\! 0 &\!\! 0 &\!\! \id_{x_n}
\end{pmatrix}}}
\ \text{and  } 
{(\sigma_x)}_n={\footnotesize{
\begin{pmatrix}
\!\!-\id_{x_{n-2}} &\!\! 0 &\!\! 0 &\!\! 0\\
\!\!0 &\!\! 0 &\!\! \id_{x_{n-1}} &\!\! 0\\
\!\!0 &\!\! \id_{x_{n-1}} &\!\! 0 &\!\! 0\\
\!\!0 &\!\! 0 &\!\! 0 &\!\! \id_{x_n}
\end{pmatrix}}}
.$$ 
Then the system $(\Ch_b(\cE),C,\iota,r,\sigma)$ 
forms a bicomplicial category.\\ 
$\mathrm{(2)}$ 
More generally, 
a complicial exact category in \cite[3.2.2]{Sch11} can be 
regarded as a bicomplicial category as follows. 
Let $\Ch_b(\bbZ)$ be the category of bounded chain complexes of 
finitely generated free abelian groups. 
(But in Ibid, 
we use the cohomological notation for complexes). 
Here the bilinear operation $\otimes$ on $\Ch_b(\bbZ)$ 
is given by the usual tensor products of complexes of abelian groups. 
(For more detail, see Ibid). 
We can easily learn that a complex in $\Ch_b(\bbZ)$ is 
isomorphic to finite direct sum of degree shift of the following two 
typical complexes: 
The unit complex $1\!\!1$ is $\bbZ$ 
in degree $0$ and $0$ elsewhere. 
For any positive integer $m$, 
the complex $C_m:=[\bbZ \onto{m} \bbZ]$ is 
$\bbZ$ in degrees $0$ and $1$ 
and the only non-trivial boundary map $d_1$ is 
given by multiplication of $m$. 
The complex $C_1$ has the natural 
commutative monoid object structure 
$I:1\!\!1 \to C_1$, $R:C_1\otimes C_1 \to C_1$ 
and $s:C_1\otimes C_1 \to C_1\otimes C_1$
which is implicitly explained 
in $\mathrm{(1)}$. 
In particular, 
if an exact category $\cC$ 
is complicial in the sense of \cite{Sch11}, 
that is, there is an action on $\cC$ by a biexact functor 
$-\otimes -:\Ch_b(\bbZ)\times \cC \to \cC$, 
it defines the bicomplicial structure on $\cC$ by 
$C:=C_1\otimes -$, $\iota:=I\otimes\id_{-}$, $r:=R\otimes\id_{-}$ 
and $\sigma:=s\otimes\id_{-}$. 
As we showed in \cite{Uni}, 
the almost arguments about complicial exact categories and 
complicial Waldhausen categories 
in \cite{Sch11} is essentially only using 
the commutative monoid object structure of $C_1$ and 
still works fine 
in the context of bicomplicial categories and bicomplicial pairs 
which will be defined in \ref{df:bicomp pair}. 
Therefore, 
in this paper, 
to make the definitions simplify, 
we review and utilize the theory 
of complicial exact categories (with weak equivalences) 
in \cite{Sch11} as the theory of bicomplicial categories 
and bicomplicial pairs with slight modifications.\\
$\mathrm{(3)}$ 
Recall the notations of morphisms categories from Conventions $\mathrm{(5)}$ $\mathrm{(ii)}$. 
If $\cE$ is a bicomplicial category, then $\Ar\cE$ naturally becomes 
a bicomplicial category and 
moreover $(\dom,\id)$, $(\ran,\id):\Ar\cE \to \cE$ are 
complicial functors and $\epsilon$ is a complicial natural transformation. 
\end{ex}

\begin{nt}[\bf $C$-homotopy equivalences]
\label{df:C-homotopic}
Let $(\cE,C,r,\iota,\sigma)$ be a bicomplicial category. 
$\mathrm{(1)}$ 
Morphisms $f$, $g:x\to y$ in $\cE$ are ({\bf $C$-}){\bf homotopic} 
if there exists a morphism $H:Cx \to y$ such that 
$f-g=H\iota_x$.\\ 
$\mathrm{(2)}$ 
A morphism $f:x\to y$ in $\cE$ is 
a ({\bf $C$-}){\bf homotopy equivalence} 
if there exists a morphism $g:y \to x$ such that 
$gf$ and $fg$ are homotopic to $\id_x$ and $\id_y$ respectively. 
Then we say that $x$ and $y$ are {\bf $C$-homotopy equivalent}.\\
$\mathrm{(3)}$ 
An object $x$ is ({\bf $C$-}){\bf contractible} 
if $x$ is $C$-homotopy equivalent to the zero object. 
Namely $\id_x$ is $C$-homotopic to the zero morphism.\\
$\mathrm{(4)}$ 
We can easily prove that $C$-homotopic is an equivalence relation 
on $\Hom$-sets of $\cE$ which is compatible with the composition. 
(See \cite[2.13]{Uni}). 
Therefore we can define the {\bf homotopy category} 
$\pi_0(\cE)$ 
of $\cE$ as follows. 
The object class of $\pi_0(\cE)$ is same as that of $\cE$,  
the class of morphisms in $\pi_0(\cE)$ is 
the homotopic class of morphisms in $\cE$ 
and the composition of morphisms 
is induced from $\cE$. 
We can easily prove that $\pi_0(\cE)$ is an additive category. 
\end{nt}

\begin{nt}[\bf Mapping cone functor, Mapping cylinder functor]
\label{nt:mapping cone functor}
Let $(\cE,C)$ be a bicomplicial category.\\
$\mathrm{(1)}$ 
We define the functors 
$$\Cone,\ \Cyl:\Ar\cE \to \cE$$ 
called the {\bf mapping cone functor} and 
the {\bf mapping cylinder functor} respectively as follows: 
$$\Cyl:=\ran\oplus C\dom$$
$$\Cone:=\ran\coprod_{\dom}C\dom$$
where $\Cone$ is defined the following push out diagram: 
$$
{\footnotesize{\xymatrix{
\dom \ar@{>->}[r]^{\iota_{\dom}} \ar[d]_{\epsilon} 
\ar@{}[dr]|\bigstar
& 
C\dom \ar[d]^{\mu} \ar@{->>}[r]^{\pi_{\dom}} & 
T\dom \ar[d]^{\id_{T\dom}}& \\
\ran \ar@{>->}[r]_{\kappa} & \Cone \ar@{->>}[r]_{\psi} & T\dom 
}}}$$
where $\psi$ is induced from the universal property of $\Cone$. 
We have the conflation 
$$\dom \overset{j_1}{\rinf} \Cyl \overset{\eta}{\rdef} \Cone$$ 
where $j_1:=
\footnotesize{
\begin{pmatrix}
\epsilon\\
-\iota_{\dom}
\end{pmatrix}
}$ 
and $\eta=
\footnotesize{
\begin{pmatrix}
\kappa & \mu
\end{pmatrix}
}$.\\
$\mathrm{(2)}$ 
Moreover we define 
the natural transformations $j_2:\ran \to \Cyl$, 
$j_3:C\dom \to\Cyl$ and 
$\beta:\Cyl \to \ran$ by 
$j_2:=
\footnotesize{
\begin{pmatrix}
\id_{\ran}\\
0
\end{pmatrix}
}$, 
$
j_3:=
\footnotesize{
\begin{pmatrix}
0\\
\id_{C\dom}
\end{pmatrix}}$
and 
$\beta:=
\footnotesize{
\begin{pmatrix}
\id_{\ran} & 0
\end{pmatrix}
}$. 
Then we have the following commutative diagram:
$${\footnotesize{\xymatrix{
\dom \ar@{>->}[r]^{j_1} \ar[rd]_{\epsilon} & \Cyl \ar[d]_{\beta} & 
\ran \ar@{>->}[l]_{j_2} \ar[ld]^{\id_{\ran}}\\
& \ran & .
}}}$$ 
$\mathrm{(3)}$ 
By declaring that $T$-sequence is $T$-exact if it is isomorphic to 
the following type $T$-exact triangle
$$x \onto{f} y \onto{\kappa_f} \Cone f \onto{\psi_f} Tx,$$
we can make $(\pi_0(\cE),T)$ into a triangulated category. 
(cf. \cite[3.29]{Uni}).
\end{nt}

\begin{para}[\bf Null classes and complicial weak equivalences] 
\label{para:NC and comp weak equiv}
(cf. \cite[\S 5.1]{Uni}).
Let $(\cE,C)$ be a bicomplicial category.\\
$\mathrm{(1)}$ 
We say that a full subcategory $\cN$ of $\cE$ is a 
{\bf null class} if it contains all $C$-contractible objects 
in $\cE$ and for any admissible short exact sequences 
$ x\rinf y \rdef z$ in $\cE$, 
if two of $x$, $y$ and $z$ are in $\cN$, 
then the other is also in $\cN$. 
We call the last condition the 
{\bf two out of three for admissible short exact sequences axiom}.\\
$\mathrm{(2)}$ 
We say that a null class of $\cN$ is {\bf thick} 
if $\cN$ is closed under retractions. 
In this case, $\pi_0(\cN)$ is a 
thick triangulated subcategory of $\pi_0(\cE)$.\\
$\mathrm{(3)}$ 
If $\cE$ is essentially small, we 
denote the set of null classes (resp. thick null classes) in $\cE$ 
by $\NC_{\emptyset}(\cE)$ or $\NC(\cE)$ (resp. $\NC_{\thi}(\cE)$).\\
$\mathrm{(4)}$ 
A class of morphisms $w$ in $\cE$ is a 
{\bf class of complicial weak equivalences} 
if it satisfies the saturated, the extensional axioms and 
if it contains all $C$-homotopy equivalences. 
We call a morphism in $w$ a {\bf complicial weak equivalence}.\\
$\mathrm{(5)}$ 
A class of complicial weak equivalences is {\bf thick} 
if it satisfies the retraction axiom.\\
$\mathrm{(6)}$ 
If $\cE$ is essentially small, we 
denote the set of classes of complicial weak equivalences 
(resp. thick classes of complicial weak equivalences) in $\cE$ 
by $\CW_{\emptyset}(\cE)$ or $\CW(\cE)$ (resp. $\CW_{\thi}(\cE)$).\\
$\mathrm{(7)}$ 
We can easily check that $\CW_{\#}(\cE)$ and 
$\NC_{\#}(\cE)$ are 
complete lattices 
by \ref{lem:complete criterion} where $\#=\emptyset$ or $\thi$. 
Therefore they define the functors. 
$$\NC_{\#},\ \CW_{\#}:\biComp^{\op} \to \CLat.$$
where for any complicial functor 
$f:\cE \to \cE'$, $\NC_{\#}(f)$ and $\CW_{\#}(f)$ are defined 
by pull-back by $f$.\\
$\mathrm{(8)}$ 
For any full subcategory $\cN$ of $\cE$, 
we write $\cN_{\nul,\cE}$ or simply $\cN_{\nul}$ 
(resp. $\cN_{\thi,\cE}$ or simply $\cN_{\thi}$) for 
the smallest null class (resp. thick null class) containing $\cN$ 
and call it the {\bf null} 
(resp. {\bf thick}) {\bf closure of $\cN$}. 
For any class of morphisms $v$ in $\cE$, 
we write $v_{\comp,\cE}$ or simply $v_{\comp}$ 
(resp. $v_{\thi,\cE}$ or simply $v_{\thi}$) for 
the smallest class of complicial weak equivalences 
(resp. thick complicial weak equivalences) containing $v$ 
and call it the {\bf complicial} (resp. {\bf thick complicial}) 
{\bf closure of $v$}.\\
$\mathrm{(9)}$ 
For any full subcategory $\cN$ of $\cE$, 
we put 
$$w_{\cN}:=\{f\in\Mor\cE;\Cone f\in \cN \}.$$
Then the association $\cN \mapsto w_{\cN}$ gives the 
natural equivalence $\NC_{\#}\isoto\CW_{\#}$ ($\#=\emptyset$ or $\thi$). 
Here the inverse is given by $w \mapsto \cE^w$.\\
$\mathrm{(10)}$ 
For any class of complicial weak equivalences $w$, 
$\cE^w$ is $w$-closed in the following sense. 
For any objects $x$ and $y$ in $\cE$, 
if $x$ is in $\cE^w$ and if 
there exists a morphism $x \to y$ or $y \to x$ in $w$, 
then $y$ is also in $\cE^w$. 
For in this situation, by the saturated axiom, 
the assertion that $0 \to x$ is in $w$ 
implies the assertion that $0 \to y$ is in $w$.
\end{para}

\begin{para}[\bf Frobenius exact structure] 
\label{para:Frob exact structure}
(cf. \cite[6.5, B.16]{Sch11}, \cite[2.26, 2.27, 2.30, 2.40]{Uni}). 
Let us recall that a Quillen exact category $\cF$ is 
{\bf Frobenius} if the class of projective objects in $\cF$ is 
equal to that of injective objects in $\cF$ and if 
$\cF$ has enough proj-inj objects. 
For any bicomplicial category $(\cE,C)$, 
it naturally has 
the Frobenius exact category structure as follows. 
An admissible monomorphism 
(resp. admissible epimorphism) $x \onto{a} y$ is 
{\bf Frobenius} 
if for any object $u$ and a morphism $x \onto{f} Cu$ 
(resp. $Cu \onto{f} y$), 
there exists a morphism $y \onto{g} Cu$ 
(resp. $Cu \onto{g} x$) such that $f=ga$ (resp. $f=ag$). 
In an admissible short exact sequence 
$ x \overset{i}{\rinf} y \overset{p}{\rdef} z$ in $\cE$, 
$i$ is Frobenius if and only if $p$ is. 
In this case we call 
the sequence $x \overset{i}{\rinf} y \overset{p}{\rdef} z$ 
a {\bf Frobenius admissible short exact sequence}. 
The Quillen exact category $\cE$ with 
Frobenius admissible short exact sequences forms 
a Frobenius exact category and 
we denote it by $\cE_{\frob}$. 
An object in $\cE_{\frob}$ is a proj-inj object if and only if 
it is a $C$-contractible object. 
Moreover $(\cE_{\frob},C)$ again becomes a bicomplicial category. 
For example, for any Quillen exact category $\cF$, 
a Frobenius admissible short exact sequence 
in the standard bicomplicial category 
$\Ch_b(\cF)$ is just a degree-wised split exact sequence. 
\end{para}

\begin{df}[\bf Bicomplicial pair]
\label{df:bicomp pair}
$\mathrm{(1)}$ 
A {\bf bicomplicial pair} is a pair $\bC=(\cC,w)$ of 
bicomplicial category $\cC$ and a class of complicial weak equivalences 
$w$ in $\cC$.\\
$\mathrm{(2)}$ 
A {\bf relative complicial functor} between bicomplicial categories 
$\bC \to \bC'$ is a complicial functor such that 
the underlying functor preserves weak equivalences.\\
$\mathrm{(3)}$ 
A {\bf relative complicial natural transformation} 
between relative complicial functors 
is just a complicial natural transformation.\\
$\mathrm{(4)}$ 
A {\bf relative complicial natural weak equivalence} 
is a relative complicial natural transformation 
such that the underlying natural transformation is 
a relative natural equivalence.\\
$\mathrm{(5)}$ 
We denote the $2$-category of essentially small bicomplicial pairs, 
relative complicial functors and relative complicial natural transformations 
(resp. relative complicial natural weak equivalences) 
by $\biCompPair^{+}$ (resp. $\biCompPair$).\\
$\mathrm{(6)}$ 
A bicomplicial pair $\bE=(\cE,w)$ is {\bf thick} if 
$w$ is thick. 
We write $\biCompPair_{\thi}^{\#}$ for 
the full $2$-subcategory of thick bicomplicial pairs 
in $\biCompPair^{\#}$ for $\#\in\{+,\text{nothing}\}$.\\
$\mathrm{(7)}$ 
For any bicomplicial pair $\bC=(\cC,w)$, 
$\pi_0(\cC^w)$ is a triangulated subcategory of $\pi_0(\cC)$ 
by \ref{para:NC and comp weak equiv} 
$\mathrm{(2)}$ and $\mathrm{(9)}$ 
and we put $\cT(\bC)=\cT(\cC,w):=\pi_0(\cC)/\pi_0(\cC^w)$. 
This association defines the $2$-functor 
$\cT:\biCompPair^{+} \to \TriCat$ 
which sends relative complicial natural weak equivalences 
to triangulated natural equivalences. 
\end{df}

\begin{rem}
\label{rem:a bicomp pair is good biwaldhausen} 
$\mathrm{(1)}$ 
(\cf \cite[5.18]{Uni}). 
A bicomplicial pair is a saturated extensional 
Waldhausen exact category which satisfies the factorization axiom.\\
$\mathrm{(2)}$ 
(\cf \cite[5.3]{Uni}). 
For a thick bicomplicial pair $(\cC,w)$, 
$\pi_0(\cC^w)$ is a thick subcategory of $\pi_0(\cC)$. 
\end{rem}

\sn
Recall the definition of very strict relative exact categories 
from \ref{df:rel exact cat} $\mathrm{(5)}$. 

\begin{prop}
\label{prop:bicompair is very strict}
A bicomplicial pair is very strict. 
\end{prop}

\sn
To prove Proposition~\ref{prop:bicompair is very strict}, 
we utilize the following lemma. 

\begin{lem}
\label{lem:fully faithful}
{\rm (\cf \cite[3.1.7 (b)]{Sch11}).} 
Let $\cE$ be an exact category and $\cF\overset{i}{\rinc} \cE$ 
a strict exact subcategory of $\cE$. 
We assume that the following condition $\mathrm{(\ast)}$ holds. 
Then $i$ induces a fully faithful functor $\calD_b(\cF) \to \calD_b(\cE)$. 

\sn
$\mathrm{(\ast)}$ 
For any admissible monomorphism $x \overset{a}{\rinf} y$ in $\cE$ with $x\in \Ob\cF$, 
there exists a morphism $y \onto{b} z$ with $z \in \Ob\cF$ such that 
the composition $ba:x \to z$ is an admissible monomorphism. 
\qed
\end{lem}

\begin{proof}[\bf Proof of \ref{prop:bicompair is very strict}.]
Let $\bC=(\cC,w)$ be a bicomplicial pair. 
We are going to check that $\cC^w\rinc \cC$ satisfies 
the condition $\mathrm{(\ast}$ in \ref{lem:fully faithful}. 
For any admissible monomorphism $x \rinf y$ with $x\in\Ob\cC^w$, 
the morphism $\iota_y:y \rinf C(y)$ is 
an admissible monomorphism with $C(y)\in\Ob \cC^w$. 
Hence we get the result by \ref{lem:fully faithful}. 
\end{proof}

\sn
Recall the notations of $E(\cE)$ and $s^{\cE}$, $q^{\cE}$ 
from Conventions $\mathrm{(7)}$ $\mathrm{(ii)}$ and $\mathrm{(xiv)}$.

\begin{df}[\bf Widely exact functors]
\label{para:widely exact}
Let $\cE$ be a bicomplicial category and $(\cT,\Sigma)$ 
a triangulated category. 
A {\bf widely exact functor} from $\cE$ to $\cT$ is 
a pair of $(f,\partial)$ consisting of 
an additive functor $f:\cE \to \cT$ and a natural equivalence 
$\partial:fq^{\cE} \to \Sigma fs^{\cE}$ 
satisfies the following conditions:\\
$\mathrm{(1)}$ 
For any admissible exact sequence 
$x \overset{i}{\rinf} y \overset{p}{\rdef} z$ in $\cE$, 
$fx \onto{fi} fy \onto{fp} fz \onto{\partial_{(i,p)}} \Sigma fx$ 
is a $\Sigma$-exact triangle.\\
$\mathrm{(2)}$ 
For any $x$ in $\cE$, 
$\partial_x:=\partial_{(\iota_x,\pi_x)}:fTx \to \Sigma fx$ is an isomorphism 
where $x \overset{\iota_x}{\rinf} Cx \overset{\pi_x}{\rdef} Tx$ is 
the canonical admissible exact sequence in $\cE$.

\sn
We often omit $\partial$ in the notation. 
We write $\Ex_{\wide}(\cE,\cT)$ for 
the set of widely exact functors from $\cE$ to $\cT$. 
\end{df}

\begin{ex}[\bf Compositions]
\label{ex:composition}
Let $(f,\partial):\cE \to \cT$ be a widely exact functor 
from a bicomplicial category $\cE$ to a triangulated category 
$(\cT,\Sigma)$.\\
$\mathrm{(1)}$ 
For a complicial functor $(f,c):\cE' \to \cE$ from 
a bicomplicial category 
$\cE'$, 
we have the widely exact functor $gf:\cE' \to \cT$.\\
$\mathrm{(2)}$ 
For a triangulated functor $(h,d):\cT \to \cT'$ from $\cT$ to a 
triangulated category $\cT'$, 
we have the widely exact functor $hf:\cT \to \cT'$.
\end{ex}

\begin{lem}[\bf Fundamental properties of widely exact functors]
\label{lem:FPOWEF}
Let $(f,c):\cE\to \cT$ be a widely exact functor 
from a bicomplicial category $(\cE,C)$ to a triangulated category 
$(\cT,\Sigma)$. 
Then\\
$\mathrm{(1)}$ 
For any object $x$ in $\cE$, 
$fCx$ is the zero object in $\cT$.\\
$\mathrm{(2)}$ 
For any morphisms $a$, $b:x \to y$ in $\cE$, 
if $a$ and $b$ are homotopic, 
then we have the equality $fa=fb$.\\
$\mathrm{(3)}$ 
If $x \onto{a} y$ is a homotopy equivalence, 
then $fa$ is an isomorphism.\\
$\mathrm{(4)}$ 
If an object $x$ in $\cE$ is contractible, 
then $fx$ is the zero object in $\cT$.
\end{lem}

\begin{proof}[\bf Proof] 
$\mathrm{(1)}$ 
For any $x$ in $\cE$, by considering 
the diagram of $\Sigma$-exact triangles below 
$${\footnotesize{\xymatrix{
fx \ar[r]^{\iota_x} \ar@{=}[d] & fCx \ar[r]^{\pi_x} 
\ar[d] & fTx \ar[r]^{\partial_{(\iota_x,\pi_x)}} 
\ar[d]^{\wr}_{-\partial_{(\iota_x,\pi_x)}} & 
\Sigma fx \ar@{=}[d]\\
fx \ar[r] & 0 \ar[r] & \Sigma fx \ar[r]_{-\id_{\Sigma fx}} & \Sigma fx ,
}}}$$
we learn that $fCx$ is the zero object. 

\sn
$\mathrm{(2)}$ 
For any morphism $a$, $b:x \to y$, 
if there exists a homotopy $H:a \to b$, 
then we have the equalities $fa-fb=fHf\iota_x=0$ by (1). 

\sn
$\mathrm{(3)}$ 
For morphisms $a:x \to y$ and $b:y\to x$ such that 
$ba$ and $ab$ are homotopic to the identity morphisms, 
$fb$ and $fa$ are the inverse morphisms in each other by (2). 

\sn
$\mathrm{(4)}$ 
Let $x$ be a contractible object, then 
since the canonical morphism $x \to 0$ is a homotopy equivalence, 
$fx \to 0$ is an isomorphism in $\cT$ by (3). 
\end{proof}

\begin{cor}
\label{cor:pullbackofwidelyexact}
Let $f:\cE \to \cT$ be a widely exact functor 
from a bicomplicial category $\cE$ to a triangulated category $\cT$.\\
$\mathrm{(1)}$ 
For a triangulated subcategory $\calL$ of $\cT$, 
$f^{-1}\calL$ is a null class in $\cE$.\\
$\mathrm{(2)}$ 
In $\mathrm{(1)}$, moreover if $\calL$ is thick, 
then $f^{-1}\calL$ is also thick.\\

\sn
In particular, $f$ induces the ordered map 
$\Tri_{\#}(\cT) \to \NC_{\#}(\cE)$, $\cN \mapsto f^{-1}\cN$ 
where $\#=\emptyset$ or $\thi$. 
\end{cor}

\begin{proof}[\bf Proof]
$\mathrm{(1)}$ 
For a contractible object $x$ in $\cC$, 
since $f(x)$ is the zero object, 
$x$ is in $f^{-1}\calL$. 
For an admissible short exact sequence 
$x \rinf y \rdef z$, 
since $fx \to fy \to fz \to \Sigma fx$ is 
$\Sigma$-exact, 
if two of $fx$, $fy$ and $fz$ are in $\calL$, 
then third one is also in $\calL$ and therefore 
$f^{-1}\calL$ is a null class. 
Assertion $\mathrm{(2)}$ follows from \ref{lem:ret in tri cat}. 
\end{proof}

\begin{prop}[\bf Widely exact functors]
\label{ex:widelyexact}
$\mathrm{(1)}$ 
For a bicomplicial category $\cE$, 
the canonical functor 
$$\omega_{\cE}:\cE_{\frob} \to \pi_0(\cE)$$
is widely exact.\\
$\mathrm{(2)}$ 
In the situation above, 
let $(\cT,\Sigma)$ be a triangulated category. 
Then the association 
$$\omega_{\cE}^{\ast}:\Hom_{\TriCat}(\pi_0(\cE),\cT)\ni g 
\mapsto g\circ \omega_{\cC} \in\Ex_{\wide}(\cE_{\frob},\cT)$$
gives a bijective correspondence. 
Moreover $\omega_{\cE}$ induces the lattices isomorphism 
$$\Tri_{\#}(\pi_0(\cE))\isoto \NC_{\#}(\cE_{\frob}),\ \ \cN\mapsto \omega_{\cE}^{-1}\cN .$$ 
$\mathrm{(3)}$ 
For a bicomplicial pair $\bE=(\cE,w)$, 
the canonical functor
$$\widetilde{\omega}_{\cE}:\cE \to \cT(\cE,w)$$
is widely exact.\\
$\mathrm{(4)}$ 
In the situation above, 
if $\bE$ is thick, then 
the canonical functor 
$\widetilde{\omega}_{\cE}$ 
induces the lattice isomorphism 
$$\{\cN\in\NC_{\#}(\cE_{\frob});\cE^w \subset \cN \}\isoto
\Tri_{\#}(\cT(\cE,w)).$$
\end{prop}

\begin{proof}[\bf Proof] 
$\mathrm{(1)}$ is proven in \cite[3.25]{Uni}. 

\sn
$\mathrm{(2)}$ 
For any widely exact functor $(f,\partial):\cE_{\frob} \to \cT$, 
let us define the triangle functor 
$(\bar{f},\bar{\partial}):\pi_0(\cE) \to \cT$ 
by the formula $\bar{f}(\omega_{\cE}{a}):=f(a)$ for any morphism 
$a:x\to y$ in $\cE$. 
By virtue of \ref{lem:FPOWEF} $\mathrm{(2)}$, 
this association is well-defined and 
the association $f \mapsto \bar{f}$ gives the inverse map of 
$\omega_{\cE}^{\ast}$. 
For the second assertion, 
the association $\cN \mapsto \pi_0(\cN)$ gives the inverse map of 
$\omega_{\cE}^{-1}$. 

\sn
$\mathrm{(3)}$ 
For an admissible short exact sequence 
$x \overset{i}{\rinf} y \overset{p}{\rdef} z$ 
in $\cE$, 
let us consider the following commutative 
diagram of admissible short exact sequences 
$${\footnotesize{\xymatrix{
x \ar@{>->}[r]^{j_1(i)} \ar@{=}[d] & 
\Cyl i \ar@{->>}[r]^{\eta_i} \ar[d]_{\beta_i} & 
\Cone i \ar[d]^{\Cone(0,p)}\\
x \ar@{>->}[r]_i & y \ar@{->>}[r]_p & z .
}}}$$ 
Since $\beta_i$ is a homotopy equivalence, 
it turns out that $\Cone(0,p)$ is a weak equivalence. 
Therefore the sequence 
$x \onto{i} y \onto{p} z \onto{\psi_i\Cone(0,p)^{-1}} Tx$ 
is a $T$-exact triangle in $\cT(\bE)$.

\sn
$\mathrm{(4)}$ Since $w$ is thick, 
$\pi_0(\cE^w)$ is a thick subcategory of $\pi_0(\cE)$ 
by \ref{rem:a bicomp pair is good biwaldhausen} $\mathrm{(2)}$. 
Therefore we have the lattice isomorphisms 
$$\NC_{\#}(\cE_{\frob})\isoto\Tri_{\#}(\pi_0(\cE)) \ \ \text{and}$$
$$V_{\#}(\pi_0(\cE^w))\isoto \Tri_{\#}(\cT(\cE,w))$$
by $\mathrm{(1)}$ and \ref{prop:homthmfortricat}. 
Hence we obtain the desired lattice isomorphism. 
\end{proof}

\begin{cor} 
\label{cor:inv by thick closure}
For any bicomplicial pair $(\cE,w)$, 
the identity functor $(\cE,w) \to (\cE,w_{\thi,\cE_{\frob}})$ 
induces an equivalence of triangulated categories 
$\cT(\cE,w)\isoto\cT(\cE,w_{\thi,\cE_{\frob}})$.
\end{cor}

\begin{proof}[\bf Proof] 
By \ref{ex:widelyexact} $\mathrm{(2)}$, 
$\pi_0(\cE^{w_{\thi,\cE_{\frob}}})=\pi_0(\cE^w)_{\thi}$. 
Therefore we have 
$$\cT(\cE,w)=\pi_0(\cE)/\pi_0(\cE^w)\isoto
\pi_0(\cE)/\pi_0(\cE^w)_{\thi}=\cT(\cE,w_{\thi,\cE_{\frob}}).$$
\end{proof}

\sn
Recall the definition of cokernel of triangulated functors from 
Conventions $\mathrm{(8)}$ $\mathrm{(v)}$.

\begin{df}
\label{df:coker of rel exact cat}
$\mathrm{(1)}$ 
Let $f:\bE=(\cE,w) \to \bF=(\cF,v)$ be a relative exact functor 
between strict relative exact categories. 
Then we define the relative exact category $\coker f$ by 
$(\Ch_b(\cF),w_{f})$. 
Here the morphism $a:x\to y$ in $\Ch_b(\cF)$ is in $w_f$ 
if and only if the image of $a$ by composition of 
the canonical projection functors 
$\Ch_b(\cF)\to\calD_b(\cF)\to\coker \calD_b(f)$ is an isomorphism.\\
$\mathrm{(2)}$ 
If $f:\bE \to \bF$ is derived fully faithful, 
then we write $\bF/\bE$ for $\coker f$.\\
$\mathrm{(3)}$ 
If $\bF$ is consistent, then 
the functor $j_{\cF}:\cF \to \Ch_b(\cF)$ 
induces a relative exact functor 
$\pi_f:\bF \to \coker f$.\\
$\mathrm{(4)}$ 
Let 
$${\footnotesize{\xymatrix{
\bE \ar[r]^f \ar[d]_a & \bF \ar[d]^b\\
\bE'\ar[r]_{f'} & \bF'
}}}$$ 
be a commutative diagram of strict relative exact categories. 
Then we define $\coker(a,b):\coker f \to \coker f'$ 
to be a relative exact functor by $\coker(a,b):=\Ch_b(b)$.
\end{df}

\begin{rem}
\label{rem:coker of rel exact cat}
$\mathrm{(1)}$ 
In \ref{df:coker of rel exact cat} $\mathrm{(1)}$, 
we have the canonical equivalence of triangulated categories 
$\coker \calD_b(f)\isoto\cT(\Ch_b(\cF),w_f)$.\\
$\mathrm{(2)}$ 
In \ref{df:coker of rel exact cat} $\mathrm{(4)}$,
if both $\bF$ and $\bF'$ are consistent, then the diagram below is commutative.
$${\footnotesize{\xymatrix{
\bF \ar[r]^{\pi_f} \ar[d]_b & \coker f\ar[d]^{\coker(a,b)}\\
\bF' \ar[r]_{\pi_{f'}} & \coker f'.
}}}$$ 
\end{rem}

\begin{ex}
\label{ex:quotient of rel exact cat}
For any strict relative exact category $\bE=(\cE,w)$, 
we have the equality 
$$\Ch_b(\bE)=\coker((\cE^w,i_{\cE^w})\to (\cE,i_{\cE})).$$
\end{ex}

\sn
Recall the definition about exact sequences of 
triangulated categories 
from Conventions $\mathrm{(8)}$ $\mathrm{(vi)}$.

\begin{prop}
\label{prop:compness}
Let $f:\bE=(\cE,w)\to \bF(\cF,v)$ be 
a relative exact functor between 
strict relative exact categories. 
Then\\ 
$\mathrm{(1)}$ 
$\coker f$ is a thick bicomplicial pair. 
In particular, 
$\Ch_b(\bE)$ is a thick bicomplicial pair. 
In particular, 
$\Ch_b(\bE)$ is a Waldhausen exact category 
which satisfies 
the extensional, the saturated, the factorization and 
the retraction axioms.\\
$\mathrm{(2)}$ 
The inclusion functor $\Ch_b(\cE)^{qw}\rinc \Ch_b(\cE)$ 
and the identity functor of $\Ch_b(\cE)$ induce 
an exact sequence of triangulated categories
$$\cT(\Ch_b(\cE)^{qw},\qis) \to \cT(\Ch_b(\cE),\qis)\to \cT(\Ch_b(\cE),qw).$$
\end{prop}

\begin{proof}[\bf Proof]
$\mathrm{(1)}$ 
$w_f$ is corresponding to the null class $\Acy^{w_f}_b\cF$ 
which is the pull back of $\{0\}$ by the composition of 
the widely exact functor and the triangulated functors
$$\Ch_b(\cF) \to \calD_b(\cF) \to \calD_b(\bF) \to \coker \calD_b(f).$$
Therefore by \ref{ex:composition} and \ref{cor:pullbackofwidelyexact}, 
it is a class of thick complicial weak equivalences. 
The last assertion follows from \ref{rem:a bicomp pair is good biwaldhausen}.

\sn
$\mathrm{(2)}$ 
We only need to check that 
$\cT(\Ch_b(\cE)^{qw},\qis)$ is a thick subcategory of $\cT(\Ch_b(\cE),\qis)$. 
We apply \ref{ex:widelyexact} $\mathrm{(4)}$ 
to a thick bicomplicial pair $(\Ch_b(\cE),\qis)$. 
A thick null class ($\Acy^b(\cE)\subset$)$\Ch_b(\cE)^{qw}$ 
of $\Ch_b(\cE)_{\frob}$ 
corresponds to 
a thick subcategory $\cT(\Ch_b(\cE)^{qw},\qis)$ of $\cT(\Ch_b(\cE),\qis)$.
\end{proof}

\begin{cor}
\label{cor:Ch_b is 2-functor}
$\mathrm{(1)}$ 
The $2$-functor $\Ch_b$ induces the $2$-functor 
$\RelEx_{\strict}^{+} \to \biCompPair_{\thi}$.\\
$\mathrm{(2)}$ 
The functor $\calD_b:\RelEx_{\strict}^{+} \to \TriCat$ 
is a $2$-functor and the restriction 
$\RelEx_{\consist}^{+} \to \TriCat$ sends 
relative natural weak equivalences to 
triangulated natural weak equivalences. 
In particular, 
for a morphism $f:\bE \to \bF$ in $\RelEx_{\strict}$, 
if $F$ is consistent and if $f$ is an exact homotopy equivalence, 
then it is a derived equivalence. 
\end{cor}

\begin{proof}[\bf Proof]
Assertion $\mathrm{(1)}$ follows from \ref{prop:compness} $\mathrm{(1)}$. 
Since $\calD_b$ is canonically isomorphic to $\cT\Ch_b$, 
assertion $\mathrm{(2)}$ follows from \ref{lemdf:extensional, hom rel ex cat} $\mathrm{(3)}$. 
\end{proof}

\sn
Now we give a definition of non-connective $K$-theory of 
consistent relative exact categories.

\begin{df}[\bf Non-connective $K$-theory]
\label{df:non-connective K-theory}
For any consistent relative exact category $\bE:=(\cE,w)$, 
we define the non-connective $K$-theory $\bbK(\bE)=\bbK(\cE;w)$ of $\bE$ by 
the non-connective $K$-theory of the Frobenius pair 
$(\Ch_b(\cE)_{\frob},\Ch_b(\cE)^{qw}_{\frob})$. 
Then $\bbK$ is the functor from $\RelEx_{\consist}$ to the stable category of 
spectra. 
\end{df}

\sn 
Recall the definition of (categorical) homotopy invariant functors from 
Conventions $\mathrm{(6)}$ $\mathrm{(iii)}$. 
We will prove in \ref{cor:homotopy inv KS} 
that 
the non-connective $K$-theory is a 
categorical homotopy invarinat functor. 
This result is implicitly utilzing in the proof of Lemma~7 in \cite{Sch06} 
to do the Eilenberg swindle argument. 
(See also \ref{prop:fund property of add theory} $\mathrm{(2)}$). 
First recall the following remarks. 

\begin{rem}
\label{rem:hotopy equivalences}
It is well-known that the functor $K^W$ is 
a categorical homotopy invariant functor. 
Namely, if morphisms between 
Waldhausen exact categories are weakly homotopic, 
then they induce the same maps on 
their Waldhausen connective $K$-theory spectra in 
the stable category of spectra. (See \cite[p.330 1.3.1]{Wal85}). 
Therefore an exact homotopy equivalence between Waldhausen exact categories 
is a $K^W$-equivalence. 
Namely it induces a homotopy equivalence on Waldhausen $K$-theory. 
\end{rem}

\begin{rem}
\label{rem:Schlichting F and S}
We recall the definition of 
the functors 
$\fF$ and $\fS$ on $\biCompPair$ from 
\cite[3.2.23]{Sch11} with slight minor adjustments. 
Let $\bC=(\cC,v)$ and $\bC'=(\cC',v')$ be bicomplicial pairs. 
The argument is carried out in several steps.\\
$\mathrm{(1)}$ 
First let us recall 
the {\bf countable envelope functor} $\fF$ 
from \cite[Appendix B]{Kel90} 
where $\fF\cE$ was written by $\cE^{\sim}$. 
We just use the following properties. 
$\fF$ is a $2$-functor from $\ExCat$ to $\ExCat$ 
which preserves admissible exact sequences of exact functors 
and there exist a natural transformation $\id_{\ExCat} \to \fF$.\\
$\mathrm{(2)}$ 
Therefore by \cite[2.61]{Uni}, 
$\fF$ induces the $2$-functor 
$\fF:\biCompPair \to \biCompPair$ 
and there exists a natural transformation 
$\xi:\id_{\biCompPair}\to\fF$.\\
$\mathrm{(3)}$ 
We define the $2$-functor $\fF:\biCompPair \to \biCompPair$ 
by $\fF\bC=(\fF\cC,w_{\fF\bC})$ 
where a morphism $a:x\to y$ in $\fF\cC$ is in $w_{\fF\bC}$ 
if and only if 
the image of $a$ by compositions of 
the canonical projection functors 
$\fF\cC \to \pi_0(\fF\cC) \to \coker(\pi_0(\fF(\cC^w))\to \pi_0(\fF\cC))$ 
is an isomorphism. 
Then $\xi$ induces a natural transformation 
$\xi:\id_{\biCompPair} \to \fF$.\\
$\mathrm{(4)}$ 
By using \ref{para:NC and comp weak equiv} $\mathrm{(9)}$, 
it turns out that 
$\fF$ preserves 
relative complicial natural weak equivalences in the following way. 
Let $q:a \to b$ be a relative complicial natural weak equivalence 
between relative complicial functors $a$, $b:\bC \to \bC'$. 
Then for any object $x$ in $\cC$, $\Cone q_x$ is in $\cC^{v'}$. 
Therefore, 
for any objects $x'$ in $\fF\cC$, $\Cone \fF q_{x'}$ 
is in $\fF(\cC^{v'})$. 
Hence $\fF q$ is in $w_{\fF\bC'}$.\\
$\mathrm{(5)}$ 
We define the $2$-functor $\fS:\biCompPair \to \biCompPair$ 
by $\fS\bC=(\fS\cC,w_{\fS\bC})$ 
where a morphism $a:x\to y$ in $\fF\cC$ is in $w_{\fS\bC}$ 
if and only if 
the image of $a$ by compositions of 
the canonical projection functors 
$\fF\cC \to \cT(\fF\bC) \to \coker(\cT(\bC)\onto{\cT(\xi_{\bC})} \cT(\fF\bC))$ 
is an isomorphism. 
By construction and $\mathrm{(4)}$, 
the functor $\fS$ 
also preserves relative complicial natural weak equivalences. 
\end{rem}

\begin{cor}
\label{cor:homotopy inv KS}
The non-connective $K$-theory is a categorical homotopy invarinat functor 
from $\RelEx_{\consist}$ to the stable category of spectra.
\end{cor}

\begin{proof}[\bf Proof]
For any non-negative integer $n$, the 
functor $K^W(\fS^n\Ch_b(-))$ from $\RelEx_{\consist}$ 
to the stable category of spectra is homotopy invariant 
by virtue of \ref{lemdf:extensional, hom rel ex cat} $\mathrm{(3)}$, 
\ref{rem:hotopy equivalences} and 
\ref{rem:Schlichting F and S}. 
Therefore we obtain the result. 
\end{proof}

\section{Total quasi-weak equivalences}
\label{sec:tw}

In this section, 
let $\bE=(\cE,w)$ be a bicomplicial pair. 
The main theme in this section is defining and studying 
the class of {\it total quasi-weak equivalences} $tw$ on 
$\Ch_b(\cE)$. 
The pivot in this section is the theorem~\ref{thm:tw=qis+lw} 
and as its by-product, 
we establish 
the universal property of $\Ch_b(\bG)$ 
for any consistent relative exact category $\bG$ 
in \ref{cor:univ} 
and 
learn that $\bE$ is 
a consistent Waldhausen exact category in \ref{prop:bicomp is adm} 
and as its corollary, 
we obtain the Gillet-Waldhausen theorem for strict relative exact categories 
in \ref{cor:comp of derived cat}. 

\begin{rem}
\label{rem:two comp str on chbE} 
$\Ch_b(\cE)$ has two bicomplicial structure. 
Namely, the complicial structure induced from $\cE$ which is 
denoted by $\bA$. 
The second complicial structure is the usual structure on 
the category of complexes as 
in \ref{ex:bicomp cat} $\mathrm{(1)}$. 
We denote it by $\bB$. 
For example $\Acy^{lw}_b\cE$ is a null class in $\bA$ but is not 
closed under $C^{\bB}$-contractible objects. 
$\Acy_b\cE$ is a null class in $\bB_{\frob}$ 
by \ref{ex:widelyexact} $\mathrm{(2)}$. 
\end{rem}

\begin{nt}[\bf Total functor]
\label{nt:tot}
(cf. \cite[4.15, 4.16]{Uni}). 
There exists the triple $(\Tot_{\cE},c^{\bA},c^{\bB})$ 
consisting of an exact functor 
$\Tot_{\cE}=\Tot:\Ch_b(\cE) \to \cE$ and 
natural equivalences $c^{\bA}:C^{\bE}\Tot\isoto\Tot C^{\bA}$ 
and $c^{\bB}:C^{\bE}\Tot\isoto\Tot C^{\bB}$ 
such that 
both $(\Tot,c^{\bA}):\bA \to \bE$ and 
$(\Tot,c^{\bB}):\bB \to \bE$ 
are complicial functors 
and $\Tot j_{\cE}=\id_{\cE}$. 
It is unique up to the unique complicial natural equivalence. 
(For more precise statement, please see Ibid). 
We call it the {\bf total functor} on 
$\Ch_b(\cE)$. 
\end{nt}

\begin{df}[\bf Total quasi-weak equivalences]
\label{df:tot qis}
We put $tw:=\Tot^{-1}w$ and call it 
the {\bf class of total quasi-weak equivalences}. 
We call a morphism in $tw$ a {total quasi-weak equivalence}. 
We put $\Acy^{tw}_b(\cE):=\Ch_b(\cE)^{tw}$. 
Since $\Tot$ is a complicial functor, 
if $w$ is thick, then $tw$ is also thick. 
\end{df}

\sn 
The hinge in this section is the following theorem. 

\begin{thm}
\label{thm:tw=qis+lw}
$\mathrm{(1)}$ 
{\bf (tw=lw$\vee$qis).} 
$\Acy_b^{tw}\cE={({(\Acy_b^{lw}\cE)}_{\qis})}_{\nul,\bB_{\frob}}=
{({(\Acy_b\cE)}_{lw})}_{\nul,\bA}$.\\
$\mathrm{(2)}$ 
In particular if 
the image of $\Acy_b\cE$ and $\Acy_b^{lw}\cE$ by the projection 
$\bA \to \pi_0(\bA)$ {\rm (}resp. $\bB \to \pi_0(\bB)${\rm )} is 
factorizable in a suitable order, 
then we have $\Acy_b^{tw}\cE={(\Acy_b\cE)}_{lw}$ 
{\rm (}resp. $\Acy_b^{tw}\cE={(\Acy_b^{lw}\cE)}_{\qis}${\rm )}.\\
$\mathrm{(3)}$ 
$qw=tw_{\thi,\bB_{\frob}}$. 
In particular if $w$ is thick, $qw=tw$. 
\end{thm}

\sn 
To prove the theorem, 
we need the several lemmata~\ref{prop:qis in tw}, 
\ref{lem:solid check} and \ref{prop:tw=qw} below. 

\begin{lem}
\label{prop:qis in tw}
Let $(\cE,w)$ be a bicomplicial pair. 
Then $\qis \subset tw$. 
\end{lem}

\begin{proof}[\bf Proof] 
We need only check that for any acyclic complex $x$ in $\Ch_b(\cE)$, 
$\Tot x$ is in $\cE^w$. 
The proof is carried out in several steps.
\begin{para}[Step 1] 
Let $x$ be a complex in $\Ch_b(\cE)$ and let us assume that 
$x$ is acyclic. 
Then there exists a strictly acyclic complex $y$ and a 
$C^{\bB}$-homotopy equivalence $f:x \to y$. 
Since $\Tot$ preserves $C$-homotopy equivalences and 
$\cE^w$ is closed under $C^{\bE}$-homotopy equivalences, 
$\Tot y\in\cE^w$ implies $\Tot x\in\cE^w$. 
We shall assume that $x$ is a strictly acyclic complex. 
\end{para}

\begin{para}[Step 2] 
Since $\Tot$ is an exact functor and $\cE^w$ is closed under extensions, 
we shall assume that the length of $x$ is $1$ by 
induction of the length of $x$ and 
by the admissible exact sequence
$$\tau_{\leq n}x \rinf x \rdef \tau_{\geq n+1}x .$$
\end{para}

\begin{para}[Step 3] 
Since $\Tot$ commutes with 
the suspension functors and $\cE^w$ is closed under 
the suspension functor, 
we shall assume that 
$x=[x_1 \onto{d_1} x_0]$ and $d_1$ is an isomorphism. 
In this case, we have the equality 
$\Tot x=\Cone d_1$ by the construction of $\Tot$ and it is in $\cE^w$. 
\end{para}
\end{proof}

\begin{lem}
\label{lem:solid check}
Let $f:x \to y$ be a morphism in $\Ch_b(\cE)$. 
Then\\
$\mathrm{(1)}$ 
If $f$ is in $lw$, then $\Cone_{\bB} f$ is in $\Acy^{lw}_b(\cE)_{\qis}$. 
In particular, $lw \subset w_{\Acy^{lw}_b(\cE)_{\qis}}$.\\
$\mathrm{(2)}$ 
If $f$ is in $\qis$, then $\Cone_{\bA} f$ is in $\Acy_b(\cE)_{lw}$. 
In particular, $\qis \subset w_{\Acy_b(\cE)_{lw}}$.
\end{lem}

\begin{proof}[\bf Proof]
Let $f:x \to y$ be a morphism in $\Ch_b(\cE)$. 
Assume that $f$ is in $lw$ (resp. $\qis$). 
Namely $\Cone_{\bA}f$ (resp. $\Cone_{\bB}f$) 
is in $\Acy^{lw}_b(\cE)$ (resp. $\Acy_b(\cE)$). 
Let us consider the push out diagram below
$${\footnotesize{\xymatrix{
x \ar@{>->}[r] \ar[d]_{f} & C^{\#}x \ar[d]\\
y \ar@{>->}[r] & \Cone^{\#}f  
}}}$$
where $\#=\bA$ (resp. $\#=\bB$). 
Then by \cite[p.406 step 1]{Kel90}, 
the morphism $[x \onto{f} y] \to [C^{\#}x \to \Cone^{\#}f]$ 
between the complexes in $\Ch_b(\Ch_b(\cE))$ is a quasi-isomorphism. 
Taking the totalized complex 
$\Tot^{\#}:\Ch_b(\Ch_b(\cE)) \to \Ch_b(\cE)$, 
it turns out that 
$\Cone^{\#}f=\Tot^{\#}[x \onto{f} y]$ is 
connected with  
the complex $\Tot^{\#}[C^{\ast}x \to \Cone^{\ast}{\cE}f]$ in 
$\Acy_b^{lw}(\cE)$ (resp. $\Acy_b\cE$) by 
the morphisms in $\qis$ (resp. $lw$) 
where $\#=\bB$ and $\ast=\bA$ (resp. $\#=\bA$ and $\ast=\bB$). 
Hence we complete the proof. 
\end{proof}

\begin{lem}
\label{prop:tw=qw}
For any complex $x$ in $\Ch_b(\cE)$, 
there is a zig-zag sequence of 
quasi-isomorphisms and level weak equivalences 
connecting it to a degree shift of 
$j_{\cE}(\Tot x)$. 
\end{lem}

\begin{proof}[\bf Proof]
Let $x=[\cdots \to 0 \to x_n \to x_{n-1} \to \cdots]$ 
be a complex in $\Ch_b(\cE)$. 
Then we have the following morphisms of complexes. 
$${\footnotesize{\xymatrix{
x_n \ar[r] \ar[d]_{d_n} & C^{\cE}(x_n) \ar[d] & 0 \ar[l]\ar[d]\\ 
x_{n-1} \ar[r] \ar[d]_{d_{n-1}} & \Cone^{\cE} d_{n-1} \ar[d] & 
\Cone^{\cE} d_{n-1} \ar[d] \ar[l]_{\id}\\
x_{n-2} \ar[d]_{d_{n-2}} \ar[r]^{\id} & x_{n-2} \ar[d]_{d_{n-2}} & 
x_{n-2} \ar[l]_{\id} \ar[d]^{d_{n-2}}\\ 
\vdots & \vdots & \vdots
}}}$$
where the left morphism is a quasi-isomorphism 
by the proof of \ref{lem:solid check} 
and obviously the right morphism is in $lw$. 
Now by induction of the length of $x$, 
we give the algorithm of connecting $x$ to a dgree shift of $j_{\cE}(\Tot x)$ 
by a zig-zag sequence of 
quasi-isomorphisms and level-weak equivalences. 
\end{proof}

\begin{proof}[\bf Proof of Theorem~\ref{thm:tw=qis+lw}] 
To prove assertion $\mathrm{(1)}$, 
we will show the following inclusions. 
$${(\Acy_b^{lw}\cE)}_{\qis} \underset{\textbf{I}}{\subset} 
{({(\Acy_b\cE)}_{lw})}_{\nul} \underset{\textbf{II}}{\subset} 
\Acy_b^{tw}\cE \underset{\textbf{III}}{\subset} 
{({(\Acy_b^{lw}\cE)}_{\qis})}_{\nul}
$$

\sn 
The proof for the inclusion \textbf{I}: 
Since the zero complex 
and any complex $x$ in $\Acy_b^{lw}\cE$ is connected by 
the canonical morphism $0 \to x$ in $lw$, 
we have $\Acy_b^{lw}\cE \subset {(\Acy_b\cE)}_{lw}$. 
Now the inclusion \textbf{I} follows from 
\ref{para:NC and comp weak equiv} $\mathrm{(10)}$ 
and \ref{lem:solid check} $\mathrm{(2)}$. 

\sn
The proof for the inclusion \textbf{II}: 
Since $\Tot(\Acy_b^{lw}\cE)\subset \cE^w$, we have $lw\subset tw$. 
Therefore the inclusion \textbf{II} follows from 
\ref{para:NC and comp weak equiv} $\mathrm{(10)}$ and 
\ref{prop:qis in tw}. 

\sn
The proof for the inclusion \textbf{III}: 
The inclusion \textbf{III} follows from 
\ref{para:NC and comp weak equiv} $\mathrm{(10)}$, 
\ref{lem:solid check} $\mathrm{(1)}$, 
\ref{prop:tw=qw} and 
the inclusion $\qis \subset w_{{\Acy_b^{lw}\cE}_{\qis}}$.

\sn
Assertion $\mathrm{(2)}$ follows 
from $\mathrm{(1)}$, \ref{cor:veethi} and 
\ref{ex:widelyexact} $\mathrm{(2)}$. 
Assertion $\mathrm{(3)}$ follows from $\mathrm{(1)}$, 
\ref{rem:acyqw} 
and \ref{ex:widelyexact} $\mathrm{(2)}$. 
\end{proof}

\begin{cor}[\bf Universal property of $\Ch_b$]
\label{cor:univ}
For any consisitent relative 
exact category $\bG=(\cG,w_{\cG})$ and 
any thick bicomplicial pair $\bC=(\cC,w_{\cC})$ 
and any relative exact functor $f:\bG\to \bC$, 
there exists a relative complicial functor 
$(\bar{f},c):\Ch_b(\bG) \to \bC$ 
such that $\bar{f}j_{\cG}=f$. 
$(\bar{f},c)$ is unique in the following sense. 
For another $(\bar{f'},c')$ such that 
$\bar{f'}j_{\cG}=f$, 
there exists a unique relative complicial natural equivalence 
$\theta:\bar{f}\isoto\bar{f'}$ such that $\theta j_{\cG}=\id_f$. 
\end{cor}

\begin{proof}[\bf Proof]
Since $\bC$ is thick, the functor 
$\Tot:\Ch_b(\cC) \to \cC$ is a relative exact functor 
$\Ch_b(\bC) \to \bC$ by \ref{thm:tw=qis+lw} $\mathrm{(3)}$. 
We put $\bar{f}:=\Tot \Ch_b(f)$. 
Then we can easily check that $\bar{f}j_{\cG}=f$. 
Uniqueness of $\bar{f}$ follows from \cite[4.9]{Uni}.
\end{proof}

\begin{lemdf}[\bf Solid axiom] 
\label{lemdf:solid axiom} 
For any strict relative exact category $\bF=(\cF,v)$, 
as in the proof of 
\ref{lemdf:adm}, we have the implications 
$\mathrm{(1)}\Rightarrow \mathrm{(2)} \Rightarrow \mathrm{(3)}$ for 
the conditions below.\\
$\mathrm{(1)}$ 
$lv \subset v_{\Acy^{lv}_b(\cF)_{\qis}}$.\\
$\mathrm{(2)}$ 
For any morphism $f:x \to y$ in $v$, 
there is a zig-zag sequence of quasi-isomorphisms 
connecting the complex 
$\Cone f=[x \onto{f} y]$ to a bounded complex in $\cF^v$.\\
$\mathrm{(3)}$ 
$lv \subset v_{{(\Acy^{lv}_b(\cF)_{\qis})}_{\nul}}$\\
We say that $v$ or $\bF$ satisfies the {\bf solid axiom} 
or $v$ or $\bF$ is solid 
if $v$ satisfies condition $\mathrm{(2)}$ above. 
We can easily check that the solid axiom implies the consistent axiom. 
We denote the full $2$-subcategory of 
solid relative exact categories 
(resp. solid Waldhausen exact categories) in $\RelEx^{\#}$ by 
$\RelEx_{\SOL}^{\#}$ (resp. $\WalEx_{\SOL}^{\#}$) 
for $\#\in\{+,\text{nothing}\}$. 
\qed
\end{lemdf}

\begin{cor}
\label{prop:bicomp is adm}
For any bicomplicial pair $(\cF,v)$, 
$v$ satisfies the solid axiom. 
In particular, 
for any strict relative exact category 
$\bG=(\cG,u)$, 
$\Ch_b(\bG)=(\Ch_b(\cG),qu)$ is a solid Waldhausen exact category. 
\qed
\end{cor}

\begin{cor}
\label{cor:comp of cone}
For any morphism $f:x \to y$ in $\Ch_b(\cE)$, 
$\Cone^{\bA}f$ is canonically quasi-weak equivalent to 
$\Cone^{\bB}f$. 
\end{cor}

\begin{proof}[\bf Proof]
The canonical morphism $[0 \to \Cone^{\bA}f] \to 
[C^{\bA}x \to \Cone^{\bA}f]$ in $\Ch_b(\Ch_b(\cE))$ is in $llw$. 
Therefore by taking the total functor, 
we have the zig-zag sequence of 
the quasi-isomorphism and the level weak equivalence 
$$\Cone^{\bB} f \to \Tot[C^{\bA} x \to \Cone^{\bA} f] \leftarrow
\Cone^{\bA} f.$$
Since a level weak equivalence is 
a quasi weak equivalence by \ref{thm:tw=qis+lw} and 
\ref{lem:solid check} $\mathrm{(1)}$, 
we get the desired result. 
\end{proof}

\begin{cor}
\label{cor:comp of derived cat}
$\mathrm{(1)}$ 
Let us assume that $\bE$ is thick. 
Then the canonical functor 
$\bE \onto{j_{\cE}} \Ch_b(\bE)$ 
induces an equivalence of triangulated categories 
$$\cT(\cE,w)\isoto \cT(\Ch_b(\cE),qw)=\calD(\cE,w).$$
$\mathrm{(2)}$ 
{\bf (Derived Gillet-Waldhausen theorem).} 
Let $\bF=(\cF,v)$ be a strict relative exact category, 
then the canonical inclusion functor 
$j_{\cF}:\bF \to \Ch_b(\bF)$ is a derived equivalence.
\end{cor}

\begin{proof}[\bf Proof] 
$\mathrm{(1)}$ 
First, we will prove that 
$\Tot:\Ch_b(\cE) \to \cE$ 
induces 
an equivalence of triangulated categories
$$\cT(\Tot):\cT(\Ch_b(\cE),tw)\isoto \cT(\cE,w).$$
Since $\pi_0(\cE^w)$ is thick by 
\ref{para:NC and comp weak equiv} $\mathrm{(2)}$, 
we have the equality
$$\pi_0(\Acy_b^{tw}(\cE))=\Ker(\pi_0(\Ch_b(\cE)) \onto{\pi_0(\Tot)} 
\pi_0(\cE) \to \cT(\cE,w)).$$
Therefore for any object $x$ in $\cT(\Ch_b(\cE),tw)$, 
$\cT(\Tot)(x)\isoto 0$ implies 
$x\isoto 0$ in $\cT(\Ch_b(\cE),tw)$. 
Since we have $\Tot j_{\cE}=\id$, 
$\cT(\Tot)$ is essentially surjective. 
Moreover by \ref{thm:tw=qis+lw} and \ref{prop:tw=qw}, 
for any complexes $x$ and $y$ in $\Ch_b(\cE)$, 
there exist isomorphisms 
$a:j_{\cE}(\Tot x)\isoto x$ and $y\isoto j_{\cE}(\Tot y)$ in 
$\cT(\Ch_b(\cE),tw)$. 
Now let us consider the commutative diagram below. 
$${\footnotesize{
\xymatrix{
\Hom_{\cT(\Ch_b(\cE);tw)}(x,y) 
\ar[r]^{\cT(\Tot)} \ar[d]^{\wr}_{\Hom(a,b)} & 
\Hom_{\cT(\cE;w)}(\Tot x,\Tot y) 
\ar[d]_{\wr}^{\Hom(\cT(\Tot)(a),\cT(\Tot)(b))}\\
\Hom_{\cT(\Ch_b(\cE);tw)}(j_{\cE}(x),j_{\cE}(y)) 
\ar[r]_{\cT(\Tot)} & 
\Hom_{\cT(\cE;w)}(\Tot x,\Tot y) .
}}}$$
Since the bottom $\cT(\Tot)$ has the section $j_{\cE}$, 
it is surjective and therefore the top $\cT(\Tot)$ is also. 
Hence $\cT(\Tot)$ is full. 
Now utilizing \cite[3.18]{Bal07}, 
it turns out that $\cT(\Tot)$ is an equivalence of triangulated categories. 
Since $\Tot j_{\cE}=\id$, $j_{\cE}$ is the inverse functor 
of the equivalence above. 
By \ref{cor:inv by thick closure} 
and \ref{thm:tw=qis+lw} $\mathrm{(3)}$, 
the identity functor $(\Ch_b(\cE),tw) \to (\Ch_b(\cE),qw)$ 
induces an equivalence of triangulated categories 
$\cT(\Ch_b(\cE),tw)\isoto \cT(\Ch_b(\cE),qw)$. 
Hence we obtain the result. 

\sn
$\mathrm{(2)}$ 
By \ref{prop:compness} and \ref{prop:bicomp is adm}, 
$\Ch_b(\bF)$ is a solid thick bicomplicial pair. 
Then by applying assertion $\mathrm{(1)}$ 
to $\bE=\Ch_b(\bF)$, 
we obtain an equivalence of triangulated categories 
$\calD_b(\bF)=\cT(\Ch_b(\cF),qu)\isoto\calD_b(\Ch_b(\bF))$. 
\end{proof}

\sn
Recall the terminologies in relative category theory 
from Conventions $\mathrm{(6)}$.

\begin{cor}
\label{cor:homotopy inv of homotopy theories}
Let $U$ be the forgetful $2$-functor from 
$\biCompPair_{\thi}^{+}$ to $\RelEx_{\consist}^{+}$. 
Then the pair $(\RelEx_{\consist}^{+},\deq) 
\substack{\overset{\Ch_b}{\to}\\ \underset{U}{\leftarrow}} 
(\biCompPair_{\thi}^{+},\deq)$ 
are relative funcrors wirh 
adjunction morphisms 
$j;\id_{\RelEx_{\consist}^{+}}\to U\Ch_b$ and $\Ch_bU\to 
\id_{\biCompPair_{\thi}^{+}}$. 
Moreover $j$ and $\Tot$ are relative natural equivalences. 
In particular the homotopy theories of 
consistent relative exact categories 
and thick bicomplicial pairs are homotopy equivalent.
\qed
\end{cor}

\sn
Recall the definition of exact and weakly exact functors 
from \ref{df:exact seq in RelEx} 
and the definition of quotient of strictly relative exact categories 
from \ref{df:coker of rel exact cat}. 

\begin{cor}
\label{cor:weakly exact seq}
$\mathrm{(1)}$ 
The functor $\Ch_b:\RelEx_{\consist} \to \RelEx_{\consist}$ 
is exact and weakly exact.\\
$\mathrm{(2)}$ 
Let $\bG=(\cG,u)$ be a very strict solid relative exact category. 
Then\\
$\mathrm{(i)}$ 
The inclusion functor 
$\Ch_b(\cG^u) \rinc \Ch_b(\cG)^{qu}$ 
induces an equivalence of triangulated categories 
$\calD_b(\cG^u)\isoto\cT(\Ch_b(\cG)^{qu},\qis)$.\\
$\mathrm{(ii)}$ 
The inclusion functor $\cG^u\rinc \cG$ and the identity functor of $\cG$ 
induce an exact sequence
$$(\cG^u,i_{\cG^u})\to (\cG,i_{\cG})\to (\cG,u).$$
$\mathrm{(3)}$ 
Let $f:\bF=(\cF,v) \to \bG=(\cG,u)$ 
be a derived fully faithful relative exact functor 
from a strict relative exact category $\bF$ to 
a consistent relative exact category $\bG$. 
Then the sequence $\bF \onto{f}\bG \onto{\pi_f} \bG/\bF$ 
is weakly exact.\\
$\mathrm{(4)}$ 
Let 
$${\footnotesize{\xymatrix{
\bF \ar[r]^f \ar[d]_a & \bG=(\cG,u) \ar[d]^b\\
\bF' \ar[r]_{f'} & \bG'=(\cG',u')
}}}$$
be a commutative diagram of strict relative exact categories. 
if both $a$ and $b$ are derived equivalences, 
then $\coker(a,b)$ is also a derived equivalence.
\end{cor}

\begin{proof}[\bf Proof]
$\mathrm{(1)}$ 
Let 
$\bG=(\cG,w_{\cG}) \onto{a} \bH=(\calH,w_{\calH}) \onto{b} \bI=(\cI,w_{\cI})$ 
be a sequence of consistent relative exact categories. 
Consider the commutative diagram of triangulated categories 
$${\footnotesize{\xymatrix{
\calD_b(\bG) \ar[r]^{\calD_b(a)} \ar[d]_{\calD_b(j_{\cG})} & 
\calD_b(\bH) \ar[r]^{\calD_b(b)} \ar[d]^{\calD_b(j_{\calH})} & \calD_b(\bI) 
\ar[d]^{\calD_b(j_{\cI})} \\
\calD_b(\Ch_b(\bG)) \ar[r]_{\calD_b(\Ch_b(a))} & 
\calD_b(\Ch_b(\bH)) \ar[r]_{\calD_b(\Ch_b(b))} & 
\calD_b(\Ch_b(\bI)).
}}}$$
Here the vertical morphisms are equivalences of triangulated categories 
by \ref{cor:comp of derived cat} $\mathrm{(2)}$. 
Hence if the top line is exact (resp. weakly exact), 
then the bottom line is also exact (resp. weakly exact).

\sn
$\mathrm{(2)}$ 
Let us consider the commutative diagram
\begin{equation}
\label{equ:very strict solid}
{\footnotesize{\xymatrix{
\calD_b(\cG^u) \ar[r]^{\textbf{I}} \ar[d]_{\textbf{III}} & 
\calD_b(\cG) \ar[r] \ar@{=}[d] & \calD_b(\bG)\ar@{=}[d]\\
\cT(\Ch_b(\cG)^{qu},\qis) \ar[r]_{\textbf{II}} & 
\cT(\Ch_b(\cG),\qis) \ar[r] & 
\cT(\Ch_b(\cG),qu)
}}}
\end{equation}
The functors \textbf{I} and \textbf{II} are fully faithful 
by the assumption and \ref{prop:compness} $\mathrm{(2)}$ respectively. 
Therefore the functor \textbf{III} is also fully faithful. 
Since $\bG$ is solid, 
we have an equality $\Ch_b(\cG)^{qu}=\Ch_b(\cG^u)_{\qis}$. 
Therefore the functor \textbf{III} is essentially surjective. 
Hence we complete the proof of $\mathrm{(i)}$. 
Now the exactness of the bottom line 
in the diagram (\ref{equ:very strict solid}) implies 
the exactness of the top line. 
We obtain the proof of $\mathrm{(ii)}$.

\sn
$\mathrm{(3)}$ 
Consider the following commutaive diagram of triangulated categories. 
$${\footnotesize{\xymatrix{
\calD_b(\bF) \ar[r]^{\calD_b(f)} \ar@{=}[d] & 
\calD_b(\bG) \ar[r]^{\cT(\id_{\Ch_b(\cG)})} \ar@{=}[d] & 
\cT(\Ch_b(\cG),w_f) \ar[d]^{\cT(j_{\Ch_b(\cG)})}\\
\calD_b(\bF) \ar[r]_{\calD_b(f)} & 
\calD_b(\bG) \ar[r]_{\calD_b(\pi_f)} & 
\calD_b(\bG/\bF).
}}}$$
Here the top line is weakly exact and 
the right vertical line is an equivalences of triangulated categories 
by \ref{cor:comp of derived cat} $\mathrm{(2)}$. 
Hence we obtaion the result.

\sn
$\mathrm{(4)}$ 
Consider the following commutative diagram
$${\footnotesize{\xymatrix{
\coker \calD_b(f) 
\ar[r]_{\!\!\!\!\!\!\!\!\sim}^{\!\!\!\!\!\!\!\!\!\!{\textbf{I}}} 
\ar[d]_{\textbf{II}} & 
\cT(\Ch_b(\cG),w_f) \ar[r]^{\ \ \ {\textbf{III}}}_{\ \ \ \sim} 
\ar[d]^{\cT(\Ch_b(b))} &
\calD_b(\bG/\bF) \ar[d]^{\calD_b(\coker(a,b))}\\
\coker \calD_b(f') 
\ar[r]^{\!\!\!\!\!\!\!\!\sim}_{\!\!\!\!\!\!\!\!\!\!{\textbf{I}}} & 
\cT(\Ch_b(\cG'),w_{f'}) \ar[r]_{\ \ \ {\textbf{III}}}^{\ \ \ \sim} & 
\calD_b(\bG'/\bF').
}}}$$
Here the morphisms \textbf{I}, \textbf{II} and \textbf{III} 
are equivalences of triangulated categories 
by \ref{rem:coker of rel exact cat} $\mathrm{(1)}$, assumption and 
\ref{cor:comp of derived cat} $\mathrm{(2)}$ respectively. 
Hence the relativ exact functor $\coker(a,b)$ is a derived equivalence.
\end{proof}

\begin{para}[\bf Proof of Corollary~\ref{cor:localization}]
\label{proof:locinv}
For any weakly exact sequence of consistent relative exact categories 
$\bE \onto{u} \bF \onto{v} \bG$, the induced sequence 
$\Ch_b(\bE) \onto{\Ch_b(u)} \Ch_b(\bF) \onto{\Ch_b(v)} \Ch_b(\bG)$ 
is a weakly exact sequence of bicomplicial pairs and 
complicial exact functors 
by \ref{prop:compness} $\mathrm{(1)}$ and 
\ref{cor:weakly exact seq} $\mathrm{(1)}$. 
Therefore the assertion follows from \cite[3.2.27]{Sch06}. 
\qed
\end{para}

\section{Resolution theorems}
\label{sec:resol thm}

In this section, we review the (strongly) resolution conditions 
in \ref{df:resol cond} 
and introduce the resolution theorems in \ref{thm:resol thms}. 
Recall the definition of multiplicative systems from Conventions $\mathrm{(5)}$ $\mathrm{(xiv)}$.

\begin{para}
\label{para:multi closed}
For a pair of a category $\cC$ and a multiplicative system $v$ of $\cC$, 
we define the simplicial subcategory $\cC(-,v)$ in $[m] \mapsto \cC^{[m]}$ 
$$[m]\mapsto \cC(m,v)$$
where $\cC(m,v)$ is the full subcategory of $\cC^{[m]}$ 
consisting of those functors 
which take values in $v$. 
\end{para}

\begin{df}[\bf Resolution conditions]
\label{df:resol cond}
$\mathrm{(1)}$ 
We say that 
the inclusion functor of Quillen exact categories 
$\cE \rinc \cF$ 
satisfies {\bf the resolution conditions} 
if it satisfies the following three conditions.\\
{\bf (Res 1)} $\cE$ is closed under extensions in $\cF$.\\
{\bf (Res 2)} For any object $x$ in $\cF$, 
there are an object $y$ in $\cE$ and an admissible epimorphism $y \rdef x$.\\
{\bf (Res 3)} For any admissible short exact sequences 
$x\rinf y \rdef z$ in $\cF$, 
if $y$ is in $\cE$, then $x$ is also in $\cE$.\\
$\mathrm{(2)}$ 
(cf. \cite[1.12]{Moc11}). 
Moreover assume that 
there exists a class of morphisms $v$ in $\cF$ 
such that the pair $(\cF,v)$ is a Waldhausen exact category. 
Let us put $w=\cE\cap v$. 
We say that the inclusion functor 
$(\cE,w)\rinc (\cF,v)$ 
satisfies {\bf the strongly resolution conditions} 
if for any non-negative integer $m$, 
the inclusion functor $\cE(m,w)\rinc \cF(m,v)$ satisfies 
the resolution conditions.
\end{df}

\begin{thm}
\label{thm:resol thms}
Let $i:\bE=(\cE,w) \rinc \bF=(\cF,v)$ be 
an inclusion functor between 
strict relative exact categories.\\
$\mathrm{(1)}$ 
{\bf (Derived resolution theorem).}
Let us assume that both the inclusion functors 
$\cE \rinc \cF$ and $\cE^w \rinc \cF^v$ satisfy the resolution conditions. 
Then $i$ is a derived equivalence.\\
$\mathrm{(2)}$ 
{\bf (Resolution theorem).}
Let us assume that the following conditions hold.\\
$\mathrm{(i)}$ 
$v\cap \cE=w$.\\ 
$\mathrm{(ii)}$ 
Both $\bE$ and $\bF$ are Waldhausen exact categories.\\
$\mathrm{(iii)}$ 
$(\cE,w) \rinc (\cF,v)$ satisfies the strongly resolution conditions.\\
Then $i$ is a $K^W$-equivalence. 
\end{thm}

\begin{proof}[\bf Proof] 
$\mathrm{(1)}$ 
If the classes $v$ and $w$ are the class of all isomorphisms, 
then the assertion is proven in \cite[3.3.8]{Sch11}. 
Therefore for general cases, 
by the assumption, 
the inclusion functors $\cE \rinc \cF$ and $\cE^w \rinc \cF^v$ 
induce equivalences of triangulated categories
$$\calD_b(\cE)\isoto \calD_b(\cF),\ \ \calD_b(\cE^w)\isoto\calD_b(\cF^v).$$ 
Hence the inclusion functor $(\cE,w) \rinc (\cF,v)$ induces an equivalence of 
triangulated categories
$$\calD_b(\cE,w)\isoto\calD_b(\cF,v).$$
Assertion $\mathrm{(2)}$ is proven in \cite[\S 1]{Moc11}.
\end{proof}

\begin{cor}
\label{cor:res cond and consisitent}
Let $i:\bE=(\cE,w) \rinc \bF=(\cF,v)$ 
be an inclusion functor between 
strict relative exact categories. 
Assume that the following conditions hold.\\
$\mathrm{(1)}$ 
$\cE^w \rinc \cF^v$ satisfies the resolution conditions.\\
$\mathrm{(2)}$ 
$v$ satisfies the solid axiom.\\ 
Then $w$ also satisfies the solid axiom. 
\end{cor}

\begin{proof}[\bf Proof]
Let $f:x\to y$ be a morphism in $w$. 
Then by assumption $\mathrm{(2)}$, 
the complex $\Cone f=[x \onto{f} y]$ is quasi isomorphic to a complex 
in $\Ch_b(\cF^v)$. 
Now by the assumption $\mathrm{(1)}$, 
$(\cE^w,i_{\cE^w}) \to (\cF^v,i_{\cF^v})$ is a derived equivalence 
by \ref{thm:resol thms} $\mathrm{(1)}$. 
Therefore the complex $\Cone f$ is quasi-isomorphic to a complex 
in $\Ch_b(\cE^w)$. 
\end{proof}

\section{Quasi-split exact sequences}
\label{sec:quasi-split}

\sn
In this section, 
we will prepare the terminologies about 
quasi-split exact sequences and flags of 
triangulated categories or particular relative exact categories. 
We start by recollecting a profitable lemma to manage 
adjoint functors. 
Recall the definition of triangle adjoint from 
Convention $\mathrm{(8)}$ $\mathrm{(vii)}$. 

\begin{lem}
\label{lem:adjoint functor}
$\mathrm{(1)}$ 
Let $f:\cX \to \cY$ be a functor between 
categories. 
Assume that $f$ admits a right {\rm (}resp. left{\rm )} 
adjoint functor $g:\cY \to \cX$ with 
adjunction maps $fg \onto{A} \id_{\cY}$ and $\id_{\cX} \onto{B} gf$ 
{\rm (}resp. $id_{\cY} \onto{A} fg $ and $gf \onto{B} \id_{\cX}${\rm )} 
and assume that $B$ is a natural equivalence. 
Then $f$ is fully faithful.\\
$\mathrm{(2)}$ 
Let $(\cT,\Sigma)$ and $(\cT',\Sigma')$ be 
triangulated categories, 
$(g,\rho):\cT' \to \cT$ 
a triangle functor, 
$f$ a left {\rm (}resp. right{\rm )} adjoint of $g$ 
with $\Phi:fg \to \id_{\cT}$ and $\Psi:\id_{\cT'}\to gf$ 
{\rm (}resp. 
$\Phi:\id_{\cT} \to fg$ and $\Psi:gf\to \id_{\cT'}$
{\rm )} 
adjunction morphisms 
and $\lambda:=(\Phi\Sigma' L)(L\rho^{-1} L)(L\Sigma \Psi)$ 
{\rm (}resp. 
$\lambda:={\{(f\Sigma \Psi)(f\rho f)(\Phi\Sigma' f)\}}^{-1}$
{\rm )}. 
Then $(L,\lambda)$ is a triangle functor and 
it is a left {\rm (}resp. 
right
{\rm )} 
triangle adjoint of $(R,\rho)$. 
\end{lem}

\begin{proof}[\bf Proof]
Assertion $\mathrm{(2)}$ for left adjoint case 
is mentioned in \cite[8.3]{Kel96}. 
We will prove assertion $\mathrm{(1)}$. 
Since $f$ has a left quasi-inverse functor $g$, 
$f$ is faithful. 
We will prove that $f$ is full. 
For any morphism $a:f(x) \to f(y)$ in $\cY$, 
we put $b:={(By)}^{-1}ga Bx$ (resp. $b:=Byga{(Bx)}^{-1}$). 
Then we have the equality 
$$fb={(fBy)}^{-1}fga fBx={(fBy)}^{-1}{(Afy)}^{-1}aAfxfBx=a$$ 
$$\text{(resp. }fByfga{(fBx)}^{-1}=fByAfya{(Afx)}^{-1}{(fBx)}^{-1}=a \text{)}.$$ 
Hence $f$ is full. 
\end{proof}

\begin{df}[\bf Relative exact adjoint functors]
\label{df:relative exact adj}
Let 
$f:\bE=(\cE,w) \to \bE'=(\cE',w')$ 
and $g:\bE' \to \bE$ 
be relative exact functors between relative exact categories 
and $A:\id_{\cE} \to gf$ and $B:fg\to \id_{\cE'}$ natural transformations 
such that 
$(Bf)(fA)=\id_f$ and $(gB)(Ag)=\id_g$. 
Then we say that 
$f$ (resp. $g$) 
is a {\bf left} (resp. {\bf right}) 
relative exact adjoint functor of $g$ (resp. $f$).
\end{df}

\begin{lemdf}[\bf Quasi-split exact sequences]
\label{lemdf:quasi-split exact seq}
$\mathrm{(1)}$ 
Let 
$$\displaystyle{\bE=(\cE,w)
\substack{\overset{i}{\to}\\ \underset{q}{\leftarrow}}
\bF=(\cF,v)
\substack{\overset{p}{\to}\\ \underset{j}{\leftarrow}}
\bG=(\cG,u)}$$
be relative exact functors between relative exact categories 
and 
$A:iq\to \id_{\cF}$, $B:\id_{\cE}\to qi$, 
$C:\id_{\cF}\to jp$ and $D:pj\to \id_{\cG}$ are 
natural transformations such that 
$(Ai)(iB)=\id_i$, $(qA)(Bq)=\id_q$, 
$(jD)(Cj)=\id_j$ and $(Dp)(pC)=\id_p$. 
Suppose that 
a sequence 
$iq \onto{A}\id_{\cF} \onto{C} jp$ is 
admissible exact. Then\\
$\mathrm{(i)}$ 
The natural transformations 
$Ai$, $iB$, $Cj$ and $jD$ are natural equivalences.\\
$\mathrm{(ii)}$ 
The bifunctor $\Hom_{\cF}(i(-),j(-))$ 
from $\cE^{\op}\times \cG$ to the category of abelian groups 
is trivial.\\
$\mathrm{(iii)}$ 
The functors $pi$ and $qj$ are trivial.\\
$\mathrm{(iv)}$ 
The following conditions are equivalent.\\
$\mathrm{(a)}$ 
$i$ {\rm (}resp. $j${\rm )} is fully faithful.\\
$\mathrm{(b)}$ 
$i$ {\rm (}resp. $j${\rm )} is conservative.\\
$\mathrm{(c)}$ 
$B$ {\rm (}resp. $D${\rm )} is a natural equivalence.\\
$\mathrm{(v)}$ 
If the conditions in $\mathrm{(iv)}$ is verified, 
then the functor $i$ {\rm (}resp. $j${\rm )} 
reflects exactness and 
the functor 
$pi$ {\rm (}resp. $qj${\rm )} is trivial.\\
If the equivalent conditions in $\mathrm{(iv)}$ hold, 
namely both $i$ and $j$ are fully faithful and both 
$B$ and $D$ are natural equivalences, 
then we call the sequence 
$\bE \onto{i} \bF \onto{p} \bG$ 
(resp. $\bG\onto{j}\bF\onto{q}\bE$) 
a {\bf right} (resp. {\bf left}) {\bf quasi-split exact sequence}. 
We say that a system 
$(j,q,A,B,C,D)$ (resp. $(i,p,C,D,A,B)$) or shortyl 
$(j,q)$ (resp. $(i,p)$) 
is a {\bf right} (resp. {\bf left}) {\bf quasi-splitting} of 
a sequence $(i,p)$ (resp. $(j,q)$).\\
$\mathrm{(2)}$ 
Let $(\cT,\Sigma) \onto{i} (\cT',\Sigma')
\onto{p} (\cT'',\Sigma'')$ be a sequence of 
triangulated categories. 
We say that a sequence $(i,p)$ is 
{\bf right} (resp. {\bf left}) {\bf quasi-split} 
if both $i$ and $p$ admit 
right (resp. left) adjoint functors 
$q:\cT'\to \cT$ and $j:\cT'' \to \cT'$ with 
adjunction maps 
$iq \onto{A} \id_{\cT'}$, $\id_{\cT}\onto{B}qi$, 
$\id_{\cT'}\onto{C} jp$ and $pj\onto{D}\id_{\cT''}$ 
(resp. 
$\id_{\cT'} \onto{A} iq$, $qi\onto{B}\id_{\cT}$, 
$jp\onto{C}\id_{\cT'}$ and $\id_{\cT''}\onto{D}pj$
) 
respectively 
such that $B$ and $D$ are natural equivalences and 
if there exists a triangle natural transformation 
$jp\onto{E}\Sigma' iq$ 
(resp. $iq\onto{E}\Sigma' jp$) 
such that a triangle 
$(A,C,E)$ 
(resp. $(C,A,E)$) 
is a $\Sigma'$-exact triangle. 
We call a system $(j,q,A,B,C,D,E)$ (resp. $(j,q,C,D,A,B,E)$) 
or shortly $(j,q,E)$ 
a {\bf right} (resp. {\bf left}) {\bf splitting} of 
a sequence $(i,p)$. 
Then\\
$\mathrm{(i)}$ 
The functors $i$ and $j$ are fully faithful.\\
$\mathrm{(ii)}$ 
The natural transformations $Ai$ and $Cj$ are natural equivalences.\\
$\mathrm{(iii)}$ 
The functors $pi$ and $qj$ are trivial.\\
$\mathrm{(iv)}$ 
The bifunctor $\Hom_{\cT'}(i(-),j(-))$ 
(resp. $\Hom_{\cT'}(j(-),i(-))$) from 
$\cT^{\op}\times\cT''$ (resp. ${\cT''}^{\op}\times\cT$) 
to the category of abelian groups is trivial.\\
$\mathrm{(3)}$ 
A sequence of strict relative exact categories 
$\bE\onto{i} \bF\onto{p} \bG$ 
is a {\bf right} (resp. {\bf left}) {\bf quasi-split exact sequence} 
if the induced sequence 
$\calD_b(\bE)\onto{\calD_b(i)}\calD_b(\bF)\onto{\calD_b(p)}\calD_b(\bG)$ 
is a right (resp. left) quasi-split exact sequence of 
triangulated categories.
\end{lemdf}

\begin{proof}[\bf Proof]
$\mathrm{(1)}$ 
$\mathrm{(i)}$ 
By exactness $iq\overset{A}{\rinf}\id_{\cF} \overset{C}{\rdef} jp$, 
$Aix$ (resp. $Cjx$) is a monomorphism (resp. an epimorphism) 
for any object $x$ in $\cE$ (resp. $\cG$). 
On the other hand, 
$Aix$ (resp. $Cjx$) has a right (resp. left) inverse 
$iBx$ (resp. $jDx$). 
Therefore $Aix$ (resp. $Cjx$) and its right (resp. left) 
inverser $iBx$ (resp. $jDx$) are isomorphisms. 

\sn
$\mathrm{(ii)}$ 
In the exact sequence 
$iqi\overset{Ai}{\rinf}i \overset{Ci}{\rdef} jpi $ 
(resp. $iqj\overset{Aj}{\rinf}j \overset{Cj}{\rdef} jpj$), 
since $Ai$ (resp. $Cj$) is an isomorphism, 
we have $jpi\isoto 0$ (resp. $iqj\isoto 0$). 
Then, 
for any objects $x$ in $\cE$ and $z$ in $\cG$ and 
any morphism $a:ix \to jz$, 
we have equalities 
$$a={(Cjz)}^{-1}(Cjz)a={(Cjz)}^{-1}(jpa)(Cix)=0.$$
$${\footnotesize{\xymatrix{
ix \ar[r]^{Cix} \ar[d]_a & jpix \ar[r]^{\sim} \ar[d]^{jpa}& 0\\
jz \ar[r]^{\!\!\!\sim}_{Cjz} & jpjz.
}}}$$

\sn
$\mathrm{(iii)}$ 
By assumption, we have the isomorphisms of bifunctors
$$\Homo_{\cT'}(i(-),j(-))\isoto\Hom_{\cT''}(-,qj(-))\isoto\Hom_{\cT}(pi(-),-)$$
and they are trivial by $\mathrm{(ii)}$. 
Hence we obtain the result.

\sn
$\mathrm{(iv)}$ 
Implications $\mathrm{(a)} \Rightarrow \mathrm{(b)} \Rightarrow \mathrm{(c)}$ 
are straightfoward. 
By \ref{lem:adjoint functor} $\mathrm{(1)}$, 
assertion $\mathrm{(c)}$ implies assertion $\mathrm{(a)}$. 

\sn
$\mathrm{(v)}$ 
Let $x \onto{a} y \onto{b} z$ be 
a sequence in $\cE$ and assume that 
$ix \overset{ia}{\rinf} iy \overset{ib}{\rdef} iz$ 
is an admissible exact sequence in $\cF$. 
Consider the commutative diagram below.
$${\footnotesize{\xymatrix{
x \ar[r]^a \ar[d]_{Bx}^{\wr} & 
y \ar[r]^b \ar[d]_{By}^{\wr} & 
z \ar[d]^{Bz}_{\wr}\\
qix \ar@{>->}[r]_{qia} & 
qiy \ar@{->>}[r]_{qib} & 
qiz. 
}}}$$
Here the bottom line is an admissible exact sequence and 
the vertical lines are isomorphisms. 
Hence a sequence $(a,b)$ is an admissible exact sequence in $\cE$. 
The other assertions are trivial.\\
We can similarly prove the assertions in $\mathrm{(2)}$. 
\end{proof}

\begin{rem}
\label{rem:quasi-split}
For a right (resp. left) 
quasi-split exact sequence of 
triangulated categroies, 
we can take a right (resp. left) quasi-splitting 
of it as a system of 
triangulated adjoint functors and 
triangle natural transformations 
by virtue of \ref{lem:adjoint functor} $\mathrm{(2)}$. 
\end{rem}

\sn
Recall the definition of the category of admissible exact sequences 
in an exact category from 
Conventions $\mathrm{(7)}$ $\mathrm{(xiv)}$ and 
\ref{lemdf:extensional, hom rel ex cat} $\mathrm{(1)}$.

\begin{cor}
\label{cor:str of qsexseq}
Let $\bE=(\cE,w) \onto{i} \bF=(\cF,v) \onto{p} \bG=(\cG,u)$ 
be a right quasi-split exact sequence of 
relative exact categories with 
a right quasi-splitting 
$(j,q,A,B,C,D)$. 
Then\\ 
$\mathrm{(1)}$ 
For any admissible exact sequence in $\cF$, 
$ix \overset{a}{\rinf} y \overset{b}{\rdef} jz$ 
with $x$ in $\cE$, $y$ in $\cF$ and $z$ in $\cG$, 
there are unique isomorphisms 
$\alpha:ix\isoto iqy$ and $\beta:jz\isoto jpy$ 
which make the diagram below commutative. 
$${\footnotesize{\xymatrix{
ix \ar@{>->}[r]^a \ar[d]_{\alpha}^{\wr} & 
y \ar@{->>}[r]^c \ar@{=}[d] & 
jz \ar[d]^{\beta}_{\wr}\\
iqy \ar@{>->}[r]_{Ay} & 
y \ar@{->>}[r]_{Cy} & 
jpy.
}}}$$
$\mathrm{(2)}$ 
We regard $\cE$ and $\cG$ as strict exact subcategories of 
$\cF$ by the exact functors $i$ and $j$ respectively. 
The functor $\Theta:\cF\to E(\cE,\cF,\cG)$ 
which sends an object $y$ in $\cF$ to 
an admissible exact sequence $iqy \overset{Ay}{\rinf} y \overset{Cy}{\rdef} 
jpy$ gives an equivalence of 
relative exact categories 
$\bF \to E(\bE,\bF,\bG)$ which makes the diagrams below commutative 
up to unique natural equivalences. 
$${\footnotesize{
\xymatrix{
\bE \ar[r]^i \ar@{=}[d] &
\bF \ar[r]^p \ar[d]_{\Theta} & 
\bG \ar@{=}[d]\\
\bE \ar[r]_{\!\!\!\!\In_{\cE}} &
E(\bE,\bF,\bG) \ar[r]_{\ \ \ \ q_{E(\cE,\cF,\cG)}} & 
\bG,
}\ \ \ 
\xymatrix{
\bG \ar[r]^j \ar@{=}[d] &
\bF \ar[r]^q \ar[d]_{\Theta} & 
\bE \ar@{=}[d]\\
\bG \ar[r]_{\!\!\!\!\In_{\cG}} &
E(\bE,\bF,\bG) \ar[r]_{\ \ \ \ s_{E(\cE,\cF,\cG)}} & 
\bE.
}}}$$
\end{cor}

\begin{proof}[\bf Proof]
$\mathrm{(1)}$ 
Since the compositions $Cy a$ and $cAy$ are trivial 
by \ref{lemdf:quasi-split exact seq} 
$\mathrm{(1)}$ $\mathrm{(iii)}$ 
we have the morphisms 
$\alpha:ix\to iqy$, $\beta:iz\to jpy$, 
$\alpha':iqy \to ix$ and $\beta':jpy \to jz$ 
such that $Ay\alpha=a$, $\beta c=Cy$, 
$a\alpha'=Ay$ and $\beta' Cy=c$. 
Since $a$ and $Ay$ are monomorphisms and $c$ and $Cy$ are epimorphisms, 
$\alpha$, $\alpha'$, $\beta$ and $\beta'$ are unique. 
By uniqueness, it turns out that 
$\alpha$ and $\alpha'$, $\beta$ and $\beta'$ are inverse morphisms in each others respectively. 

\sn
$\mathrm{(2)}$ 
The inverse functor of $\Theta$ is given by $m=m_{E(\cE,\cF,\cG)}$. 
$m\Theta=\id_{\cF}$ is trivial. 
The unique natural equivalence 
$\Omega:\Theta m \to \id_{E(\cE,\cF,\cG)}$ 
such that $m\Omega=\id_{\id_{\cF}}$ is given by assertion $\mathrm{(1)}$.
\end{proof}

\begin{cor}
\label{cor:quasi split is exact}
$\mathrm{(1)}$ 
A right {\rm(}resp. left{\rm )} quasi-split exact sequence 
of triangulated categories 
is exact.\\
$\mathrm{(2)}$ 
A derived right {\rm(}resp. left{\rm )} quasi-split exact 
sequence of strict relative exact categories is exact.
\end{cor}

\begin{proof}[\bf Proof]
$\mathrm{(1)}$ 
We will only give a proof for a right quasi-split case. 
Let $(\cT,\Sigma) \onto{i} (\cT',\Sigma')\onto{p} (\cT'',\Sigma'')$ 
be a right quasi-split exact sequence 
with a right quasi-splitting $(j,q,A,B,C,D,E)$. 
We write same letters $i$, $q$, $p$ and $j$ 
for the induced functors $\cT\substack{\rightarrow\\ \leftarrow}\Ker p$, 
$\cT'/\cT \substack{\rightarrow\\ \leftarrow} \cT''$. 
By assumption, 
$B$ and $D$ gives equivalences of functors 
$\id_{\cT}\isoto qi$ and $pj\isoto \id_{\cT''}$ respectively. 
Since $jp$ (resp. $iq$) is trivial on 
$\Ker p$ (resp. $\cT'/\cT$), 
it turns out that $A$ (resp. $C$) 
gives an equivalence of functors 
$iq\isoto\id_{\Ker p} $ (resp. $\id_{\cT'/\cT}\isoto jp$) by 
the $\Sigma'$-exact triangle 
$iq \onto{A}\id_{\cT'}\onto{C} jp \onto{E}\Sigma'iq$. 
Hence we obtain the result.\\
Assertion $\mathrm{(2)}$ is a direct consequence of assertion $\mathrm{(1)}$.
\end{proof}

\begin{rem}
\label{rem:quasi split is derived quasi-split}
Let 
$\bE_1=(\cE_1,w_1)\onto{i}\bE_2=(\cE_2,w_2)\onto{p}\bE_3=(\cE_3,w_3)$ 
be a right (resp. left) 
quasi-exact sequence of 
strict relative exact categories 
with a right (resp. left) 
quasi-splitting $(j,q,A,B,C,D)$. 
Then\\
$\mathrm{(1)}$ 
The sequence 
$\Ch_b(i)\Ch_b(q)
\overset{\Ch_b(A)}{\rinf} \id_{\Ch_b(\cE_2)} 
\overset{\Ch_b(C)}{\rdef}
\Ch_b(j)\Ch_b(p)$ is 
an admissible exact sequence of exact endofunctors 
on $\Ch_b(\cE_2)$ by \ref{lemdf:extensional, hom rel ex cat} $\mathrm{(3)}$.\\
$\mathrm{(2)}$ 
Let a pair
$(\tilde{\omega}_{\Ch_b(\cE_2)},\partial):\Ch_b(\cE_2)\to\cT(\Ch_b(\bE_2))
(=\calD_b(\bE_2))$ 
be a widely exact functor in \ref{ex:widelyexact} $\mathrm{(3)}$. 
We put $E:=\partial_{\Ch_b(A),\Ch_b(C)}$. 
Then the sequence of triangulated categories 
$\calD_b(\bE_1)\onto{\calD_b(i)}\calD_b(\bE_2)\onto{\calD_b(p)}\calD_b(\bE_3)$ 
is a right (resp. left) quasi-split exact sequence 
with a right (resp. left) quasi-splitting 
$(\calD_b(j),\calD_b(q),\calD_b(A),\calD_b(B),\calD_b(C),\calD_b(D),E)$.\\
$\mathrm{(3)}$ 
In particular, 
a right (resp. left) quasi-split exact sequence 
of strict relative exact categories is 
a derived right (resp. left) quasi-split exact sequence.
\end{rem}

\begin{ex}[\bf Quasi-split exact sequences]
\label{ex:quasi-split}
$\mathrm{(1)}$ 
We enumerate examples of quasi-split exact sequences 
from \cite[A.2.8]{Sch11}.\\
$\mathrm{(i)}$ 
Let $p:\cT \to \cT'$ 
be a triangle functor which admits a right adjoint 
functor $j:\cT'\to \cT$ with adjunction morphisms 
$D:pj\to \id_{\cT'}$ and $C:\id_{\cT}\to jp$ 
such that $D$ is a natural equivalence and $i:\ker p\to \cT$ 
a natural inclusion functor. 
Then the sequence $\ker p\onto{i}\cT\onto{p}\cT'$ 
is a right quasi-split exact sequence.\\ 
$\mathrm{(ii)}$ 
Let $\cT$ be a triangulated category and 
$\cT_0$ and $\cT_1$ triangulated subcategories of $\cT$ 
such that the following conditions hold.\\
$\mathrm{(a)}$ 
$\cT=\cT_0\vee\cT_1$.\\
$\mathrm{(b)}$ 
$\Hom_{\cT}(x,y)=0$ for any objects $x$ in $\cT_0$ and $y$ in $\cT_1$.\\
We write $i_j:\cT_j \to \cT$ and $\pi_j:\cT \to \cT/\cT_j$ for 
the natural inclusion functor and the natural quotient functor 
for $j=0$, $1$. 
Then the composition $x:\cT_1 \overset{i_1}{\rinc} \cT \onto{\pi_0} \cT/\cT_0$ 
is an equivalence of triangulated categories and 
the sequence $\cT_0 \onto{i_0} \cT \onto{x^{-1}\pi_1} \cT_1$ is a right quasi-split exact sequence.\\
$\mathrm{(2)}$ 
Now we illustrate the typical example of Bousfield-Neeman localization 
of triangulated categories. 
We assume that readers are familiar with Voevodsky's motive category, 
for example \cite[\S 14]{MVW06}. 
Let $k$ be a perfect field and $\calD^{-}(\Sh_{\Nis}(\SmCor(k)))$ 
the derived category of Nisnevich sheaves with transfers over $k$, 
$\DM^{\eff}(k)$ the full subcategory of $\bbA^1$-local objects in 
$\calD^{-}(\Sh_{\Nis}(\SmCor(k)))$ and 
$j:\DM^{\eff}(k)\rinc \calD^{-}(\Sh_{\Nis}(\SmCor(k)))$ 
the inclusion functor. 
Then $j$ has a left adjoint functor 
$C_{\ast}:\calD^{-}(\Sh_{\Nis}(\SmCor(k))) \to \DM^{\eff}(k)$ 
such that the adjunction morphism $C_{\ast}j \to \id_{\DM^{\eff}(k)}$ 
is a natural equivalence. (See Ibid). 
Therefore the sequence 
$\ker C_{\ast} \to
\calD^{-}(\Sh_{\Nis}(\SmCor(k))) \onto{C_{\ast}} 
\DM^{\eff}(k)$ 
is a right quasi-split exact sequence by $\mathrm{(1)}$ $\mathrm{(i)}$. 
\end{ex}

\begin{df}[\bf Extension closed subcategory]
\label{df:ext closed}
Let $\calR\rinc\calR'$ be 
full subcategories of $\RelEx$ such that $\calR'$ contains 
the trivial relative exact category $0$. 
We say that $\calR$ is {\bf closed under extensions} in $\calR'$ 
if it contains the trivial relative exact category $0$ and 
if for any right quasi-split exact sequence of relative exact categories 
$\bE_1\to \bE_2\to \bE_3$ in $\calR'$, 
if $\bE_1$ and $\bE_3$ are in $\calR$, 
$\bE_2$ is also in $\calR$.
\end{df}

\begin{prop}
\label{prop:extension closed}
$\RelEx_{\consist}$ is closed under extensions in $\RelEx_{\strict}$.
\end{prop}

\begin{proof}[\bf Proof]
Let $\bE_1 \onto{i} \bE_2 \onto{p} \bE_3$ be 
a right (resp. left) 
quasi-split exact sequence of strict relative exact category 
and assume that $\bE_1$ and $\bE_3$ are consistent. 
We put $\bE_i=(\cE_i,w_i)$ for $i=1$, $2$ and $3$ 
and let $(j,q,A,B,C,D)$ be a right quasi-splitting of the sequence $(i,p)$. 
For any morphism $a:x\to y$ in $w_2$, 
we consider the commutative diagram of admissible exact sequences 
in $\Ch_b(\cE_2)$ 
$${\footnotesize{\xymatrix{
\Ch_b(i)j_{\cE_1}qx \ar@{>->}[r] \ar[d]_{\Ch_b(i)j_{\cE_1}(qa)} & 
j_{\cE_2}x \ar@{->>}[r] \ar[d]_{j_{\cE_2}a} &
\Ch_b(j)j_{\cE_3}px \ar[d]^{\Ch_b(j)j_{\cE_3}pa}\\
\Ch_b(i)j_{\cE_1}qy \ar@{>->}[r] & 
j_{\cE_2}y \ar@{->>}[r] &
\Ch_b(j)j_{\cE_3}pa.
}}}$$
Here both $\Ch_b(i)j_{\cE_1}(qa)$($=j_{\cE_2}iqa$) and 
$\Ch_b(j)j_{\cE_3}pa$($=j_{\cE_2}jpa$) are in $qw_2$ by assumption. 
Since $qw_2$ is closed under extensions 
by \ref{prop:compness} $\mathrm{(1)}$, 
$j_{\cE_2}a$ is also in $qw_2$. 
Hence we obtain the result.
\end{proof}

\begin{df}[\bf Flag]
\label{df:flag}
$\mathrm{(1)}$ 
A {\bf right} (resp. {\bf left}) {\bf flag} of a triangulated category 
$\cT$ is a finite sequence of fully faithful functors 
$$\{0\}=\cT_0 \onto{k_0}\cT_1\onto{k_1}\cT_2\onto{k_2}\cdots\onto{k_{n-1}}
\cT_n=\cT$$
such that for any $i$, 
the canonical sequence $\cT_i\onto{k_i}\cT_{i+1}\to \cT_{i+1}/\cT_i$ 
is a right (resp. left) quasi-split exact sequence.\\
$\mathrm{(2)}$ 
A {\bf derived right} (resp. {\bf left}) {\bf flag} of a 
relative exact category $\bE$ is 
a finite sequence of derived fully faithful relative exact functors 
$$\{0\}=\bE_0 \onto{k_0}\bE_1\onto{k_1}\bE_2\onto{k_2}\cdots\onto{k_{n-1}}
\bE_n=\bE$$
such that the induced sequence of triangulated categories 
$$\{0\}=\calD_b(\bE_0) \onto{\calD_b(k_0)}\calD_b(\bE_1)\onto{\calD_b(k_1)}
\calD_b(\bE_2)\onto{\calD_b(k_2)}\cdots\onto{\calD_b(k_{n-1})}
\calD_b(\bE_n)=\calD_b(\bE)$$
is a right (resp. left) flag of $\calD_b(\bE)$.
\end{df}

\begin{ex}
\label{ex:flag}
Let $A$ be a commutative noetherian ring with unit and 
$\bE=(\Perf_{\bbP^n_A},\qis)$ 
the relative exact category of perfect complexes over $\bbP^n_A$. 
For any integer $0\leq k\leq n$, 
we let $\cT_k$ and $\cT_k'$ be the thick subcategories of 
$\cT(\bE)\isoto\calD_b(\bE)$ 
spanned by $\cO(l)$ where $-k\leq l\leq 0$ and $\cO(-k)$ respectively and 
$\cE_k$ the pull-back of $\cT_k$ by the projection functor 
$\Perf_{\bbP^n_A} \to \cT(\bE)\isoto\calD_b(\bE)$ and 
we put $\bE_k=(\cE_k,\qis)$. 
Then we can easily check that 
the subcategories $\cT_{k-1}$, $\cT_k'\rinc \cT_k$ satisfy the 
conditions in \ref{ex:quasi-split} $\mathrm{(1)}$ $\mathrm{(ii)}$. 
(See \cite[3.5.1]{Sch11}). 
Therefore the sequence of the inclusion functors 
$$0 \to \bE_1 \to \bE_2 \to \cdots \to \bE_n=\bE$$
is a derived right flag. 
\end{ex}

\section{Additive and localizing theories}
\label{sec:Localizing theory}

\sn
In this section, 
we will axiomize both 
$K^W$- and $\bbK$-theories for suitable 
relative exact categories. 
We will describe the relationship 
between the Gillet-Waldhausen type formula 
and the fibration theorem in \ref{lemdf:excellent}. 
Using this observation and the results in previous sections, 
we will bring examples of relative exact categories 
which behave well to $K^W$- and $\bbK$-theories in \ref{prop:excellent} 
and in \ref{ex:device}. 
We will also formulate additive and localizing theories for 
certain relative exact categories and straighten up 
their fundamental properties in \ref{prop:fund property of add theory} 
and \ref{prop:fund property of loc theory} respectively. 

\begin{df}[\bf Reasonable subcategories]
\label{df:Res subcat}
A full subcategory $\calR$ of $\RelEx_{\strict}$ is 
{\bf reasonable} if $\calR$ contains the trivial relative exact category $0$ 
and if for any relative exact categories $\bE=(\cE,w)$ in $\calR$, 
$(\cE,i_{\cE})$, $(\cE^w,i_{\cE^w})$ and $\Ch_b\bE$ are also in $\calR$.
\end{df}

\begin{rem}
\label{rem:i omit}
Let $F$ be a functor from a reasonable subcategory $\calR$ of $\RelEx_{\strict}$ 
to a category $\cX$ and $\cE$ an exact category such that $(\cE,i_{\cE})$ is in $\calR$. 
We sometimes write $F(\cE)$ for $F(\cE,i_{\cE})$. 
We say that a relative exact functor 
$f:\bF \to \bG$ in $\calR$ is a {\bf $F$-equivalence} if 
$F(f):F(\bF)\to F(\bG)$ is an isomorphism in $\cX$. 
\end{rem}

\begin{ex}[\bf Reasonable subcategories]
\label{ex:resonable subcat}
The categories $\RelEx_{?}^{\#}$ and $\WalEx^{\#}_{?}$ 
for $\#\in\{+,\text{nothing}\}$ and $?\in\{\consist,\SOL\}$
are reasonable subcategories 
by \ref{prop:compness} 
and \ref{prop:bicomp is adm}.\\
\end{ex}

\begin{lemdf}[\bf Excellent relative exact categories]
\label{lemdf:excellent} 
Let $\calR$ be a reasonable subcategory of $\RelEx_{\consist}$ 
and $F$ a functor from $\calR$ to the stable category of 
spectra and $\bE=(\cE,w)$ a consistent relative exact category in $\calR$. 
We assume the following axiom holds.

\sn
{\bf ($F$-weak Gillet-Waldhausen axiom).} 
The canonical functors 
$$j_{\cE}:(\cE,i_{\cE}) \to (\Ch_b(\cE),\qis) \ \ \text{and}$$
$$j_{\cE^w}:(\cE^w,i_{\cE^w}) \to (\Ch_b(\cE^w),\qis)$$
are $F$-equivalences.

\sn
Then two of the following axioms imply the other axiom. 

\sn
{\bf ($F$-Gillet-Waldhausen axiom).} 
The canonical functors $j_{\cE}:\bE \to \Ch_b(\bE)$ 
is a $F$-equivalence. 

\sn
{\bf ($F$-weak fibration axiom).} 
The inclusion functor $\cE^w\rinc \cE$ and the identity functor of $\cE$ 
induces a fibration sequence of spectra
$$F(\Ch_b(\cE^w),\qis) \to F(\Ch_b(\cE),\qis) \to F(\Ch_b(\bE)).$$

\sn
{\bf ($F$-fibration axiom).} 
The inclusion functor $\cE^w\rinc \cE$ and the identity functor of $\cE$ 
induces a fibration sequence of spectra
$$F(\cE^w,i_{\cE^w}) \to F(\cE,i_{\cE}) \to F(\bE).$$

\sn
We say that $\bE$ is {\bf $F$-excellent} if $\bE$ satisfies four axioms above. 
We write $F-\calR$ for the full subcategory of $F$-excellent relative 
exact categories in $\calR$. 
\end{lemdf}

\begin{proof}[\bf Proof]
Consider the commutative diagram below.
$${\footnotesize{\xymatrix{
F(\cE^w,i_{\cE^w}) \ar[r] \ar[d]_{F(j_{\cE^w})}^{\wr} &
F(\cE,i_{\cE}) \ar[r] \ar[d]_{F(j_{\cE})}^{\wr} & 
F(\bE) \ar[d]^{F(j_{\cE})}\\
F(\Ch_b(\cE^w),\qis) \ar[r] &
F(\Ch_b(\cE),\qis) \ar[r] & 
F(\Ch_b(\bE)) .
}}}$$
Then if we assume both the horizontal lines 
above are fibration sequences of spectra, 
then the map $F(j_{\cE}):F(\bE) \to F(\Ch_b(\bE))$ 
is a homotopy equivalence of spectra. 
Next if we assume the map $F(j_{\cE}):F(\bE)\to F(\Ch_b(\bE))$ is a homotopy equivalence of spectra, 
then the top line is a fibration sequence of spectra if and only if 
the bottom line is. 
\end{proof}

\begin{prop}[\bf Examples of excellent relative exact categories]
\label{prop:excellent}
$\mathrm{(1)}$ 
A very strict solid Waldhausen 
exact category which satisfies the $K^W$-fibration 
axiom is $K^W$-excellent.\\
$\mathrm{(2)}$ 
A very strict consistent relative exact category is $\bbK$-excellent. 
\end{prop}

\begin{proof}[\bf Proof]
Let $\bE=(\cE,w)$ be a very strict solid Waldhausen 
(resp. very strict relative) 
exact category. 
Then the sequence 
$$(\Ch_b(\cE^w),\qis) \to (\Ch_b(\cE),\qis) \to \Ch_b(\bE)$$
induced from the inclusion functor $\cE^w\rinc \cE$ and 
the identity functor of $\cE$ 
is an exact sequence 
(resp. weakly exact sequence) 
of complicial Waldhausen categories 
by \ref{cor:weakly exact seq}. 
Therefore by \cite[3.2.23]{Sch11} (resp. \cite[3.2.27]{Sch11}), 
$\bE$ satisfies the $K^W$-weak fibration 
(resp. $\bbK$-weak fibration) axiom. 
On the other hand, 
for any exact category (resp. consistent relative exact category), 
it satisfies $K^W$-weak Gillet-Waldhausen 
(resp. $\bbK$-Gillet-Waldhausen)
axiom by \cite{Cis02} (resp. \cite[3.2.29]{Sch11} and 
\ref{cor:comp of derived cat}). 
Hence we obtain the result by \ref{lemdf:excellent}. 
\end{proof}

\begin{ex}
\label{ex:device}
$\mathrm{(1)}$ 
Let $\cE$ be an exact category. 
Since $\cE^{i_{\cE}}$ is trivial, 
the pair $(\cE,i_{\cE})$ 
is a very strict solid Waldhausen exact category 
which satisfies the $K^W$-fibration axiom. 
In particular, 
it is $K^W$-excellent and $\bbK$-excellent by \ref{prop:excellent}.\\
$\mathrm{(2)}$ 
In \cite[Theorem 11]{Sch06}, 
Schlichting showed that a Waldhausen category which satisfies 
the extensional, the saturated and the factorization axioms, 
also satisfies the $K^W$-fibration axiom. 
In particular, a bicomplicial pair 
satisfies the $K^W$-fibration axiom 
by \ref{rem:a bicomp pair is good biwaldhausen}. 
Hence a bicomplicial pair is  
$K^W$-excellent and $\bbK$-excellent 
by \ref{prop:bicompair is very strict}, 
\ref{prop:bicomp is adm} and \ref{prop:excellent}. 
\end{ex}

\begin{para}[\bf Proof of Theorem~\ref{thm:agreement}]
\label{proof:agreement}
Let $\bE=(\cE,w)$ be a consistent relative exact category. 
Assertion $\mathrm{(1)}$ follows from \cite[Theorem 8]{Sch06}. 

\sn
$\mathrm{(2)}$ 
If $w$ is the class of all isomorphisms in $\cE$, 
the definition of $\bbK(\bE)$ is compatible with Definition~8 in \cite{Sch06}. 
If $\bE$ is a complicial Waldhausen category, 
the result follows from \ref{cor:comp of derived cat} $\mathrm{(2)}$ 
and \cite[Proposition~3]{Sch06}.

\sn
$\mathrm{(3)}$ 
We have a homotopy equivalence of spectra 
$\bbK_n(\bE)=\bbK^S_n(\Ch_b(\bE))\isoto K_n^W(\Ch_b(\bE))$ 
for any positive integer $n$ 
by \cite[Theorem 8]{Sch06} again. 
If $\bE$ is a very strict solid Waldhausen exact category 
and satisfies the $K^W$-fibration axiom, 
then $j_{\cE}$ induces a homotopy equivalence of spectra between 
$K_n^W(\bE)$ and $K_n^W(\Ch_b(\bE))$ 
by \ref{prop:excellent}. 
\qed
\end{para}

\sn
Recall the definition of $E(\cE)$ 
and the exact functors $s$, $m$, $q$ from $E(\cE)$ to $\cE$ 
from Conventions $\mathrm{(7)}$ $\mathrm{(xiv)}$.

\begin{lemdf}[\bf Additive and localizing theories] 
\label{lemdf: Loc th}
$\mathrm{(1)}$ 
(\cf \cite[1.3.2]{Wal85}, \cite[3.1]{GSVW92}). 
Let $F$ 
be a functor from a full subcategory $\calR$ of $\RelEx$ 
which closed under extensions to an additive category $\cB$. 
Each of the following assertions implies 
all the three others.\\
$\mathrm{(i)}$ 
For any right quasi-split exact sequence 
$\bE\onto{i}\bF\onto{p}\bG$ 
in $\calR$ 
with a right quasi-splitting 
$(j,q)$, the morphism 
$$\displaystyle{\begin{pmatrix}F(q)\\ F(p)\end{pmatrix}:
F(\bF) \to F(\bE)\oplus F(\bG)}$$
is an isomorphism.\\
$\mathrm{(ii)}$ 
For any relative exact category $\bE$ in $\calR$, 
the following projection is an isomorphism 
$$
\displaystyle{\begin{pmatrix}F(s)\\ F(q)\end{pmatrix}:
F(E(\bE)) \to F(\bE)\oplus F(\bE)}.
$$
$\mathrm{(iii)}$ 
For any relative exact category $\bE$ in $\calR$, 
we have the equality $F(m)=F(s)+F(q)$ 
for morphisms $F(E(\bE))\to F(\bE)$.\\
$\mathrm{(iv)}$ 
For any admissible exact sequence $f\overset{A}{\rinf}g\overset{B}{\rdef}h$ 
of relative exact functors 
between relative exact categories $\bE\to \bF=(\cF,v)$ in $\calR$, 
we have the equality $F(g)=F(f)+F(h)$.\\
We say that $F$ is an {\bf additive theory} 
if $F$ satisfies the assertions above.\\
$\mathrm{(2)}$ 
A functor $F$ from a reasonable subcategory $\calR$ of $\RelEx_{\strict}$ to 
the stable category of spectra is 
a {\bf localizing} (resp. {\bf strictly localizing}) {\bf theory} 
if $F$ sends a weakly exact (resp. exact) sequence to 
a fibration sequence of spectra. 
\qed
\end{lemdf}

\begin{prop}
\label{prop:fund property of add theory}
Let $\calR$ be a full $2$-subcategory of $\RelEx^{+}$ 
which closed under extensions 
and $F$ is an additive theory from $\calR$ to an additive category $\cB$. 
Then\\
$\mathrm{(1)}$ 
The canonical map $F(0)\to 0$ is an isomorphism.\\
$\mathrm{(2)}$ 
Assume that $F$ is categorical homotopy invariant. 
Then for any bicomplicial pair $\bC=(\cC,v)$ with 
the suspension functor $T:\cC \to \cC$ in $\calR$, 
we have the equality $F(T)=-\id_{F(\cC)}$.
\end{prop}

\begin{proof}[\bf Proof]
$\mathrm{(1)}$ 
Consider the right quasi-split sequence $0 \to 0 \to 0$. 
Then we obtain the isomorphism $F(0)\isoto F(0)\oplus F(0)$ 
by additivity. 
Hence we obtain the result.

\sn
$\mathrm{(2)}$ 
By applying \ref{lemdf: Loc th} $\mathrm{(1)}$ $\mathrm{(iv)}$ 
to the admissible exact sequence 
$\id_{\cC} \to C \to T$ on $\cC$, 
we get the equality $\id_{F(\cC)}+F(T)=F(C)$. 
On the other hand, since the natural transformation $0 \to C$ is 
a relative natural equivalence, 
$F(C)=0$ by homotopy invariance of $F$.
\end{proof}

\begin{prop}
\label{prop:fund property of loc theory}
Let $F$ be a localizing {\rm (}resp. strictly localizing{\rm )} 
theory on a reasonable subcategory $\calR$ 
of $\RelEx_{\strict}$. 
Then\\
$\mathrm{(1)}$ 
The canonical map $F(0) \to 0$ is a homotopy equivalence of spectra.\\
$\mathrm{(2)}$ 
{\rm {\bf (Approximation theorem).}} 
Let $f:\bE \to \bF$ be a relative exact functor in $\calR$. 
If $f$ is a weakly derived equivalence 
{\rm (}resp. derived equivalence{\rm )}, 
then $f$ is a $F$-equivalence.\\
$\mathrm{(3)}$ 
{\rm {\bf (Additivity theorem).}} 
Assume that moreover $F$ is a categorical homotopy invariant functor 
and $\calR$ is closed under extensions. 
Then $F$ is an additivity theory.
\end{prop}

\begin{proof}[\bf Proof]
$\mathrm{(1)}$ 
Consider the following commutative diagram of fibration sequences of spectra 
$$
{\footnotesize{\xymatrix{
F(0) \ar[r] \ar@{=}[d] & 
F(0) \ar[r] \ar@{=}[d] &
F(0) \ar[d]\\
F(0) \ar[r] & 
F(0) \ar[r] &
0.
}}}
$$
Then it turns out that $F(0) \to 0$ is a homotopy equivalent of spectra. 

\sn
$\mathrm{(2)}$ 
We just apply the assumption of $F$ to a weakly exact 
(resp. exact) sequence
$0 \to \bE \onto{f} \bF$.

\sn
$\mathrm{(3)}$ 
let $\bE \onto{i} \bF \onto{p} \bG$ be a right 
quasi-split exact sequence with a right quasi-splitting $(q,j)$. 
By \ref{cor:quasi split is exact} and 
\ref{rem:quasi split is derived quasi-split}, 
the sequence $(i,p)$ is exact. 
Consider the following commutative diagram of fibration sequences of spectra 
$$
{\footnotesize{\xymatrix{
F(\bE) \ar[r]^{F(i)} \ar@{=}[d] \ar@{}[rd]|{\bigstar} & 
F(\bF) \ar[r]^{F(p)} \ar[d]^{\begin{pmatrix}F(q)\\ F(p)\end{pmatrix}} & 
F(\bG) \ar@{=}[d]\\
F(\bE) \ar[r]_{\begin{pmatrix}\id_{F(\bE)}\\ 0\end{pmatrix}} & 
F(\bE)\oplus F(\bG) \ar[r]_{\begin{pmatrix}0 & \id_{F(\bG)}\end{pmatrix}} & 
F(\bG).
}}}$$
Here we use homotopy invariance of $F$ to prove the commutativity of 
$\bigstar$. 
Then we learn that the map 
$\displaystyle{\begin{pmatrix}F(q)\\ F(p)\end{pmatrix}:
F(\bF) \to F(\bE)\oplus F(\bG)}$ 
is a homotopy equivalence of spectra.
\end{proof}

\begin{cor}
\label{cor:restrict localzing}
Let $F$ be a functor from 
a reasonable subcategory $\calR$ of 
$\RelEx_{\consist}$ to the stable category of spectra.
Assume that the following three conditions hold.\\
$\mathrm{(1)}$ 
For any essentially small exact category $\cE$, 
$(\cE,i_{\cE})$ is in $F-\calR$.\\
$\mathrm{(2)}$ 
$\Ch_b(F-\calR)\subset F-\calR$.\\
$\mathrm{(3)}$ 
The functor $F\Ch_b$ on $F-\calR$ is a localizing 
(resp. strictly localizing) theory.\\
Then $F$ is a localizing (resp. strictly localizing) theory on $F-\calR$. 
\end{cor}

\begin{proof}[\bf Proof]
For any weakly exact 
(resp. exact) 
sequence $\bE=(\cE,u) \onto{i} \bF=(\cF,v) \onto{j} \bG=(\cG,w)$ in $F-\calR$, 
we have the commutative diagram in the stable category of spectra below
$${\footnotesize{\xymatrix{
F(\bE) \ar[r]^{F(i)} \ar[d]_{j_{\cE}} & F(\bF) \ar[r]^{F(j)} \ar[d]_{j_{\cF}} & F(\bG) \ar[d]^{j_{\cG}}\\
F(\Ch_b(\bE)) \ar[r]_{F(\Ch_b(i))} & F(\Ch_b(\bF)) \ar[r]_{F(\Ch_b(j))} & F(\Ch_b(\bE)).
}}}$$
Here the bottom line is a fibration sequence 
by assumption 
and all vertical lines are 
homotopy equivalence by $F$-Gillet-Waldhausen axiom. 
Therefore we get the desired fibration sequence 
$$F(\bE) \onto{F(i)} F(\bF) \onto{F(j)} F(\bG).$$
\end{proof}

\begin{ex}
\label{ex:K^W is strictly localizing}
Exact categories and bicomplicial pairs are $K^W$-excellent by \ref{ex:device}. 
Moreover $K^W\Ch_b$ is a strongly localizing theory 
on $K^W-\RelEx_{\consist}$ by \ref{cor:weakly exact seq} $\mathrm{(1)}$ 
and \cite[3.2.23]{Sch11}. 
Hence the functor $K^W$ is a strictly localizing theory 
on $K^W-\RelEx_{\consist}$ by \ref{cor:restrict localzing}. 
\end{ex}

\section{Multi semi-direct products}
\label{sec:Multi semi-direct} 

\sn
In this section, 
we elaborate the theory of multi semi-direct products of exact categories 
as a continuation of \cite{Moc11}. 
In the last of this section, 
we will get \ref{thm:flag} which 
is an abstraction of Corollary~5.14 in \cite{Moc11}.

\begin{para}
\label{para:cubedf} 
For a set $S$, 
an {\bf $S$-cube} in a category $\cC$ is a contravariant functor 
from $\cP(S)$ to $\cC$. 
We denote the category of $S$-cubes in a category $\cC$ by $\Cub^S\cC$ where 
morphisms between cubes are just natural transformations. 
Let $x$ be an $S$-cube in $\cC$. 
For any $T\in\cP(S)$, 
we denote $x(T)$ by $x_T$ and call it a 
{\bf vertex of $x$} ({\bf at $T$}). 
For $k \in T$, 
we also write $d^{x,k}_T$ or 
shortly $d^k_T$ 
for $x(T\ssm\{k\} \rinc T)$ 
and call it 
a ({\bf $k-$}){\bf boundary morphism of $x$} ({\bf at $T$}). 
An $S$-cube $x$ is {\bf monic} if for any pair of subsets 
$U\subset T$ in $S$, $x(U\subset V)$ is a monomorphism. 
\end{para}

\sn
In the rest of this section, 
we assume that $S$ is a finite set.

\begin{para}[\bf Admissible cubes]
\label{para:admcube}
Fix an $S$-cube $x$ 
in an abelian category $\cA$. 
For any element $k$ in $S$, 
the {\bf $k$-direction $0$-th homology} of $x$ 
is an $S\ssm\{k\}$-cube $\Homo_0^k(x)$ in $\cA$ 
and defined by $\Homo_0^k(x)_T:=\coker d_{T\cup\{k\}}^k$. 
For any $T\in\cP(S)$ and $k\in S\ssm T$, 
we denote the canonical projection morphism 
$x_T \to \Homo_0^k(x)_T$ by $\pi^{k,x}_T$ or simply $\pi^k_T$. 
When $\# S=1$, 
we say that $x$ is {\bf admissible} if $x$ is monic, 
namely if its unique boundary morphism is a monomorphism. 
For $\# S>1$, 
we define the notion of an admissible cube inductively 
by saying that 
$x$ is {\bf admissible} if $x$ is monic and if 
for every $k$ in $S$, 
$\Homo^k_0(x)$ is admissible. 
If $x$ is admissible, 
then for any distinct elements 
$i_1,\ldots,i_k$ in $S$ and for any automorphism $\sigma$ of 
the set $\{i_1,\ldots,i_k\}$, 
the identity morphism on $x$ induces an isomorphism:
$$\Homo^{i_1}_0(\Homo^{i_2}_0(\cdots(\Homo_0^{i_k}(x))\cdots))
\isoto 
\Homo^{i_{\sigma(1)}}_0(\Homo^{i_{\sigma(2)}}_0(\cdots(
\Homo_0^{i_{\sigma(k)}}(x))\cdots))$$
where $\sigma$ is a bijection on $S$. (cf. \cite[3.11]{Moc11}). 
For an admissible $S$-cube $x$ and a subset $T=\{i_1,\ldots,i_k\}\subset S$, 
we put 
$\Homo^T_0(x):=\Homo^{i_1}_0(\Homo^{i_2}_0(\cdots(\Homo^{i_k}_0(x))\cdots))$ 
and $\Homo^{\emptyset}_0(x)=x$. 
Notice that $\Homo^T_0(x)$ is an $S\ssm T$-cube for any $T\in\cP(S)$. 
Then we have the isomorphisms
\begin{equation}
\label{equ:Totisom}
\Homo_p(\Tot(x))\isoto 
\begin{cases}
\Homo^S_0(x) & \text{for $p=0$}\\
0 & \text{otherwise}
\end{cases}
\ \ .
\end{equation}
See \cite[3.13]{Moc11}.
\end{para}

\sn
In the rest of this section, 
let $U$ and $V$ be a pair of disjoint subsets of $S$. 

\begin{df}[\bf Multi semi-direct products]
\label{df:mult semi-direct prod}
Let $\fF=\{\cF_T\}_{T\in\cP(S)}$ be 
a family of full subcategories of $\cA$.\\
$\mathrm{(1)}$ 
We put $\fF|_U^V:=\{\cF_{V\coprod T}\}_{T\in\cP(U)}$ 
and call it the {\bf restriction of $\fF$} 
(to $U$ along $V$).\\
$\mathrm{(2)}$ 
Then we define 
$\ltimes \fF=\underset{T\in\cP(S)}{\ltimes} \cF_T$ 
the {\bf multi semi-direct products 
of the family $\fF$} 
as follows. 
$\ltimes \fF$ 
is the full subcategory of $\Cub^S(\cA)$ 
consisting of 
those cube $x$ 
such that $x$ is admissible and each vertex of 
$\Homo^T_0(x)$ is in $\cF_T$ for any $T\in\cP(S)$.\\
$\mathrm{(3)}$ 
If $S$ is a singleton (namely $\#S=1$), 
then we write $\cF_{S}\ltimes \cF_{\emptyset}$ for $\ltimes \fF$.\\ 
\end{df}

\begin{rem}
\label{rem:indformulaformultsemiprod}
For any element $u$ in $U$, 
we have the equality $\ltimes\fF|_U^V=(\ltimes \fF|_{U\ssm\{u\}}^{V\coprod\{u\}})
\ltimes (\ltimes \fF|^{V}_{U\ssm\{u\}})$. 
(See \cite[3.19]{Moc11}).
\end{rem}

\begin{df}[\bf Exact family] 
\label{df:exact family} 
Let $\fF=\{\cF_T\}_{T\in\cP(S)}$ be 
a family of strict exact subcategories of an 
abelian category $\cA$. 
We say that $\fF$ is an {\bf exact family} (of $\cA$) 
if for any disjoint pair of subsets $P$ and $Q$ of $S$, 
$\ltimes \fF|_P^Q$ is a strict exact subcategory of 
$\Cub^P\cA$.\\
\end{df}

\begin{lem}
\label{lem:exact family}
{\rm (\cf \cite[3.20]{Moc11}).} 
Let $\fF=\{\cF_T\}_{T\in\cP(S)}$ be 
a family of full subcategories of $\cA$. 
If $\cF_T$ is closed under either extensions or 
taking sub- and quotient objects and direct sums in $\cA$, 
then $\fF$ is an exact family.
\qed
\end{lem}

\sn
In the rest of this section, 
let $\fF=\{\cF_T\}_{T\in\cP(S)}$ be an exact family of 
$\cA$. 

\begin{lemdf}
\label{lemdf:three exact functors}
For any non-empty subset $W$ of $U$ 
and $j=0$ or $1$, 
we will define 
$$\ext^V_{U\ssm W,W}:\ltimes\fF|_{U\ssm W}^V\to \ltimes\fF|_U^V,$$ 
$$\res_{U,W}^{V,j}:\ltimes\fF|_U^V \to \ltimes\fF|_{U\ssm W}^{V} 
\ \ \ \text{and}$$ 
$$\Homo_U^{V,W}:\ltimes\fF|_U^V\to\ltimes\fF|_{U\ssm W}^{V\coprod W}$$ 
to be exact functors by induction on the cardinality of $W$ 
and call them the 
{\bf extension functor}, the {\bf restriction functor} and the 
{\bf homology functor}. 
First assume that $W$ is a singleton $W=\{w\}$. 
Then we have the equality $\ltimes\fF|_U^V=
(\ltimes \fF|_{U\ssm W}^{V\coprod W})\ltimes 
(\ltimes\fF|_{U\ssm W}^{V})$ 
by \ref{rem:indformulaformultsemiprod}. 
Regarding $\ltimes \fF|_U^V$ 
as a subcategory of one dimensional cubes in 
$\ltimes \fF|^V_{U\ssm W}$, 
we define the three exact functors 
$\ext^V_{U\ssm W,W}$, 
$\res_{U,W}^{V,j}$ and 
$\Homo_U^{V,W}$ 
by sending an object 
$x$ in $\ltimes\fF|_{U\ssm W}^V$ 
and $[x_1 \onto{d^x} x_0]$ in $\ltimes \fF|_U^V$ 
to $[x\onto{\id_x}x]$, $x_j$ and $\coker d^x$ respectively.\\
Next for any non-trivial disjoint decomposition of $W=W_1\coprod W_2$, 
we put 
$$\ext^V_{U\ssm W,W}:=\ext^V_{U\ssm W_2,W_2}\ext^V_{U\ssm W,W_1},$$
$$\res^{V,j}_{U,W}:=\res^{V,j}_{U\ssm W_1,W_2}
\res^{V,j}_{U,W_1} \ \ \ \text{and}$$
$$\Homo_U^{V,W}:=\Homo_{U\ssm W_1}^{V\coprod W_1,W_2}\Homo_U^{V,W_1}.$$
Then the definitions of $\ext^V_{U\ssm W,W}$, 
$\res_{U,W}^{V,j}$ and $\Homo_U^{V,W}$ 
do not depend upon a choice of disjoint decomposition of $W$ 
up to canonical isomorphisms and they are exact functors. 
\qed
\end{lemdf}

\sn
we can easily prove the following lemma. 

\begin{lem}
\label{lem:compat of three exact functors}
$\mathrm{(1)}$ 
For any non-empty subset $W$ of $U$, 
we have the equality 
$$\res_{U,W}^{V,j}\ext_{U\ssm W}^{V,W}=\id_{\ltimes \fF|^V_{U\ssm W}}$$ 
for $j=0$, $1$.\\
$\mathrm{(2)}$ 
For any pair of disjoint non-empty subsets $W_1$ and $W_2$ of $U$, 
we have the equality 
$$\Homo_U^{V,W_1}\ext_{U\ssm W_2}^{V,W_2}=
\ext_{U\ssm W_1}^{V\coprod W_1,W_2}\Homo_{U\ssm W_2}^{V,W_1}.$$
\qed
\end{lem}

\sn
In the rest of this section, 
let $w$ be a class of morphisms in $\cF_S$ and 
assume that for any zero objects $0$ and $0'$ in $\cF_S$, 
the canonical morphism $0 \to 0'$ is in $w$. 

\begin{df}
\label{df:weak equivalences on ltimesfF}
We will define the class of {\bf total weak equivalences} 
$tw(\ltimes \fF|_U^V)$ or simply $tw$ 
in $\ltimes \fF|_U^V$. 
First assume that $S=U\coprod V$. 
Then $tw$ is defined by pull-back of $w$ by the exact functor 
$\Homo_U^{V,U}:\ltimes\fF|_U^V \to \cF_S$.\\
Next we assume that $U\coprod V\neq S$. 
Then $tw$ is defined 
by pull-back of $tw(\ltimes \fF|_{S\ssm V}^V)$ by the exact functor 
$\ext_{U,S\ssm(U\coprod V)}^V:\ltimes\fF|_U^V \to \ltimes \fF|_{S\ssm V}^V$.
\end{df}

\begin{rem}
\label{rem:total qis}
$\mathrm{(1)}$ 
(\cf. \cite[3.19]{Moc11}). 
For any element $u$ of $U$, we have the equality 
$${(\ltimes \fF|_U^V)}^{tw}
={(\ltimes \fF|_{U\ssm\{u\}}^{V\coprod\{u\}})}^{tw}
\ltimes(\ltimes \fF|^V_{U\ssm\{u\}}).$$
$\mathrm{(2)}$ 
For any non-empty subset $W$ of $U$ and $j=0$ or $1$, 
the exact functors $\ext^V_{U\ssm W,W}$, $\res^{V,j}_{U,W}$ and $\Homo_U^{V,W}$ 
preserve $tw$.
\end{rem}

\sn
We can easily prove the following lemma. 

\begin{lem}
\label{lem:denpan}
if $w$ satisfies the extensional 
{\rm (}resp. gluing, cogluing{\rm )} axiom, 
then $tw$ in $\ltimes \fF|_U^V$ also satisfies the axiom. 
\qed
\end{lem}

\begin{df}
\label{df:comaptible weak eq}
We say that $w$ is {\bf compatible with} $\fF$ 
if for any disjoint pair of subsets $P$ and $Q$ of $S$, 
$(\ltimes\fF|_P^Q,tw)$ is a strict relative exact category. 
\end{df}

\begin{cor}
\label{cor:denpan}
If either $(\cF_S,w)$ satisfies the extensional axiom or 
$(\cF_S,w)$ is a Waldhausen exact category, 
then $w$ is compatilbe with $\fF$. 
\end{cor}

\begin{proof}[\bf Proof]
If either $(\cF_S,w)$ satisfies the extensional axiom or 
$(\cF_S,w)$ is a Waldhausen exact category, 
then $(\ltimes\fF|_{P}^{Q},tw)$ 
is also for any pair of disjoint subsets $P$ and $Q$ of $S$ 
by \ref{lem:denpan}. 
Hence $(\ltimes\fF|_{P}^{Q},tw)$ is 
strict by \ref{prop:strict exact categories}. 
\end{proof}

\begin{df}[\bf Adroit system] 
\label{df:good triple}
(\cf \cite[2.20]{Moc11}). 
An {\bf adroit} (resp. a {\bf strongly adroit}) {\bf system} in 
an abelian category $\cA$ is a system $\cX=(\cE_1,\cE_2,\cF)$ consisting of 
strict exact subcategories $\cE_1\rinc \cE_2\linc \cF$ in $\cA$ and 
they satisfies the following axioms 
{\bf (Adr 1)}, {\bf (Adr 2)}, {\bf (Adr 3)} and 
{\bf (Adr 4)} 
(resp. {\bf (Adr 1)}, {\bf (Adr 2)}, 
{\bf (Adr 3)} and  {\bf (Adr 5)}). 

\sn
{\bf (Adr 1)} 
$\cF\ltimes \cE_1$ and $\cF\ltimes\cE_2$ are 
strict exact subcategories of $\Ch_b(\cA)$. \\
{\bf (Adr 2)} 
$\cE_1$ is closed under extensions in $\cE_2$.\\
{\bf (Adr 3)} 
Let 
$x \rinf y \rdef z$ 
be an admissible short exact sequence in $\cA$. 
Assume that $y$ is isomorphic 
to an object in $\cE_1$ 
and $z$ is isomorphic 
to an object in $\cE_1$ or $\cF$. 
Then $x$ is isomorphic 
to an object in $\cE_1$.\\
{\bf (Adr 4)} 
For any object $z$ in $\cE_2$, 
there exists an object $y$ in $\cE_1$ 
and an admissible epimorphism $y \rdef z$.\\ 
{\bf (Adr 5)} 
For any non-negative integer $m$ 
and an object $z$ in $\cE_2^{[m]}$, 
there exists an object $y$ in $\cE_1^{[m]}$ 
and an admissible epimorphism $y \rdef z$.
\end{df}

\begin{rem}
\label{rem:adroit system}
Let $(\cE_1,\cE_2,\cF)$ 
be an adroit (resp. a strongly adroit) system of $\cA$ 
and $\calH$ a strict exact subcategory of $\cF$. 
Then a triple $(\cE_1,\cE_2,\calH)$ 
is an adroit (resp. a strongly adroit) system.
\end{rem}

\sn
Next theorem is a variation of Theorem~2.21 in \cite{Moc11}.

\begin{thm}
\label{thm:adroit system}
Let $\cX=(\cE_1, \cE_2, \cF)$ be an adroit system and 
$v$ a class of morphisms in $\cF$ such that 
either $v$ satisfies the extension axiom or 
$(\cF,v)$ is a Waldhausen exact category. 
Then\\
$\mathrm{(1)}$ 
The canonical inclusion functors
$$\cF\ltimes \cE_1 \rinc \cF\ltimes \cE_2 \text{  and  }  
\cF^v\ltimes \cE_1 \rinc \cF^v\ltimes \cE_2$$
satisfy the resolution conditions. 
In particular, the inclusion functor 
$$(\cF\ltimes \cE_1,tv) \to (\cF\ltimes \cE_2,tv)$$ 
is a derived equivalence.\\
$\mathrm{(2)}$ 
Assume that $(\cF,v)$ is a consistent relative exact category. 
Then the relative exact functor 
$\Homo_0:(\cF\ltimes\cE_1,tw)\to(\cF,w)$ 
is a derived equivalence.\\
$\mathrm{(3)}$ 
The exact functors $\cE_1\to \cF\ltimes\cE_1$, $x\mapsto [x\onto{\id_x}x]$ 
and $\Homo_0:\cF\ltimes\cE_1 \to \cF$ yield 
a right quasi-split exact sequence 
$(\cE_1,i_{\cE_1})\to (\cF\ltimes\cE_1,i_{\cF\ltimes\cE_1})\to 
(\cF,i_{\cF})$. 
\end{thm}

\begin{proof}[\bf Proof]
$\mathrm{(1)}$ 
In \cite[2.21]{Moc11}, 
we prove that if $\cX$ is an adroit system, 
then the inclusion functor $\cF\ltimes \cE_1\rinc \cF\ltimes \cE_2$ 
satisfies the resolution conditions. 
Notice that if $\cX=(\cE_1,\cE_2,\cF)$ is an adroit 
system, 
then $(\cE_1,\cE_2,\cF^w)$ 
is also an adroit system by \ref{prop:strict exact categories} 
and \ref{rem:adroit system}. 
Therefore we learn that the inclusion functor 
$\cF^w\ltimes \cE_1\rinc \cF^w\ltimes \cE_2$ 
also satisfies the resolution conditions. 
The last assertion follows from 
the derived resolution theorem in \ref{thm:resol thms}. 

\sn
$\mathrm{(2)}$ 
First we prove that 
the exact functor 
$\Homo_0:(\cF\ltimes \cE_2,tv)\to (\cF,v)$ 
is a categorical homotopy equivalence. 
This functor has a section $s:\cF \to \cF\ltimes \cE_2$, $x \mapsto [0\to x]$. 
Moreover for any $x=[x_1\onto{d^x} x_0]$ in $\cF\ltimes \cE$, 
the canonical quotient morphism $x_1 \rdef \Homo_0(x)$ 
induces the relative natural equivalence $\id_{\cF\ltimes\cE_2} \to s\Homo_0$. 
Hence $\Homo_0:(\cF\ltimes \cE_2,tv) \to (\cF,v)$ 
is a categorical homotopy equivalence and it is a derived equivalence 
by \ref{cor:Ch_b is 2-functor} $\mathrm{(2)}$. 
It turns out that 
the composition $(\cF\ltimes\cE_1,tw)\to (\cF\ltimes \cE_2,tw) 
\onto{\Homo_0} (\cF,v)$ is also a derived equivalence by $\mathrm{(1)}$. 

\sn
$\mathrm{(3)}$ 
Consider the commutative diagram below:
$${\footnotesize{\xymatrix{
(\cE_1,i_{\cE_1}) \ar[r] \ar[d] &
(\cF\ltimes\cE_1,i_{\cF\ltimes\cE_1}) \ar[r] \ar[d] & 
(\cF,i_{\cF}) \ar@{=}[d]\\
(\cE_1,i_{\cE_2}) \ar[r]  &
(\cF\ltimes\cE_1,i_{\cF\ltimes\cE_2}) \ar[r] & 
(\cF,i_{\cF}). 
}}}$$
Here the vertical lines are induced from the inclusion functors 
and they are derived equivalences by $\mathrm{(1)}$. 
The bottom horizontal 
line is a right quasi-split exact sequence by \cite[2.19]{Moc11}. 
Hence we obtain the result. 
\end{proof}

\begin{lem}
\label{lem:solid denpan}
Let $\cF\rinc \cE$ be strict exact subcategories of $\cA$ and $v$ 
a class of morphisms in $\cF$. 
If $(\cF,v)$ is solid, then $(\cF\ltimes\cE,tv)$ is also solid.
\end{lem}

\begin{proof}[\bf Proof] 
We consider an object $x$ in $\cF\ltimes \cE$ 
to be a complex $[x_1 \onto{d^x} x_0]$ in $\Ch_b(\cE)$. 
For any morphism $f:x \to y$ in $tv(\cF\ltimes \cE)$, 
$\Homo_0(f):\Homo_0(x) \to \Homo_0(y)$ is in $v$.  
Therefore $\Cone \Homo_0(f)$ is in $\Acy_{b}^{qv}(\cF)$ by assumption. 
Notice that there exists the admissible exact sequence 
$$\Cone\Cone(f_1,f_1) \rinf \Cone f \rdef \Ch_b(s)(\Cone\Homo_0(f))$$ 
in $\Ch_b(\cF\ltimes \cE)$ where $s:\cF \to \cF\ltimes \cE$ is 
the exact functor defined by sending an object $x$ in $\cF$ 
to an object $[x \onto{\id_x} x]$ in $\cF\ltimes \cE$. 
Namely we consider the following commutative diagram. 
$${\footnotesize{
\xymatrix{
& x_1 \ar[rr]^{\id_{x_1}} \ar[ld]_{\id_{x_1}} \ar@{-}[d] 
& & x_1 \ar@{-}[d] \ar[ld]_{d^x} \ar[rr] 
& & 0 \ar[ld] \ar[dd]\\
x_1 \ar[rr]^{\ \ \ \ \ \ \ \ d^x} \ar[dd]_{f_1} & 
\ar[d] & x_0 \ar[dd]^{\!\!\! f_0} \ar[rr] & 
\ar[d] & 
\Homo_0x \ar[dd]\\
 & y_1 \ar[ld]^{\!\!\!\! \id_{y_1}} 
\ar@{-}[r] & 
\ar[r] & 
y_1 \ar[ld]^{d^y} \ar@{-}[r] & \ar[r] & 0 \ar[ld]\\
y_1 \ar[rr]_{d^y} & & y_0 \ar[rr] & & \Homo_0y.
}}}$$
Then $\Cone\Cone(f_1,f_1)$ 
and $\Ch_b(s)(\Cone\Homo_0(f))$ are in $\Acy_{b}^{qtv}(\cF\ltimes\cE)$. 
Since $\Acy_{b}^{qtv}(\cF\ltimes\cE)$ is closed under extensions, 
$\Cone f$ is also in $\Acy_{b}^{qtv}(\cF\ltimes \cE)$. 
This means that $(\cF\ltimes \cE,tv)$ satisfies the solid axiom. 
\end{proof}

\begin{cor}
\label{cor:solid denpan}
Let $(\cE_1,\cE_2,\cF)$ 
is an adroit system of $\cA$ and $v$ a class of 
morphisms in $\cF$. 
If $(\cF,v)$ is solid, then $(\cF\ltimes\cE_1,tv)$ is also solid.
\end{cor}

\begin{proof}[\bf Proof]
By \ref{thm:adroit system} $\mathrm{(1)}$ and 
\ref{lem:solid denpan}, 
the inclusion functor 
${(\cF\ltimes \cE_1)}^{tv} \rinc {(\cF\ltimes \cE_2)}^{tv}$ 
satisfies the resolution condition and 
$(\cF\ltimes\cE_2,tv)$ is solid. 
Hence $(\cF\ltimes\cE_1,tv)$ is also solid 
by \ref{cor:res cond and consisitent}. 
\end{proof}

\sn
Next theorem is an abstraction of Corollary~5.14 in \cite{Moc11}.

\begin{thm}
\label{thm:flag} 
Let $2\leq n$ be a positive integer and 
we put $S=(n]$. 
Assume that 
$w$ is compatible with $\fF$.\\
$\mathrm{(1)}$ 
Assume that 
for any integer $0\leq k\leq n-1$, 
there exists a strict exact subcategory $x_k$ of $\Cub^{(n-k-1]}\cA$ 
such that 
a triple 
$(\ltimes\fF|_{(n-k-1]}^{[n-k+1,n]},x_k,\ltimes\fF|_{(n-k-1]}^{[n-k,n]})$ 
is an adroit {\rm (}resp. a strongly adroit{\rm )} 
system of $\Cub^{(n-k-1]}\cA$ 
{\rm (}resp. and $(\cF_S,w)$ is a Waldhausen exact category{\rm )}. 
Then 
the relative exact functor 
$\Homo_{(n-k]}^{[n-k+1,n],\{n-k\}}:
(\ltimes\fF|_{(n-k]}^{[n-k+1,n]},tw)\to 
(\ltimes\fF|_{(n-k-1]}^{[n-k,n]},tw)$ 
is a derived equivalence 
{\rm (}resp. $K^W$-equivalence{\rm )}. 
In particular 
the relative exact functor $\Homo_0^S:(\ltimes \fF,tw) \to (\cF_S,w)$ is 
a derived equivalence {\rm (}resp. $K^W$-equivalence{\rm )}.\\
$\mathrm{(2)}$ 
Assume that for any integer $1\leq k\leq n-1$, 
there exists strict exact subcategory $y_k$ of $\Cub^{(k]}\cA$ 
such that $(\ltimes\fF|_{(k]}^{\emptyset},y_k,\ltimes\fF|_{(k]}^{\{k+1\}})$ 
is an adroit system. 
Then\\
$\mathrm{(i)}$ 
There is a derived right flag of $\ltimes \fF$, 
$$0\to \ltimes \fF|_{(1]}^{\emptyset} \onto{\ext^{\emptyset}_{(1],\{2\}}} 
\ltimes \fF|_{(2]}^{\emptyset}\onto{\ext^{\emptyset}_{(2],\{3\}}} 
\cdots 
\onto{\ext^{\emptyset}_{(n-1],\{n\}}} 
\ltimes \fF|_{(n]}^{\emptyset}=\ltimes \fF.$$
$\mathrm{(ii)}$ 
Moreover if $(\cF_S,w)$ is solid, 
then $(\ltimes \fF,tw)$ is also solid.\\
$\mathrm{(3)}$ 
Let $\calR$ be a reasonable subcategory of $\RelEx_{\strict}$ 
such that for any disjoint pair of 
subsets $U$ and $V$ of $(n]$, 
$(\ltimes\fF|_U^V,tw)$ is in $\calR$ and  
$F$ a localizing theory on $\calR$ {\rm (}resp. $F=K^W${\rm )}. 
Assume that for any integer $1\leq k\leq n-1$, 
there exists strict exact subcategory $z_k$ of $\Cub^{(k]}\cA$ 
such that $(\ltimes\fF|_{(k]}^{\emptyset},z_k,\ltimes\fF|_{(k]}^{\{k+1\}})$ 
is an adroit system 
{\rm (}resp. and $(\cF_S,w)$ is a Waldhausen exact category{\rm )}. 
Then\\
$\mathrm{(i)}$ 
The exact functors 
$$\displaystyle{\lambda_n:\ltimes \fF \to \prod_{T\in\cP(S)}\cF_T}, 
\ \ \ x \mapsto (\Homo_0^T(x)_{\emptyset})_T \ \ \ \text{and}$$ 
$$\displaystyle{\lambda_n':\ltimes \fF^{tw} \to \cF_S^{w}\times 
\prod_{T\in\cP(S)\ssm\{S\}}\cF_T }, 
\ \ \ x \mapsto (\Homo_0^T(x)_{\emptyset})_T$$ 
are $F$-equivalences.\\
$\mathrm{(ii)}$ 
Moreover 
if $(\cF_S,w)$ satisfies the $F$-fibration axiom, 
then $(\ltimes\fF,tw)$ also satisfies the $F$-fibration axiom. 
\end{thm}

\begin{proof}[\bf Proof]
$\mathrm{(1)}$ 
It is just a consequence of \ref{thm:adroit system} 
(resp. \cite[2.21]{Moc11}) $\mathrm{(2)}$. 

\sn
$\mathrm{(2)}$ 
We proceed by induction on $n$. 
$\mathrm{(i)}$ follows from \ref{thm:adroit system} $\mathrm{(3)}$ 
and $\mathrm{(ii)}$ is a consequence of \ref{cor:solid denpan}.

\sn
$\mathrm{(3)}$ 
$\mathrm{(i)}$ 
We proceed by induction on $n$. 
For any $1\leq k\leq n-1$, 
we define $$\underline{\lambda_{k}}:\ltimes\fF|_{(k]}^{\{k+1\}}\to 
\prod_{T\in \cP((k])}\cF_{T\coprod\{k+1\}}$$ 
to be an exact functor by sending 
an object $x$ in $\ltimes\fF|_{(k]}^{\{k+1\}}$ to $(\Homo_0^{T\coprod\{k+1\}}(x))_T$. 
We consider the following commutative diagram.
$${\footnotesize{\xymatrix{
F(\ltimes \fF|^{\emptyset}_{(k]}) 
\ar[r]^{\!\!\!\!\!\!\!\!\!\!\!\!\!\!\!\!\!\!\!\!\textbf{II}} \ar[d] & 
F(\ltimes \fF|^{\{k-1\}}_{(k-1]})\oplus F(\ltimes \fF|^{\emptyset}_{(k-1]}) 
\ar[r]_{\ \ \ \ \ \ \ \ \ \textbf{III}} \ar[d] & 
\bigoplus_{T\in\cP(S)}F(\cF_T)\\
F((\ltimes \fF|^{\{k\}}_{(k-1]})\ltimes z_k) 
\ar[r]_{\textbf{I}} & 
F(\ltimes \fF|^{\{k\}}_{(k-1]})\oplus F(z_k) & 
}}}$$
where the vertical lines are induced from the inclusion functors 
and the maps \textbf{I}, \textbf{II} and \textbf{III} are 
$\begin{pmatrix}F(\Homo_{(k]}^{\emptyset,\{k\}})\\ 
F(\res^{\emptyset,1}_{(k]})\end{pmatrix}$, 
$\begin{pmatrix}F(\Homo_{(k]}^{\emptyset,\{k\}})\\ 
F(\res^{\emptyset,1}_{(k]})\end{pmatrix}$
and 
$\begin{pmatrix}F(\underline{\lambda_{k-1}}) & 
F(\lambda_{k-1})\end{pmatrix}$
respectively. 
Notice that the composition the map \textbf{II} with the map \textbf{III} is just 
the map $F(\lambda_n)$. 
Then since the sequence 
$$(z_k,i_{z_k})\to ((\ltimes \fF|^{\{k\}}_{(k-1]})\ltimes z_k,i_{(\ltimes 
\fF|^{\{k\}}_{(k-1]})\ltimes z_k}) 
\onto{\Homo_0} 
(\ltimes \fF|^{\{k\}}_{(k-1]},i_{\ltimes \fF|^{\{k\}}_{(k-1]}})$$
is a right quasi-split exact sequence by \cite[2.19]{Moc11}, 
the map \textbf{I} is a homotopy equivalence of spectra 
by \ref{prop:fund property of loc theory} $\mathrm{(3)}$. 
Since the vertical lines are homotopy equivalences of spectra by 
\ref{thm:adroit system} $\mathrm{(1)}$, 
it turns out that the map \textbf{II} is also. 
Hence we obtain the result for $n=2$. 
By inductive hypothesis, the map \textbf{III} is also a homotopy equivalence of spectra 
and we obtain the result. 

\sn
For the second assertion, we just apply 
$\fF'=\{\cF_T\}_{T\in\cP(S)\ssm\{S\}}\coprod\{\cF_S^w\}$ 
to the first sentense. 
Notice that by virtue of \ref{cor:denpan} and \ref{rem:adroit system}, 
$\fF'$ satisfies the assumption of $\mathrm{(3)}$ $\mathrm{(i)}$.

\sn
$\mathrm{(ii)}$ 
By the commutative diagram below 
$${\footnotesize{
\xymatrix{
F({(\ltimes \fF)}^{tw}) \ar[r] \ar[d]_{\wr} & 
F(\ltimes \fF) \ar[r] \ar[d]_{\wr} & 
F(\ltimes \fF,tw) \ar[d]^{\wr}\\
\underset{T\in\cP(S)\ssm\{S\}}{\bigoplus} 
F(\cF_T) \oplus F(\cF_S^w) \ar[r] & 
\underset{T\in\cP(S)}{\bigoplus} F(\cF_T) \ar[r] &
F(\cF_S,tw) ,
}}}$$
it turns out that $(\ltimes \fF,tw)$ 
also satisfies the $F$-fibration axiom. 
\end{proof}

\begin{rem}
\label{rem:cubeisnotfibrational}
For an abelian category $\cA$, 
we put $\bA=(\cA\ltimes\cA,\qis)$. 
Then $\bA$ satisfies $K^W$-fibration axiom 
by \ref{thm:flag} $\mathrm{(3)}$. 
But $\bA$ does not satisfy the factorization axiom. 
For assume that for an non acyclic complex $x$ in $\cA\ltimes\cA$, 
if the morphism $x \to 0$ admits a factorization 
$x \overset{i}{\rinf} y \onto{u} 0$ with $u\in\qis$ and then 
we need to have 
$0\neq \Homo_0(x) \overset{\Homo_0(i)}{\rinf} \Homo_0(y)=0$. 
It is a contradiction. 
\end{rem}

\section{Koszul cubes}
\label{sec:Kos cube}

In this section, we prove that the category of 
Koszul cubes with 
the class of total quasi-isomorphisms 
is both $\bbK$- and $K^W$-excellent 
in \ref{cor:Kos is devices} and 
it is derived equivalent to 
the complicial Waldhausen category of perfect complexes in 
\ref{para:proof of WGP}. 
In this section, 
fix a non-empty finite set $S$ and 
we denote the category of 
finitely generated $A$-modules 
by $\cM_A$. 
We start by reviewing the notion $A$-sequences. 

\sn
Let $\{f_s\}_{s\in S}$ be a family of elements in $A$.
We say that the sequence $\{f_s\}_{s\in S}$ is an {\bf $A$-sequence} if 
$\{f_s\}_{s\in S}$ forms an $A$-regular sequences in any order. 
Fix an $A$-sequence $\ff_S=\{f_s\}_{s\in S}$. 
For any subset $T$, we denote the family 
$\{f_t\}_{t\in T}$ by $\ff_T$.

\begin{df}[\bf Koszul cube]
\label{df:Koszul cube df}
(\cf \cite[4.8]{Moc11})
A {\bf Koszul cube} $x$ associated with 
an $A$-sequence 
$\ff_S=\{f_s\}_{s\in S}$ 
is an $S$-cube in $\cP_A$ the category of 
finitely generated projective $A$-modules such that 
for each subset $T$ of $S$ and $k$ in $T$, 
$d^k_T$ is an injection and $f_k^{m_k}\coker d^k_T=0$ for some $m_k$. 
We denote the full subcategory of $\Cub^S\cP_A$ 
consisting of 
those Koszul cubes 
associated with $\ff_S$ by $\Kos_A^{\ff_S}$. 
\end{df}

\sn
Recall the notation of $\cM^I_A(q)$ from Conventions $\mathrm{(4)}$ $\mathrm{(iii)}$. 
The category of Koszul cubes is described in terms of 
multi-semi direct products of 
Quillen exact categories $\cM_A^{\ff_T}(\# T)$ as 
in \ref{Thm:Kosasmultisemidirect}.

\begin{thm}
\label{Thm:Kosasmultisemidirect} 
{\rm (\cf \cite[4.20]{Moc11}).}  
We have the equality 
$$\Kos_A^{\ff_S}=\underset{T\in \cP(S)}{\ltimes} \cM_{A}^{\ff_T}(\# T).$$
\end{thm}

\sn
The description of the category of 
Koszul cubes above gives a motivation to define the 
following categories. 
In the rest of this section, 
fix a disjoint decomposition $S=U\coprod V$. 

\begin{para}
\label{para:genkoscube}
For any non-negative integer $p$, we define the category 
$\cM_A(\ff_U;\ff_V)(p)$  
which is a full subcategory of $\Cub^V\cM_A$ by 
$$\cM_{A}(\ff_U;\ff_V)(p)=
\underset{T\in\cP(V)}{\ltimes}\cM_{A}^{\ff_{T\coprod U}}(p+\# T).$$
Then  
$\cM_A(\ff_U;\ff_V)(p)$ is closed under extensions in 
$\Cub^V\cM_A$. 
In particular it becomes 
a Quillen exact category in the natural way. 
We write $\tq(\cM_A(\ff_U;\ff_V)(p))$ or shortly $\tq$ for 
the class of total weak equivalences 
in $\cM_A(\ff_U;\ff_V)(p)$ associated with 
the class of all isomorphisms in $\cM_A^{\ff_S}(p+\# S)$. 
A morphism in $\tq$ is said to be a {\bf total quasi-isomorphism}. 
Note that we have the equalities 
$$\cM_A(\ff_{\emptyset};\ff_S)(0)=\Kos_A^{\ff_S}\ \ \ \text{and}$$
$$\cM_A(\ff_S;\ff_{\emptyset})(p)=\cM_A^{\ff_S}(p).$$
\end{para}

\sn
In the rest of this section, 
let $p\geq \# U$ be an integer and 
we put $\fF=\{\cM_A^{\ff_{U\coprod T}}(p+\# T) \}_{T\in\cP(V)}$. 
Recall the definition of the resolution conditions 
from \ref{df:resol cond}. 
The following proposition is essentially proven in 
\cite[5.9, 5.11]{Moc11} 

\begin{prop}
\label{prop:tworesol}
The inclusionf functors 
$$i:\cM_A(\ff_U;\ff_V)(p)\rinc  
\cM_A(\ff_U;\ff_V)(p+1)\ \ \ \text{and}$$
$$i':{\cM_A(\ff_U;\ff_V)(p)}^{\tq}\rinc  
{\cM_A(\ff_U;\ff_V)(p+1)}^{\tq}$$ satisfy the resolution conditions. 
\qed
\end{prop}

\sn
Recall the definition of the (strongly) adroit systems 
from \ref{df:good triple}. 

\begin{thm}
\label{thm:Koszul resol thm}
{\rm (\cf \cite[5.13]{Moc11}).}
Assume that $V$ is a non-empty set. 
Then the  triple
$$(\cM_{A}(\ff_U;\ff_{V\ssm\{v\}})(p),
\cM_{A}(\ff_U;\ff_{V\ssm\{v\}})(p+1),
\cM_{A}(\ff_{U\coprod\{v\}};\ff_{V\ssm\{v\}})(p+1))$$
is a strongly adroit system in $\Cub^{V\ssm\{v\}}\cM_A$ for any $v\in V$. 
\end{thm}

\sn
In the rest of this section, for simplicity 
we put $\calH=\cM_A(\ff_U;\ff_V)(p)$. 
Recall the definitions of the flag from \ref{df:flag} 
and of the excellent relative exact categories from \ref{lemdf:excellent}. 

\begin{cor}
\label{cor:Kos is devices}
$\mathrm{(1)}$ 
$(\calH,\tq)$ is a 
very strict solid Waldhausen exact category and satisfies 
the $K^W$-fibration axiom. 
In particular it is both $K^W$- and $\bbK$-excellent.\\
$\mathrm{(2)}$  
The exact functor 
$\Homo_0\Tot:(\calH,\tq) \to (\cM_A^{\ff_S}(p+\#V),i)$ 
is a derived equivalence. 
In particular, the exact functor 
$\Homo^S_0:(\Kos_A^{\ff_S},\tq) \to (\cM_A^{\ff_S}(\# S),i)$ is 
a derived equivalence.\\
$\mathrm{(3)}$ 
Let $V=\{v_1,v_2,\cdots,v_r\}$ and we put 
$V_k=\{v_i;1\leq i\leq k \}$ for any $1\leq k\leq r$ and $V_0=\emptyset$ 
and $e_k:=\ext_{}^{\emptyset,\{k+1\}}:
\cM_A(\ff_U;\ff_{V_k})(p)=\ltimes \fF|^{\emptyset}_{(k]}\to \cM_A(\ff_U;\ff_{V_{k+1}})(p)$. Then the sequence 
$$0 \to\cM_A(\ff_U;\ff_{V_0})(p) \onto{e_0}
\cM_A(\ff_U;\ff_{V_1})(p) \onto{e_1}\cdots
\onto{e_{r-1}}\calH
$$
is a derived right flag.\\
$\mathrm{(4)}$ 
The inclusion functor 
${\calH}^{tq}\rinc\calH$ and 
the identity functor of $\calH$ induce a derived 
right quasi-split exact sequence 
$$({\calH}^{\tq},i_{{\calH}^{\tq}}) \to 
(\calH,i_{\calH}) 
\to (\calH,\tq).$$
\end{cor}

\begin{proof}[\bf Proof]
First notice 
that the class of all isomorphisms in $\cM_A^{\ff_S}(p+\# S)$ 
is compatible with $\fF$ and $(\calH,\tq)$ 
is a Waldhausen exact category which satisfies the extensional axiom 
by \ref{lem:denpan} and \ref{cor:denpan}. 
Assertions $\mathrm{(2)}$ and $\mathrm{(3)}$ 
and the fact that $(\calH,\tq)$ 
satisfies the both the solid and the $K^W$-fibration axioms 
follow from \ref{thm:flag} and \ref{thm:Koszul resol thm}. 
The fact that $(\calH,\tq)$ is 
very strict is just a consequence of $\mathrm{(4)}$. 
The last assertion of $\mathrm{(1)}$ follows from \ref{prop:excellent}. 

\sn
Proof of assertion $\mathrm{(4)}$: 
By virtue of $\mathrm{(2)}$, 
we only need to check that 
the sequence 
$$(\calH^{\tq},i_{\calH^{\tq}})\to(\calH,i_{\calH})\onto{\Homo_0^V} 
(\cM_A^{\ff_S}(p+\# V),i_{\cM_A^{\ff_S}(p+\# V)})$$
is a derived right quasi-split exact sequence. 
For simplicily we put $\calH':=\cM_A(\ff_U;\ff_V)(p+\#V)$ and 
let $I:\calH \rinc \calH'$ and $J:\calH^{\tq} \rinc {\calH'}^{\tq}$ 
be the inclusion functors. 
We define $s:\cM_A^{\ff_S}(p+\# V) \to \calH'$ 
to be an exact functor 
by sending an object $x$ to $s(x)$ where 
$s(x)_T$ is $x$ if $T=\emptyset$ and $0$ if $T\neq \emptyset$. 
There is a natural transformation $C:I \to s\Homo_0^V$ such that 
for any object $x$ in $\calH$, 
$${(Cx)}_{\emptyset}:x_{\emptyset} \to 
s\Homo_0^V(x)=\coker \left (\bigoplus_{v\in V} x_{\{v\}} \to x_{\emptyset} \right )$$
is the natural quotient morphism. 
We define the exact functor $r:\calH \to {\calH'}^{\tq}$ 
by $r:=\ker(I \onto{C} s\Homo^V_0)$. 
We illustrate the situation with the commutative diagram below. 
$${\footnotesize{\xymatrix{
\calH^{\tq} \ar[r]^i \ar[d]_J & 
\calH \ar[r]^{\!\!\!\!\!\!\!\!\Homo_0^V} \ar[d]^{I} \ar[ld]_r & 
\cM_A^{\ff_S}(p+\# V) \ar[ld]^s\\
{\calH'}^{\tq} \ar[r]_{i'} &
\calH' & .
}}}$$
By definition, there exists an admissible exact sequence 
of exact functors from $\calH$ to ${\calH'}^{\tq}$ 
$$i'r \overset{A}{\rinf} I \overset{C}{\rdef} s\Homo_0.$$
Since $I$ and $J$ are derived equivalences by \ref{prop:tworesol}, 
the pair $(r,s)$ yields a right quasi-splitting of a sequence 
$$\calD_b(\calH^{\tq}) \onto{\calD_b(i)} \calD_b(\calH)\onto{\calD_b(\Homo_0)}
\calD_b(\cM_A^{\ff_S}(p+\# V)).$$
Hence we obtain the result.
\end{proof}

\sn
Recall the definition of $\Perf_{\Spec A}^{V(\ff_S)}$ from the introduction.

\begin{para}[\bf Proof of Theorem~\ref{thm:weak geom present}]
\label{para:proof of WGP} 
Let $j':\cM_A^{\ff_S}(\# S) \to \Perf_{\Spec A}^{V(\ff_S)}$ denote 
the composition of $j_{\cM_A^{\ff_S}(\# S)}:\cM_A^{\ff_S}(\# S)\to \Ch_b(\cM_A^{\ff_S}(\# S))$ 
and the inclusion functor $\Ch_b(\cM_A^{\ff_S}(\# S)) \rinc \Perf_{\Spec A}^{V(\ff_S)}$. 
There exists a canonical relative natural equivalence $\Tot \to j'\Homo_0\Tot$ 
induced by the quotient map 
$(\Tot x)_0 \rdef \Homo_0\Tot x$ for any $x$ in $\Kos_A^{\ff_S}$. 
Hence $\calD_b\Tot=\calD_bj'\calD_b\Homo_0\Tot$ 
by \ref{cor:Ch_b is 2-functor} $\mathrm{(2)}$ and 
the functor $\calD_b\Homo_0\Tot$ and $\calD_bj'$ are 
equivalences of triangulated categories by \ref{cor:Kos is devices} $\mathrm{(2)}$, 
and 
\ref{cor:comp of derived cat} $\mathrm{(2)}$ 
and 
\cite[3.3]{HM10} respectively. 
Therefore we obtain the result. 
\qed
\end{para}


\begin{thebibliography}{1234567}
\bibitem[Bal07]{Bal07} 
P.~Balmer, 
\emph{{S}upports and filtrations in algebraic geometry 
and modular representation theory}, 
American Journal of Mathematics \textbf{129} (2007), 
pp. 1227-1250. 

\bibitem[BK12]{BK12}
C.~Barwick and D.~M.~Kan, 
\emph{{R}elative categories: another model for the homotopy theory of homotopy theories}, 
Indag. Math. \textbf{23} (2012), p.42-68.

\bibitem[Bas68]{Bas68}
H.~Bass, 
\emph{{A}lgebraic $K$-theory}, 
W.~A.~Benjamin, Inc., New York, Amsterdam (1968).


\bibitem[BGT13]{BGT10}
A.~J.~Blumberg, D.~Gepner and G.~Tabuada, 
\emph{{A} universal characterization of higher algebraic {$K$}-theory}, 
preprint, available at arXiv:1001.2282 (2010).




\bibitem[BKS07]{BKS07}
A.~B.~Buan, H.~Krause and 
{\O}.~Solberg, 
\emph{Support varieties: An ideal approach}, 
Homology, Homotopy and Applications. 
vol. \textbf{9} (2007), 
p.45-74.


\bibitem[Car80]{Car80}
D.~W.~Carter, 
\emph{{L}ocalization in lower algebraic {$K$}-theory}, 
Comm. Algebra \textbf{8} (1980), 
p.603-622.


\bibitem[Cis02]{Cis02}
D.~-C.~Cisinski, 
\emph{{T}h{\'e}or{\`e}ms de cofinalit{\'e} en {$K$}-th{\'e}orie (d'apr{\`e}s Thomason)}, 
avilable at 
{\tt{http:www.math.univ-toulouse.fr/\~{}dcisinsk/publications.html}}

\bibitem[CT11]{CT11}
D.-C.~Cisinski and G.~Tabuada,
\emph{{N}on-connective {$K$}-theory via universal invariants}, 
Compositio Mathematica \textbf{147} (2011), 
p.~1281-1320.

\bibitem[CHSW08]{CHSW08}
G.~Corti\~nas, C.~Haesemeyer, 
M.~Schlichting and C.~A.~Weibel, 
\emph{Cyclic homology, cdh-cohomology and negative K-theory}, 
Annals of Mathematics \textbf{167} (2008), 
p.549-573.


\bibitem[FL12]{FL12}
T.~M.~Fiore and W.~L{\"u}ck, 
\emph{Waldhausen additivity: classical and quasicategorical}, 
arXiv:1207.6613 (2012).


\bibitem[Ger73]{Ger73} 
S.~Gersten, 
\emph{{S}ome exact sequences in the higher {$K$}-theory of rings}, 
In Higher K-theories, Springer Lect. Notes Math. \textbf{341} (1973), 
p.211-243.

\bibitem[Gra95]{Gra95}
D.~R.~Grayson, 
\emph{{W}eight filtrations via commuting automorphisms}, 
$K$-theory \textbf{9} (1995), 
p.139-172.


\bibitem[GSVW92]{GSVW92}
T.~Gunnarsson, R.~Schw\"anzl, R.~M.~Vogt and F.~Waldhausen, 
\emph{{A}n un delooped version of algebraic {$K$}-theory}, 
Journal of Pure and Applied Algebra \textbf{79} (1992), 
p.255-270.


\bibitem[HM10]{HM10}
T.~Hiranouchi and S.~Mochizuki, 
\emph{{P}ure weight perfect modules over divisorial schemes}, 
in {D}eformation Spaces: {P}erspectives on {A}lgebro-geometric {M}oduli 
(2010), 
p.75-89.




\bibitem[Kar70]{Kar70}
M.~Karoubi, 
\emph{{F}oncteurs d\'eriv\'es et $K$-th\'eorie}, 
in S\'eminaire Heidelberg-Saabr\"ucken-Strasbourg sur 
la $K$-th\'eorie (1967/68), 
Lecture Notes in Mathematics, vol.\textbf{136}, 
Springer, Berlin (1970), 
p.107-186.


\bibitem[Kel90]{Kel90}
B.~Keller, 
\emph{{C}hain complexes and stable categories}, 
manus. math. \textbf{67} (1990), p.379-417.



\bibitem[Kel96]{Kel96}
B.~Keller, 
\emph{{D}erived categories and their uses}, 
Handbook of algebra, Vol.\
  1, North-Holland, Amsterdam, (1996), 
p.~671-701.

\bibitem[MVW06]{MVW06}
C.~Mazza, V.~Voevodsky and C.~A.~Weibel, 
\emph{Lecture notes on motivic cohomology},
Clay Mathematics Monograph, Volume \textbf{2} (2006).

\bibitem[Moc10]{Uni}
S.~Mochizuki, 
\emph{{U}niversal property of the category of 
bounded chain complexes 
on exact categories}, 
avilable at 
{\tt{http://www.math.uiuc.edu/K-theory/0975/}} (2010).

\bibitem[Moc13]{Moc11}
S.~Mochizuki, 
\emph{{H}igher $K$-theory of Koszul cubes}, 
avilable at arXiv:1303.1239, (2013). 

\bibitem[Nee01]{Nee01}
A.~Neeman, 
\emph{{T}riangulated categories}, 
Annals of Mathematics Studies \textbf{148}, (2001)

\bibitem[Ped84]{Ped84}
E.~K.~Pedersen, 
\emph{{O}n the $K_{-i}$-functors}, 
J. Algebra \textbf{90} (1984), 
p.461-475.

\bibitem[PW89]{PW89}
E.~K.~Pedersen and C.~A.~Weibel, 
\emph{$K$-theory homology of spaces}, 
In Algebraic Topology, 
Springer Lect. Notes Math. \textbf{1370} (1989), 
p.346-361.

\bibitem[Qui73]{Qui73} 
D.~Quillen, 
\emph{{H}igher algebraic {$K$}-theory {I}}, 
In Higher $K$-theories, 
Springer Lect. Notes Math. \textbf{341} (1973), 
p.85-147.

\bibitem[Sch04]{Sch04}
M.~Schlichting, 
\emph{{D}elooping the {$K$}-theory of exact categories}, 
Topology \textbf{43} (2004), 
p.1089-1103.


\bibitem[Sch06]{Sch06}
M.~Schlichting, 
\emph{{N}egative {$K$}-theory of derived categories}, 
Math. Z. \textbf{253} (2006), 
p.97-134.


\bibitem[Sch11]{Sch11}
M.~Schlichting, 
\emph{{H}igher algebraic {$K$}-theory (after Quillen, Thomason and others)}, 
Topics in Algebraic and Topological $K$-theory, 
Springer Lecture Notes in Math. \textbf{2008} (2011), 
p.167-242. 

\bibitem[Tab08]{Tab08}
G. Tabuada, 
\emph{{H}igher {$K$}-theory via universal invariants}, 
Duke Math. J. \textbf{145} (2008), 
p.121-206.

\bibitem[TT90]{TT90}
R.~W.~Thomason and T.~Trobaugh, 
\emph{{H}igher algebraic {$K$}-theory of schemes 
and of derived categories}, 
The Grothendieck Festschrift, Vol.\ III, Progr.
  Math., vol.~\textbf{88}, 
Birkh\"auser Boston, Boston, MA, (1990), 
p.247-435.

\bibitem[Tho93]{Tho93}
R.~W.~Thomason, 
\emph{{L}es {$K$}-groups d'un sch\'ema 
\'eclat\'e et une formule d'intersection exc\'edentaire}, 
Invent. Math. \textbf{112} (1993), 
p.195-215.

\bibitem[Wag72]{Wag72}
J.~B.~Wagoner, 
\emph{Delooping classifying spaces in algebraic $K$-theory}, 
Topology \textbf{11} 
(1972), 
p.349-370.


\bibitem[Wal85]{Wal85} 
F.~Waldhausen, 
\emph{{A}lgebraic {$K$}-theory of spaces}, 
In Algebraic and geometric topology, 
Springer Lect. Notes Math. \textbf{1126} (1985), 
p.318-419.

\bibitem[Wal00]{Wal00}
M.~E.~Walker, 
\emph{{A}dams operations for bivariant {$K$}-theory and 
a filtration using projective lines}, 
$K$-theory \textbf{21} (2000), 
p.101-140.


\bibitem[Wei80]{Wei80}
C.~A.~Weibel, 
\emph{{$K$}-theory and analytic isomorphisms}, 
Invent. Math. \textbf{61} (1980), 
p.177-197.


\bibitem[Wei94]{Wei94}
C.~A.~Weibel, 
\emph{{A}n introduction to homological algebra}, 
Cambridge Studies in Advanced Mathemaics \textbf{38} (1994).

\bibitem[Wei13]{Wei12}
C.~A.~Weibel, 
\emph{{T}he K-book:\ An introduction to algebraic $K$-theory}, 
Grad. Texts in Math., AMS (2013). 


\bibitem[Yao92]{Yao92}
D.~Yao, 
\emph{{H}igher algebraic $K$-theory of 
admissible abelian categories and localization theorems}, 
J. Pure Appl. Algebra \textbf{77} (1992), 
p.263-339.

\end{thebibliography}
\end{document}